\documentclass[10pt]{amsart}
\usepackage[table]{xcolor}
\usepackage{graphicx,fullpage,xcolor,comment,bm} 
\usepackage{amsmath,amsfonts}
\usepackage[margin=1in, top = 0.7in]{geometry}
\usepackage[english]{babel}
\usepackage{amsthm}
\usepackage{algorithm}
\usepackage{textcomp}
\usepackage{algpseudocode}
\usepackage{bbm}
\usepackage[shortlabels]{enumitem}
\usepackage{wrapfig}
\usepackage{caption}
\usepackage{comment}
\usepackage{hyperref}
\usepackage[table]{xcolor}

\graphicspath{{images/}}

\newtheorem{theorem}{Theorem}
\newtheorem{lemma}{Lemma}

\newcommand{\E}{\mathbb{E}}

\newtheorem{fact}{Fact}
\newtheorem{definition}{Definition}
\newcommand{\tens}[1]{\bm{\mathcal{#1}}}
\newcommand{\mat}[1]{\bm{#1}}
\def\tA{{\tens{A}}}  
\def\tB{{\tens{B}}}  

\def\tE{{\tens{E}}}

\def\tH{{\tens{H}}}
\def\tI{{\tens{I}}}

\def\tM{{\tens{M}}}

\def\tP{{\tens{P}}}

\def\tU{{\tens{U}}}
\def\tV{{\tens{V}}}
\def\tW{{\tens{W}}}
\def\tX{{\tens{X}}}  
\def\tY{{\tens{Y}}}
\def\tZ{{\tens{Z}}}

\def\vb{{\bm{b}}}

\def\ve{{\bm{e}}}

\def\vh{{\bm{h}}}

\def\vv{{\bm{v}}}

\def\vx{{\bm{x}}}

\def\R{{\mathbb{R}}} 
\def\E{{\mathbb{E}}} 

\def\bcirc{{\mathrm{bcirc}}}
\def\unfold{{\mathrm{unfold}}}
\def\fold{{\mathrm{fold}}}

\def\bcirc{{\mathrm{bcirc}}}

\title{Block Gauss-Seidel methods for t-product tensor regression}

\author{Alejandra Castillo, Jamie Haddock, Iryna Hartsock, Paulina Hoyos, Lara Kassab, Alona Kryshchenko, Kamila Larripa, Deanna Needell, Shambhavi Suryanarayanan, Karamatou Yacoubou Djima}

\begin{document}

\maketitle

\begin{abstract}
    Randomized iterative algorithms, such as the randomized Kaczmarz method and the randomized Gauss-Seidel method, have gained considerable popularity due to their efficacy in solving matrix-vector and matrix-matrix regression problems. Our present work leverages the insights gained from studying such algorithms to develop regression methods for tensors, which are the natural setting for many application problems, e.g., image deblurring. In particular, we extend two variants of the block-randomized Gauss-Seidel method to solve a t-product tensor regression problem.  We additionally develop methods for the special case where the measurement tensor is given in factorized form. We provide theoretical guarantees of the exponential convergence rate of our algorithms, accompanied by illustrative numerical simulations. 
\end{abstract}

\section{Introduction}\label{sec:intro} 

Throughout this paper, we consider the tensor linear regression problem  
\begin{equation}
    \min_{\tX \in \mathbb{R}^{n \times l \times p}} \|\tA \tX - \tB\|_F^2, \label{eq:regression}
\end{equation} 
where $\tA \in \mathbb{R}^{m \times n \times p}$ is the measurement operator or dictionary, $\tB \in \mathbb{R}^{m \times l \times p}$ represents the measurements or data, $\tX \in \mathbb{R}^{n \times l \times p}$ is the signal of interest, and $\tA \tX$ is the t-product~\cite{kilmer2011factorization} between $\tA$ and $\tX$ defined in Section~\ref{sec:notation}.
We also pay special attention to the particular case of the problem where the measurement operator $\tA$ is given in factorized form as $\tA = \tU \tV$, where $\tU \in \R^{m \times m_1 \times p}, \tV \in \R^{m_1 \times n \times p}$, 
\begin{equation}
    \min_{\tX \in \mathbb{R}^{n \times l \times p}} \|\tU \tV \tX - \tB\|_F^2. \label{eq:fact_regression}
\end{equation}

Recently, there has been increased interest in solving problems in the form of~\eqref{eq:regression} and~\eqref{eq:fact_regression}. This is because tensors can be viewed as higher dimensional analogues of matrices, and hence are a natural setting for many problems in the real world, which is inherently multidimensional. For example, MRI and hyperspectral images collect data in three dimensions. Although this data can be analyzed as individual two-dimensional slices, leveraging a tensor structure enables the integration of patterns across the entire dataset. 

The methods we propose in this paper are related to numerous recent developments in the area of randomized iterative methods for tensor linear systems and tensor regression, including the tensor randomized Kaczmarz (TRK) method~\cite{ma2022randomized}, the randomized Kaczmarz methods for t-product tensor linear systems with factorized operators~\cite{castillo2024randomized}, and the tensor randomized regularized Kaczmarz method for convex optimization with linear constraints~\cite{chen2021regularized,du2021randomized}. The above tensor algorithms typically generalize techniques for matrix-vector and matrix-matrix regression problems. For example, TRK is an extension of the popular randomized Kaczmarz (RK) algorithm~\cite{strohmer2009randomized}, while the methods in~\cite{castillo2024randomized} are extensions of iterative methods for solving factorized matrix linear systems \cite{ma2018iterative}. The present paper builds on the work in~\cite{leventhal2010randomized} by extending the randomized Gauss-Seidel (RGS) method to the tensor setting. The RGS method and its variants utilize the columns of the measurement matrix, whereas RK and its variants utilize rows. We construct novel block RGS algorithms for problems~\eqref{eq:regression} and~\eqref{eq:fact_regression}, including pseudoinverse free algorithms. 
We can prove that linear convergence is attained in all these settings and provide numerical experiments illustrating our convergence results.

\subsection{Notation}\label{sec:notation}
We let $[n]$ denote the set $\{1, 2, \cdots, n\}$.  Let $T \subset \mathcal{P}([n])$, where $\mathcal{P}(S)$ denotes the power set of the set $S$.  We define $c_{\min}(T) = \text{argmin}_{i \in [n]} |\{\tau \in T : i \in \tau\}|$ as the minimum number of times the index $i$ appears in the sets $\tau \subset [n]$ in $T$.

We use boldfaced lower-case Latin letters (e.g., $\vx$) to denote vectors, boldfaced upper-case Latin letters (e.g., $\mat A$) to denote matrices, and boldfaced upper-case calligraphic Latin letters (e.g., $\tA$) to denote higher-order tensors. We use non-bold lower-case Latin and Roman letters (e.g., $q$ and $\beta$) to denote scalars.  Throughout, we denote by $\sigma_{\min}{(\mat A)}$ the smallest singular values of the matrix $\mat A$ (that is, the smallest eigenvalue of the matrix $\sqrt{\mat A^\top \mat A}$).  We use $\mat A \otimes \mat B$ to denote the Kronecker product of matrices $\mat A$ and $\mat B$.  

We use ``MATLAB" notation; e.g., $\mat A_{i :}$ is the $i$th row of matrix $\mat A$ and $\tA_{: j :}$ is the $j$th column-slice of tensor $\tA$.  We use $\tA^*$ to denote the \emph{conjugate transpose} of the tensor $\tA \in \mathbb{C}^{m \times n \times p}$, which is obtained by taking the conjugate transpose of each of the frontal slices and then reversing the order of transposed frontal slices 2 through $p$. The \emph{identity tensor} $\tI \in \R^{n\times n \times p}$ is the tensor whose first frontal slice is the $n\times n$ identity matrix, and all other slices are the $n \times n$ zero matrix. Additionally, a tensor $\tB \in \R^{n\times n \times p}$ is the \emph{inverse tensor} of $\tA \in \R^{n\times n \times p}$ if $\tA\tB = \tI = \tB\tA$.

The notation $\|\vv\|$ denotes the Euclidean norm of a vector $\vv$ and $\|\cdot\|_F$ the Frobenius norm of any tensor input. 
Note that \begin{align*}
    \langle \tA, \tB \rangle = \sum_{i,j,k} \tA_{ijk}\tB_{ijk}
\end{align*}
defines an inner product on spaces of real-valued tensors whose dimensions agree.  This inner product induces the Frobenious norm of a tensor $\tA$, i.e., $\langle \tA, \tA \rangle = \|\tA\|_F^2$. 

We now provide background on the tensor-tensor t-product~\cite{kilmer2011factorization}.
\begin{definition}
\label{def:t-product}
The \emph{tensor-tensor t-product} between $\tA \in \R^{m \times n \times p}$ and $\tB \in \R^{n \times l \times p}$ is defined as
\begin{equation*}
\tA \tB = \fold (\bcirc(\tA) \unfold (\tB)) \in \R^{m \times l \times p},
\end{equation*}
where $\bcirc (\tA)$ denotes the block-circulant matrix
\[\bcirc(\tA) = \begin{pmatrix}
\tA_{1} & \tA_{p} & \tA_{p-1} & \dots & \tA_{2} \\
\tA_{2} & \tA_{1} & \tA_{p} & \dots & \tA_{3} \\
\vdots & \vdots & \vdots & \dots & \vdots \\
\tA_{p} & \tA_{p-1} & \tA_{p-2} & \dots & \tA_{1} \\
\end{pmatrix} \in \R^{mp \times np},\]
$\unfold(\tB)$ denotes the operation defined as
\[\unfold(\tB) = \begin{pmatrix}
 \tB_{1}\\
 \tB_{2}\\
 \vdots \\
 \tB_{p}
\end{pmatrix} \in \R^{np \times l},\]
and $\fold(\unfold (\tB))) = \tB$.
\end{definition}

Finally, we note that the definition of the t-product implies that a tensor linear system may be reformulated as an equivalent (matrix) linear system.
\begin{fact}
   The tensor linear system $$\tA \tX = \tB$$ is equivalent to the matrix-matrix linear system $$\bcirc(\tA) \unfold(\tX) = \unfold(\tB);$$ that is, solutions to the tensor linear system, $\tX$, after unfolding, $\unfold(\tX)$, are solutions to the matrix linear system, and vice versa. \label{fact:equivalent_systems}
\end{fact}
As with matrix-vector linear system iterative methods, the range of a given tensor under the tensor t-product is an important concept for defining and understanding tensor regression.
  The \emph{$k$-range space} of tensor $\tA$ is defined as
        \[ \mathcal{R}_k(\tA) = \{\tA \tX \colon \tX \in \R^{m\times k\times p} \}. \]

To describe our algorithms, we will use the notation $\tE_j\in \mathcal{R}^{n \times 1 \times p}$ for a vertical slice tensor with the first frontal slice being a standard basis vector $\ve_j$ with a one in the $jth$ coordinate and zeros elsewhere, and the remaining frontal slices all entirely zero. For $\tau \subset [n]$, we additionally define $\tE_{\tau} \in \mathcal{R}^{n \times |\tau| \times p}$ to be a vertical slice tensor with the first frontal slice being the columns of the $n \times n$ identity matrix indexed by $\tau$, and the rest of the frontal slices all zero. We will often use the fact that, for $\tA \in \R^{m\times n \times p}$,  $\tA\tE_{\tau}  = \tA_{:\tau:}$.

\subsection{Randomized Gauss-Seidel (RGS) method and block extensions}\label{subsec:RGS}

Column-action methods are gaining popularity in the literature because of their computational advantages. In settings where row slices of the measurement tensor are very large and exceed the capacity of active memory, or when the tensor data is inherently organized as column-slice components (e.g., distributed across computational servers or primarily indexed by columns), accessing the tensor through its column slices may be the only practical and dependable option. 

The \emph{Jacobi} and \emph{Gauss-Seidel} methods are classic examples of column-action iterative algorithms used to approximate a solution $\bm{x} \in \R^{n}$ in the system of linear equations $\mat A\bm{x}=\bm{b}$, where $\mat A \in \R^{n \times n}$ and $\bm{b} \in \R^{n}$ are given. Both algorithms start by decomposing the matrix $\mat A$ into the sum of a strictly lower triangular matrix $\mat L$, a diagonal matrix $\mat D$, and a strictly upper triangular matrix $\mat  U$ such that $\mat  A = \mat D + \mat  L + \mat U$. The matrix equation above can then be reformulated as $\mat D \vx + \mat L\vx+ \mat U \vx=\bm{b}$. Based on this decomposition, each method uses a fixed-point equation to construct an iterative algorithm, often expressed entry-wise. In each case, the convergence properties of the method are proved using the spectral radii of the matrices involved in each algorithm update.

The main difference between the two methods is that for each new iterate $\bm{x}^{(k)}$, the Jacobi method updates the approximation of the solution using only the entries of the previous iteration vector $\bm{x}^{(k-1)}$, while the Gauss-Seidel algorithm uses the most recently updated entry values in $\bm{x}^{(k)}$ to obtain the subsequent entries. Since the Gauss-Seidel method generally requires fewer iterations to reach the desired accuracy~\cite{GVL96:Matrix-Computations}, it is therefore favored over the Jacobi method when one desires faster convergence to a solution. However, if parallelization is a priority, the Jacobi method may be more suitable due to its simpler implementation in parallel computing environments. Furthermore, while both methods achieve optimal performance when $\mat A$ is diagonally dominant, the Gauss-Seidel method can still converge in certain scenarios where the Jacobi method fails to do so~\cite{GVL96:Matrix-Computations}.

The \emph{randomized Gauss-Seidel} (RGS) method~\cite{leventhal2010randomized} is a variant of randomized coordinate descent applied to a least-squares objective. The update is defined by minimizing a subset of the residual error with respect to a single coordinate. RGS can also be viewed as a variant of the classical Jacobi method or the classical Gauss-Seidel method applied to the normal equations $\mat{A}^\top \mat{A} \vx = \mat{A}^\top \vb$ in which only a single coordinate is updated. 

The formula for the $k$th iterate of RGS is given as 
\begin{equation}
    \vx^{(k)} = \vx^{(k-1)} - \frac{ \mat  A_{: i_k}^T(\mat  A \vx^{(k-1)} - \vb)}{\| \mat  A_{: i_{k}}\|^2} \ve_{i_k}, \label{eq:GSupdate},
\end{equation} 
where $\mat  A_{: i_k}$ is the $i_k$th column of $\mat A$. The probability of sampling the $j$th column of $\mat A$ is $\|\mat A_{: j}\|^2/\|\mat A\|_F^2$, i.e., proportional to the square of the Euclidean norm of the column. Under certain conditions, the RGS algorithm is proven to exhibit expected linear convergence:
\begin{equation}
    \mathbb{E}\|\vx^{(k)} - \vx^*\|^2 \le \left(1 - \frac{\sigma_{\min}^2(\mat{A})}{\|\mat{A}\|_F^2}\right)^k \|\vx^{(0)} - \vx^*\|^2. \label{eq:RGSrate}
\end{equation}  
Given its straightforward implementation, relevance to certain data structures, and desirable convergence rate, RGS has found applications in subroutines for multigrid methods~\cite{rude1993mathematical,trottenberg2000multigrid}, high-performance computing~\cite{wolfson2017distributed}, and PDEs~\cite{glusa2020scalable,magoules2017asynchronous}, and has also garnered much interest in further theoretical extensions and generalizations. In~\cite{Ma2015convergence}, a unified theoretical framework is developed for the Kaczmarz and Gauss-Seidel algorithms, highlighting their relationships in both under-determined and over-determined settings. Additionally, an extended variant of the randomized Gauss-Seidel (RGS) method is introduced, and, unlike the standard RGS, this variant achieves linear convergence to the least-norm solution in the under-determined case. The paper~\cite{richtarik2016parallel} demonstrates that randomized (block) coordinate descent methods can achieve acceleration through parallelization when applied to the optimization problem of minimizing the sum of a partially separable smooth convex function and a simple separable convex function. Further extensions and surveys can be found in~\cite{frommer2023convergence,HefnyRows15}. In~\cite{du2021randomized}, two pseudoinverse-free versions of randomized block Gauss-Seidel are proposed to solve both consistent and inconsistent linear systems. These algorithms utilize two user-defined discrete or continuous random matrices, one for row sampling and the other for column sampling. In addition to a proof of linear convergence in the mean-square sense, numerical experiments are provided to illustrate the algorithms' performance.

\subsection{Tensor randomized Kaczmarz method}\label{subsec:TRK}

The \emph{tensor randomized Kaczmarz} (TRK) method is introduced in~\cite{ma2022randomized} as a generalization of the randomized Kaczmarz method tailored for tensor linear systems defined using the t-product. Starting from an initial approximation $\tX^{(0)}$ to the solution $\tX^*$ of the tensor linear system $\tA \tX = \tB$, TRK updates approximations of $\tX^*$ via successive iterations with two steps: 1) sampling a row slice of the tensor system defined by $\tA_{i_k::}$ and $\tB_{i_k::}$, and 2) projecting the previous iterate onto the space of solutions to this sampled subsystem. This method is proven to converge at least linearly in expectation to the unique solution of the system, $\tX^*$.

Moreover, the authors noted that while the TRK method can be interpreted as a block Kaczmarz method applied to a matrix-matrix system with a specific block selection, a straightforward application of the block Kaczmarz method yields weaker theoretical guarantees compared to those provided by their analysis of TRK. 

An extension to TRK, the tensor randomized average Kaczmarz (TRAK), is provided in~\cite{Bao2022RandAvg}. The TRAK algorithm finds the least Frobenius-norm solution for consistent tensor linear systems of the form $\tA \tX = \tB$, where $\tA \in \R^{m \times n \times p}$, $\tX \in \R^{n \times l \times p}$, and $\tB \in \R^{m \times l \times p}$. This method has the advantage of being pseudoinverse- and inverse-free and offers a speed-up over the TRK method for solving tensor linear systems $\tA \tX = \tB$. The algorithm converges at least linearly in expectation to the unique least Frobenius-norm solution $\tX^\ast = \tA^\dagger \tB$. 

Further developments for tensor Kaczmarz include the tensor regular sketch-and-project algorithm (TESP)~\cite{tang2022sketch}, which, like the two previously discussed methods, tackles a consistent tensor system. In~\cite{huang2023tensor}, the tensor randomized extended Kaczmarz (TREK) method, an extended variant of TRK, is constructed to solve an inconsistent system and converges at least linearly in expectation to the solution of the least-squares tensor solution. 

\subsection{Factorized Kaczmarz methods}\label{subsec:FKmethods}

In~\cite{ma2018iterative}, factorized Kaczmarz methods are developed to solve a linear system of equations $\mat A \vx = \vb$, where the measurement matrix $\mat A$ is expressed as the product of two matrices $\mat U$ and $\mat V$. Their algorithms, RK-RK and REK-RK, compute the optimal solution for a system defined by a factorized measurement matrix in the consistent and inconsistent regime, respectively. The authors of the present paper recently proposed extensions of the RK-RK and REK-RK algorithms to the tensor setting in~\cite{castillo2024randomized}. Specifically, these extensions produce an approximate solution for a (consistent or inconsistent) tensor system in the form
\begin{equation}
    \tU \tV \tX= \tB,
\end{equation}
where $\tU \in \R^{m \times m_1 \times p}, \tV \in \R^{m_1 \times n \times p}$, $\tX \in \R^{n \times l \times p}$ and $\tY \in \R^{m \times l \times p},$ by iteratively computing approximate solutions to the systems 
\begin{align}
     \tU \tZ &= \tB, \text{ and }\label{eq:outer_system}\\
     \tV \tX &= \tZ.\label{eq:inner_system}
\end{align}
The essential idea is to solve the systems via interlaced Kaczmarz steps such that, in each iteration, the algorithm finds an update for $\tZ$ that is used to compute the next update for $\tX$ in the same iteration. In both the matrix~\cite{ma2018iterative} and tensor~\cite{castillo2024randomized} cases, theoretical guarantees of linear convergence are provided along with a suite of numerical experiments that confirm the convergence results.

\subsection{Contributions and Organization}

In~\cite{ma2022randomized}, it was shown that TRK applied to a tensor system exhibited an advantage over a naive implementation of the block Kaczmarz method applied to the equivalent matricized system, which illustrates the advantage of leveraging the intrinsic structure of tensor data through row-slice or column-slice-action-based \emph{tensor method} updates, rather than relying solely on their matrix counterparts. The methods we propose are part of the broader effort to develop tensor-based approaches for addressing tensor linear systems and tensor regression problems. 

In this paper, we propose a Tensor Randomized Block Gauss-Seidel (TRBGS) algorithm for solving tensor linear regression under the t-product formulation~\eqref{eq:regression}. This method extends the Gauss-Seidel method~\cite{leventhal2010randomized} to tensor settings, drawing inspiration from the TRK algorithm~\cite{ma2022randomized}. Additionally, we introduce a pseudoinverse-free variant, Tensor Randomized Block Averaging Gauss-Seidel (TRBAGS), which eliminates the need to compute the pseudoinverse required by TRBGS. These methods are presented in Section~\ref{subsec:TRBGS}. 

We develop two additional algorithms: Factorized Tensor Randomized Block Gauss-Seidel (FacTRBGS) and Factorized Tensor Randomized Block Averaging Gauss-Seidel (FacTRBAGS). These algorithms extend TRBGS and TRBAGS to scenarios where the measurement tensor $\tA$ is expressed as the t-product of two tensors, $\tU$ and $\tV$, for the problem~\eqref{eq:fact_regression}. These methods are presented in Section~\ref{subsec:facTRBGS}. 

Alongside all proposed methods, we provide theoretical guarantees showing that they converge at least linearly in expectation to the least-norm solution of~\eqref{eq:regression} and~\eqref{eq:fact_regression} under mild assumptions on the given data.  
The proofs of these theoretical results are given in Section~\ref{subsec:proofs}, including a sequence of broadly relevant lemmas in Section~\ref{subsec:usefullemmas}. 
Finally, we validate our theoretical results with numerical experiments on synthetic data in Section~\ref{subsec:synthetic_experiments} and an application to video deblurring in Section~\ref{subsec:deblurring_experiments}.

\section{Methods and Main Results}
\label{sec:methods-results}

In this section, we present our randomized block Gauss-Seidel methods for tensor regression~\eqref{eq:regression} and their theoretical convergence results in Section~\ref{subsec:TRBGS}, and our randomized block Gauss-Seidel methods for factorized tensor regression~\eqref{eq:fact_regression} and their theoretical convergence results in Section~\ref{subsec:facTRBGS}.  We provide and prove useful lemmas in Section~\ref{subsec:usefullemmas} and the proofs of all main convergence results in Section~\ref{subsubsec:proofsofTRBGSandTRBAGS} and \ref{subsubsec:proofsoffacTRBGSandfacTRBAGS}. We note that one of the four methods explored in this current work was proposed in~\cite{tensorregression24}, but no theoretical analysis was provided and three of our four methods are newly proposed.

\subsection{Randomized block Gauss-Seidel methods for tensor regression}\label{subsec:TRBGS}

We now describe two methods, Tensor Randomized Block Gauss-Seidel
(TRBGS) method, Algorithm~\ref{alg:trbgs}, and Tensor Randomized
Block Averaging Gauss-Seidel (TRBAGS), Algorithm~\ref{alg:trbags}, for solving the regression problem~\eqref{eq:regression}. We will prove that each algorithm converges at least linearly in expectation to the least norm solution, as stated in Theorem~\ref{thm:trbgs} and Theorem~\ref{thm:trbags}, respectively. For both theorems, we provide convergence rates in terms of residual errors and, when possible, absolute errors, which, as expected, are functions of the conditioning of the tensors. 

We recall here the tensor linear regression problem~\eqref{eq:regression},
\[
    \min_{\tX \in \mathbb{R}^{n \times l \times p}} \|\tA \tX - \tB\|_F^2,
\]
with $\tA \in \R^{m \times n \times p}$, $\tX \in \R^{n \times l \times p}$, and $\tB \in \R^{m \times l \times p}$. For TRBGS, the update is a natural extension of the block-Gauss Seidel update in the matrix case, i.e.,
\begin{equation}\label{eq:TRBGSupdate}
    \tens{X}^{(k)} = \tens{X}^{(k-1)} -  \tens{E}_{\tau_k}(\tens{A}_{:\tau_k:}^{*}\tens{A}_{:\tau_k:})^{-1} \tens{A}_{:\tau_k:}^{*}\tens{R}^{(k-1)},
\end{equation}
where $\tE_{\tau_k} \in \mathcal{R}^{n \times |\tau_k| \times p}$ is a vertical slice tensor with the first frontal slice being the columns of the $n \times n$ identity matrix indexed by $\tau_k$, and the rest of the frontal slices all zero; the pseudocode is given in Algorithm~\ref{alg:trbgs}. 

\begin{algorithm}
\caption{Tensor randomized block Gauss-Seidel (TRBGS)}\label{alg:trbgs}
\begin{algorithmic}
\Procedure{TRBGS}{$\tA$, $\tB$, $K$} 
\State {$\tens{X}^{(0)} = 0$, $\tens{R}^{(0)} = \tens{A}\tens{X}^{(0)}-\tens{B}$}
\For {$k=1, 2,...,K$} 
\State{Sample $\tau_k \sim \mathcal{D}(T)$}
\State{$\tens{X}^{(k)} = \tens{X}^{(k-1)} -  \tens{E}_{\tau_k}(\tens{A}_{:\tau_k:}^{*}\tens{A}_{:\tau_k:})^{-1} \tens{A}_{:\tau_k:}^{*}\tens{R}^{(k-1)}$}
\State{$\tens{R}^{(k)} = \tens{A}\tens {X}^{(k)} -\tens{B}$}
\EndFor{}
\Return{$\tens{X}^{(K)}$}
\EndProcedure{}
\end{algorithmic}
\end{algorithm}

We now state the main convergence result for Algorithm~\ref{alg:trbgs} and formally
prove the theorem in Section~\ref{subsubsec:proofsofTRBGSandTRBAGS}.

\begin{theorem}\label{thm:trbgs}
    Suppose the blocks $\tau_k$ in Algorithm~\ref{alg:trbgs} are sampled independently from a set $T$ according to some distribution $\mathcal{D}(T)$. 
    Let $\tX^{(k)}$ denote the iterate in the $k$th iteration of Algorithm~\ref{alg:trbgs} applied to the (possibly inconsistent) system defined by $\tA$ and $\tB$ and denote the orthogonal projection operator for block $\tau$ as $$\tP_{\tA_{: \tau :}} := \tA_{: \tau :}(\tA_{: \tau :}^\ast\tA_{: \tau :})^{-1}\tA_{: \tau :}^\ast.$$ Let $\tX^\ddagger := \mathrm{argmin}_{\tX} \|\tA \tX - \tB\|_F^2$.  Then the expected residual error in the $k$th iteration satisfies 
    \begin{equation} 
    \E\left[\|\tA \tX^{(k)} - \tA \tX^\ddagger\|_F^2 | \tX^{(0)}\right] \le \alpha_{\tA}^k\| \tA \tX^{(0)} - \tA \tX^\ddagger\|_F^2,  \label{eq:trbgs_residual_theorem}
    \end{equation}
    where $\alpha_{\tA} = \left(1 - \sigma_{\min}( \bcirc(\E[\tP_{\tA_{: \tau :}}]))\right)$, and the expectation is taken with respect to the probability distribution $\mathcal{D}(T)$. 
    
    Suppose additionally that the system $\tA\tX = \tB_{\mathcal{R}_l(\tA)} := \tA \tX^\ddagger$ has unique solution $\tX = \tX^\ddagger$. Then the expected absolute error in the $k$th iteration satisfies
    \begin{equation} 
    \E\left[\|\tX^{(k)} - \tX^\ddagger\|_F^2 | \tX^{(0)}\right] \le \kappa^2(\tA)\alpha_{\tA}^k\|\tX^{(0)} - \tX^\ddagger\|_F^2, \label{eq:trbgs_error_theorem}
    \end{equation}
    where $\kappa^2(\tA) := \sigma_{\max}^2(\bcirc(\tA^\dagger)) \sigma_{\max}^2(\bcirc(\tA))$, and the expectation is taken with respect to the probability distribution $\mathcal{D}(T)$.
    \label{thm:trbgs}
\end{theorem}

In practice, the main drawback of formula~\eqref{eq:TRBGSupdate} is that each iteration is expensive since we need to compute the pseudoinverse of a tensor. To decrease the computational cost of the solution of TRBGS, we develop TRBAGS as a tensor extension of the (matrix) randomized average block Kaczmarz algorithm (e.g.,~\cite{du2021randomized},~\cite{necoaraIon40}), in which the present update is projected onto individual rows that form the block matrix in consideration, and the resulting projections are averaged to form the next iterate. Specifically, the TRBAGS update is given as
\begin{equation}\label{eq:TRBAGSupdate}
    \tens{X}^{(k)} = \tens{X}^{(k-1)} -  \omega \tens{E}_{\tau_k} \tens{A}_{:\tau_k:}^{*}(\tens{A}\tens {X}^{(k-1)} -\tens{B}),
\end{equation}
where the step size is $\omega > 0$. We will use the typical value $\omega = 1$, but note that Theorem~\ref{thm:trbags} provides an optimal interval for this parameter. Observe also that randomized average block Kaczmarz algorithms usually feature a weighted average of the projections. For simplicity, we do not consider such weighted averages in our work. The pseudocode for TRBAGS is given in Algorithm~\ref{alg:trbags}.

\begin{algorithm}
\caption{Tensor randomized block averaging Gauss-Seidel (TRBAGS)}\label{alg:trbags}
\begin{algorithmic}
\Procedure{TRBAGS}{$\tA$, $\tB$, $K$} 
\State {$\tens{X}^{(0)} = 0$, $\tens{R}^{(0)} = \tens{A}\tens{X}^{(0)}-\tens{B}$}
\For {$k=1, 2,...,K$} 
\State{Sample $\tau_k \sim \text{unif}(T)$}
\State{$\tens{X}^{(k)} = \tens{X}^{(k-1)} -  \omega \tens{E}_{\tau_k} \tens{A}_{:\tau_k:}^{*}\tens{R}^{(k-1)}$}
\State{$\tens{R}^{(k)} = \tens{A}\tens {X}^{(k)} -\tens{B}$}
\EndFor{}
\Return{$\tens{X}^{(K)}$}
\EndProcedure{}
\end{algorithmic}
\end{algorithm}

We now state the main convergence result for Algorithm~\ref{alg:trbags} and formally
prove the theorem in Section~\ref{subsubsec:proofsofTRBGSandTRBAGS}.

\begin{theorem}\label{thm:trbags}
Suppose the blocks $\tau_k$ in Algorithm~\ref{alg:trbags} are sampled independently from a set $T$ according to the distribution $\mathrm{unif}(T)$.  
    Let $\tX^{(k)}$ denote the iterate in the $k$th iteration of Algorithm~\ref{alg:trbags} applied to the (possibly inconsistent) system defined by $\tA$ and $\tB$. Let $\tX^\ddagger := \mathrm{argmin}_{\tX} \|\tA \tX - \tB\|_F^2$, and define $\sigma^2 = \max_{\tau \in T} \sigma_{\max}^2(\bcirc(\tA_{: \tau :}))$ and $c_{\min}(T) = \mathrm{argmin}_{i \in [n]} |\{\tau \in T : i \in \tau\}|$ to be the minimum number of times any column appears in the set of column blocks. Suppose $2  \omega - \omega^2  \sigma^2 > 0$.  Then the expected residual error in the $k$th iteration satisfies 
    \begin{equation} 
    \E\left[\|\tA \tX^{(k)} -\tA\tX^\ddagger\|_F^2 | \tX^{(0)}\right] \le \left(1 - (2  \omega - \omega^2  \sigma^2)\frac{c_{\min}(T)}{|T|}\sigma_{\min}^2(\bcirc(\tA))\right)^k \|\tA \tX^{(0)} - \tA\tX^\ddagger\|_F^2,
      \label{eq:trbags_residual_theorem}
    \end{equation}
    where the expectation is taken with respect to the probability distribution $\mathrm{unif}(T)$. 
    
    Suppose additionally that the system $\tA\tX = \tB_{\mathcal{R}_l(\tA)} := \tA \tX^\ddagger$ has unique solution $\tX = \tX^\ddagger$. Then the expected absolute error in the $k$th iteration satisfies
    \begin{equation} 
    \E\left[\|\tX^{(k)} -\tX^\ddagger\|_F^2 | \tX^{(0)}\right] \le \kappa^2(\tA) \left(1 - (2  \omega - \omega^2  \sigma^2)\frac{c_{\min}(T)}{|T|}\sigma_{\min}^2(\bcirc(\tA))\right)^k \|\tX^{(0)} - \tX^\ddagger\|_F^2,
     \label{eq:trbags_error_theorem}
    \end{equation}
    where $\kappa^2(\tA) := \sigma_{\max}^2(\bcirc(\tA^\dagger)) \sigma_{\max}^2(\bcirc(\tA))$ and the expectation is taken with respect to the probability distribution $\mathrm{unif}(T)$.
\end{theorem}

The absence of the pseudoinverse in the TRBAGS formula 
\eqref{eq:TRBAGSupdate} leads to expected differences in the proof strategy. Unlike the convergence proof for TRBGS, which relies on Pythagoras-type orthogonality results guaranteed due to the use of projectors in the algorithm, the convergence of TRBAGS must be established by carefully bounding certain additional interaction terms that amounted to zero in the proof of convergence for TRBGS.

\subsection{Randomized block Gauss-Seidel methods for tensor regression with factorized operators}\label{subsec:facTRBGS}

Here, we introduce two methods, the Factorized Tensor Randomized Block Gauss-Seidel
(FacTRBGS) method, Algorithm~\ref{alg:factrbgs}, and the Factorized Tensor Randomized Block Averaging Gauss-Seidel (FacTRBAGS), Algorithm~\ref{alg:factrbags}, for solving the regression problem~\eqref{eq:fact_regression}. Both algorithms converge at least linearly in expectation to a solution, as we state in Theorem~\ref{thm:factrbgs} and Theorem~\ref{thm:factrbags}, respectively. For each theorem, convergence rates are given in terms of residual errors and, when possible, absolute errors. 

We recall the factorized tensor regression problem~\eqref{eq:fact_regression}, 
\begin{equation*}
    \min_{\tX \in \mathbb{R}^{n \times l \times p}} \|\tU \tV \tX - \tB\|_F^2, 
\end{equation*}
where $\tU \in \R^{m \times m_1 \times p}$ and $\tV \in \R^{m_1 \times n \times p}$ define the measurement operator, $\tB \in \mathbb{R}^{m \times l \times p}$ represents the measurements or data, and $\tX \in \mathbb{R}^{n \times l \times p}$ is the tensor of interest.  We recall that $\mathcal{P}(S)$ denotes the power set of a set $S$. Let $T_{\tU} \subset \mathcal{P}([m])$ be the index sets of the allowable (able to be selected) blocks of rows of $\tU$ and let $d_{T_{\tU}}$ is the size of the largest set in $T_{\tU}$.  Let $T_{\tV} \subset \mathcal{P}([m_1])$ be the index sets of the allowable blocks of rows of $\tV$ and let $d_{T_{\tV}}$ is the size of the largest set in $T_{\tV}$.  Let $\mathcal{D}(T_{\tU})$ and $\mathcal{D}(T_{\tV})$ be the sampling probability distributions over the sets $T_{\tU}$ and $T_{\tV}$, respectively.  Now, our first algorithm for this problem~\eqref{eq:fact_regression} interlaces steps of Algorithm~\ref{alg:trbgs} to address the outer system defined by $\tU$ and the inner system defined by $\tV$; the pseudocode for this method is given in Algorithm~\ref{alg:factrbgs}.

\begin{algorithm}
\caption{Factorized tensor randomized block Gauss-Seidel (FacTRBGS)}\label{alg:factrbgs}
\begin{algorithmic}
\Procedure{FacTRBGS}{$\tU$, $\tV$,$\tB$,$K$} 
\State {$\tens{Z}^{(0)} = 0$, $\tens{R}_1^{(0)} = \tens{U}\tens{Z}^{(0)}-\tens{B}$}
\State{$\tens{X}^{(0)} = 0$, $\tens{R}_2^{(0)} = \tens{V}\tens{X}^{(0)}-\tens{Z}$}
\For {$k=1, 2,...,K$} 
\State{Sample $\mu_k \sim \mathcal{D}(T_{\tU})$}
\State{$\tens{Z}^{(k)} = \tens{Z}^{(k-1)} -  \tens{E}_{\mu_k}(\tens{U}_{:\mu_k:}^{*}\tens{U}_{:\mu_k:})^{-1} \tens{U}_{:\mu_k:}^{*}\tens{R}_1^{(k-1)}$}
\State{$\tens{R}_1^{(k-1)} = \tens{U}\tens {Z}^{(k-1)} -\tens{B}$}
\State{Sample $\nu_k \sim \mathcal{D}(T_{\tV})$}
\State{$\tens{X}^{(k)} = \tens{X}^{(k-1)} -  \tens{E}_{\nu_k}(\tens{V}_{:\nu_k:}^{*}\tens{V}_{:\nu_k:})^{-1} \tens{V}_{:\nu_k:}^{*}\tens{R}_2^{(k-1)}$}
\State{$\tens{R}_2^{(k-1)} = \tens{V}\tens {X}^{(k-1)} -\tens{Z}^{(k)}$}
\EndFor{}
\Return{$\tens{X}^{(K)}$}
\EndProcedure{}
\end{algorithmic}
\end{algorithm}

We now state the main convergence result for Algorithm~\ref{alg:factrbgs} and formally
prove the theorem in Section~\ref{subsubsec:proofsoffacTRBGSandfacTRBAGS}.

\begin{theorem}\label{thm:errorsboundsonfactrbgs}
    Suppose the blocks $\mu_k$ and $\nu_k$ in Algorithm~\ref{alg:factrbgs} are sampled independently from sets $T_{\tU}$ and $T_{\tV}$ according to distributions $\mathcal{D}(T_{\tU})$ and $\mathcal{D}(T_{\tV})$, respectively.  
    Let $\tZ^{(k)}$ and $\tX^{(k)}$ denote the iterates in the $k$th iteration of Algorithm~\ref{alg:factrbgs} applied to the factorized system defined by $\tU\tV$ and $\tB$, and denote the orthogonal projection operators for blocks $\mu$ and $\nu$ as $$\tP_{\tU_{: \mu :}} := \tU_{: \mu :}(\tU_{: \mu :}^\ast\tU_{: \mu :})^{-1}\tU_{: \mu :}^\ast \text{ and } \tP_{\tV_{: \nu :}} := \tV_{: \nu :}(\tV_{: \nu :}^\ast\tV_{: \nu :})^{-1}\tV_{: \nu :}^\ast,$$ respectively. 
    Let $\tZ^\ddagger = \mathrm{argmin}_{\tZ} \|\tU \tZ - \tB\|_F^2$ and $\tX^\ddagger := \mathrm{argmin}_{\tX} \|\tV \tX - \tZ^\ddagger\|_F^2$. Define $\alpha_\tV = 1 - \sigma_{\min} (\E[\mathrm{bcirc}(\mathcal{P}_{\tV_{:\nu:}})])$, $\alpha_\tU = 1 - \sigma_{\min} (\E[\mathrm{bcirc}(\mathcal{P}_{\tU_{:\mu:}})])$ and $\alpha_{\max} = \max\{\alpha_\tU, \alpha_\tV\}$ and $\alpha_{\min} := \min\left\{ \frac{\alpha_{\tU}}{\alpha_{\tV}}, \frac{\alpha_{\tU}}{\alpha_{\tV}}\right\}$. 
    
    Then the expected residual error for the outer system in the $k$th iteration satisfies 
    \begin{equation} 
    \E\left[\|\tU \tZ^{(k)} - \tU \tZ^\ddagger\|_F^2 | \tZ^{(0)}\right] \le \alpha_{\tU}^k\| \tU \tZ^{(0)} - \tU \tZ^\ddagger\|_F^2,  \label{eq:factrbgs_outer_residual_theorem}
    \end{equation}
    where the expectation is taken with respect to the probability distribution $\mathcal{D}(T_{\tU})$. 
    
    Suppose additionally that the system $\tU\tZ = \tB_{\mathcal{R}_l(\tU)} := \tU \tZ^\ddagger$ has unique solution $\tZ = \tZ^\ddagger$, then the expected absolute error for the outer system in the $k$th iteration satisfies
    \begin{equation} 
    \E\left[\|\tZ^{(k)} - \tZ^\ddagger\|_F^2 | \tZ^{(0)}\right] \le \kappa^2(\tU)\alpha_{\tU}^k\|\tZ^{(0)} - \tZ^\ddagger\|_F^2, \label{eq:factrbgs_outer_error_theorem}
    \end{equation}
    where $\kappa^2(\tU) := \sigma_{\max}^2(\bcirc(\tU^\dagger)) \sigma_{\max}^2(\bcirc(\tU))$ and the expectation is taken with respect to the probability distribution $\mathcal{D}(T_{\tU})$.  In addition, we have that the expected residual error for the inner system in the $k$th iteration satisfies
    \begin{equation*}
    \E\left[\|\tV \tX^{(k)} - \tV \tX^\ddagger\|_F^2 | \tX^{(0)}\right] \le \begin{cases} 
    \alpha_\tV^k \|\tV\tX^\ddagger\|_F^2 + \kappa^2(\tU)\cfrac{\alpha_{\max}^k \alpha_{\min}}{1 - \alpha_{\min}}\|\tZ^\ddagger\|_F^2, & \text{if }\; \alpha_\tU \neq \alpha_\tV \\ \alpha_\tV^k \|\tV\tX^\ddagger\|_F^2 + \kappa^2(\tU)k \alpha_{\max}^k \|\tZ^\ddagger\|_F^2, & \text{if }\; \alpha_\tU = \alpha_\tV \end{cases}. \label{eq:factrbgs_inner_residual_theorem}
    \end{equation*}

    Furthermore, assume that the system $\tV \tX = \tZ^{\ddagger}_{\mathcal{R}_l(\tV)} := \tV \tX^{\ddagger}$ has unique solution $\tX = \tX^\ddagger$, then the expected absolute error for the inner system in the $k$th iteration satisfies
    \begin{equation} 
    \E\left[\|\tX^{(k)} - \tX^\ddagger\|_F^2 | \tX^{(0)}\right] \le \begin{cases}     
    \kappa^2(\tV)\alpha_\tV^k \|\tX^*\|_F^2 + \sigma_{\max}^2(\bcirc(\tV^\dagger))\kappa^2(\tU)\cfrac{\alpha_{\max}^k \alpha_{\min}}{1 - \alpha_{\min}}\|\tZ^\ddagger\|_F^2, & \text{if }\; \alpha_\tU \neq \alpha_\tV \\ \kappa^2(\tV)\alpha_\tV^k \|\tX^*\|_F^2 + \sigma_{\max}^2(\bcirc(\tV^\dagger))\kappa^2(\tU)k \alpha_{\max}^k \|\tZ^\ddagger\|_F^2, & \text{if }\; \alpha_\tU = \alpha_\tV \end{cases}, \label{eq:factrbgs_inner_error_theorem}
    \end{equation}
    where $\kappa^2(\tV) := \sigma_{\max}^2(\bcirc(\tV^\dagger)) \sigma_{\max}^2(\bcirc(\tV))$ and the expectation is taken with respect to the probability distribution $\mathcal{D}(T_{\tV})$.
    \label{thm:factrbgs}
\end{theorem}

Now, as with Algorithm~\ref{alg:trbags}, we replace the pseudoinverse in the previous method with an averaging approach. Our next algorithm for problem~\eqref{eq:fact_regression} interlaces steps of Algorithm~\ref{alg:trbags} to address the outer system defined by $\tU$ and the inner system defined by $\tV$; pseudocode for this method is given in Algorithm~\ref{alg:factrbags}. The step-sizes $\omega_1$ and $\omega_2$ can be set using the conditions stated in Theorem~\ref{thm:errorsboundsonfactrbags}.

\begin{algorithm}
\caption{Factorized tensor randomized block averaging Gauss-Seidel (FacTRBAGS)}\label{alg:factrbags}
\begin{algorithmic}
\Procedure{FacTRBAGS}{$\tU$, $\tV$,$\tB$,$K$} 
\State {$\tens{Z}^{(0)} = 0$, $\tens{R}_1^{(0)} = \tens{U}\tens{Z}^{(0)}-\tens{B}$}
\State{$\tens{X}^{(0)} = 0$, $\tens{R}_2^{(0)} = \tens{V}\tens{X}^{(0)}-\tens{Z}$}
\For {$k=1, 2,...,K$}
\State{Sample $\mu_k \sim \text{unif}(T_\tU)$}
\State{$\tens{Z}^{(k)} = \tens{Z}^{(k-1)} -  \omega_1 \tens{E}_{\mu_k} \tens{U}_{:\mu_k:}^{*}\tens{R}_1^{(k-1)}$}
\State{$\tens{R}_1^{(k-1)} = \tens{U}\tens {Z}^{(k-1)} -\tens{B}$}
\State{Sample $\nu_k \sim \text{unif}(T_\tV)$}
\State{$\tens{X}^{(k)} = \tens{X}^{(k-1)} -  \omega_2 \tens{E}_{\nu_k} \tens{V}_{:\nu_k:}^{*}\tens{R}_2^{(k-1)}$}
\State{$\tens{R}_2^{(k-1)} = \tens{V}\tens {X}^{(k-1)} -\tZ^{(k)}$}

\EndFor{}
\Return{$\tens{X}^{(K)}$}
\EndProcedure{}
\end{algorithmic}
\end{algorithm}

We now state the main convergence result for Algorithm~\ref{alg:factrbags} and formally
prove the theorem in Section~\ref{subsubsec:proofsoffacTRBGSandfacTRBAGS}.

\begin{theorem}\label{thm:errorsboundsonfactrbags}
    Suppose the blocks $\mu_k$ and $\nu_k$ in Algorithm~\ref{alg:factrbags} are sampled independently from sets $T_{\tU}$ and $T_{\tV}$ according to distributions $\mathrm{unif}(T_{\mu})$ and $\mathrm{unif}(T_{\nu})$, respectively.  
    Let $\tZ^{(k)}$ and $\tX^{(k)}$ denote the iterates in the $k$th iteration of Algorithm~\ref{alg:factrbags} applied to the factorized system defined by $\tU\tV$ and $\tB$. Let $\tZ^\ddagger = \mathrm{argmin}_{\tZ} \|\tU \tZ - \tB\|_F^2$ and $\tX^\ddagger := \mathrm{argmin}_{\tX} \|\tV \tX - \tZ^\ddagger\|_F^2$. 
  
    Define $\sigma_\tU^2 = \max_{\mu \in T_\tU} \sigma_{\max}^2(\bcirc(\tU_{: \mu :}))$ and let $c_{\min}(T_\tU) = \mathrm{argmin}_{i \in [m_2]} |\{\mu \in T_\tU : i \in \mu\}|$ be the minimum number of times any column appears in the set of column blocks of $\tU$. Suppose $2 \omega_1 - \omega_1^2  \sigma_\tU^2 > 0$. Then the expected residual error in the $k$th iteration satisfies 
    \begin{equation} 
        \E\left[\|\tU \tZ^{(k)} -\tU\tZ^\ddagger\|_F^2 | \tZ^{(0)}\right] \le \beta^k_{\tU}\|\tU \tZ^{(0)} - \tU\tZ^\ddagger\|_F^2,
      \label{eq:factrbags_outer_residual_theorem}
    \end{equation}
    where $\beta_\tU = \left(1 - (2  \omega_1 - \omega_1^2  \sigma_\tU^2)\frac{c_{\min}(T_\tU)}{|T_\tU|}\sigma_{\min}^2(\bcirc(\tU))\right)$. 
    
    Suppose additionally that the system $\tU\tZ = \tB_{\mathcal{R}_l(\tU)} := \tU \tZ^\ddagger$ has unique solution $\tZ = \tZ^\ddagger$, then the expected absolute error for the outer system in the $k$th iteration satisfies
    \begin{equation} 
        \E\left[\|\tZ^{(k)} - \tZ^\ddagger\|_F^2 | \tZ^{(0)}\right] \le \kappa^2(\tU)\beta_{\tU}^k\|\tZ^{(0)} - \tZ^\ddagger\|_F^2, \label{eq:factrbags_outer_error_theorem}
    \end{equation}
    where $\kappa^2(\tU) := \sigma_{\max}^2(\bcirc(\tU^\dagger)) \sigma_{\max}^2(\bcirc(\tU))$.  
    
    Furthermore, define $\sigma_\tV^2 = \max_{\nu \in T_\tV} \sigma_{\max}^2(\bcirc(\tV_{: \nu :}))$, $\gamma_{\tV} = 2\max_{\nu \in T_\tV} \sigma_{\max}^2(\bcirc(\tV_{: \nu :}))\sigma_{\max}^2(\bcirc(\tV_{: \nu :}^*)$ and let $c_{\min}(T_\tV) = \mathrm{argmin}_{i \in [m_1]} |\{\nu \in T_\tV : i \in \nu\}|$ denote the minimum number of times any column appears in the set of column blocks of $\tV$. Set $\beta_{\max}:= \max\{\beta_\tU, \beta_\tV\}$ and $\beta_{\min}:= \min\left\{ \frac{\beta_{\tU}}{\beta_{\tV}}, \frac{\beta_{\tU}}{\beta_{\tV}}\right\}$, and assume $2 \omega_2 - \omega_2^2  \sigma_\tV^2 > 0$. Then, the expected residual error for the inner system in the $k$th iteration satisfies
    \begin{equation}
        \E\left[\|\tV \tX^{(k)} - \tV \tX^\ddagger\|_F^2 | \tX^{(0)}\right] \le \begin{cases} 
    \beta_\tV^k \|\tV\tX^\ddagger\|_F^2 + \kappa^2(\tU)\omega_2^2\gamma_{\tV} \cfrac{\beta_{\max}^k \beta_{\min}}{1 - \beta_{\min}}\|\tZ^\ddagger\|_F^2, & \text{if }\; \beta_\tU \neq \beta_\tV \\ \beta_\tV^k \|\tV\tX^\ddagger\|_F^2 + \kappa^2(\tU)k \omega_2^2\gamma_{\tV} \beta_{\max}^k \|\tZ^\ddagger\|_F^2, & \text{if }\; \beta_\tU = \beta_\tV \end{cases}, \label{eq:factrbags_inner_residual_theorem}
    \end{equation}
    where $\beta_\tV = 2 - 2(2\omega_2 - \omega_2^2 \sigma_\tV^2)\frac{c_{\min}(T_\tV)}{|T_\tV|} \sigma^2_{\min}(\mathrm{bcirc}(\tV))$.
    
    If, additionally, we assume that the system $\tV \tX = \tZ^{\ddagger}_{\mathcal{R}_l(\tV)} := \tV \tX^{\ddagger}$ has unique solution $\tX = \tX^\ddagger$, then the expected absolute error for the inner system in the $k$th iteration satisfies
    \begin{equation} 
        \E\left[\|\tX^{(k)} - \tX^\ddagger\|_F^2 | \tX^{(0)}\right] \le 
        \begin{cases}     
        \kappa^2(\tV)\beta_\tV^k \|\tX^\ddagger\|_F^2 + \sigma_{\max}^2(\bcirc(\tV^\dagger))\kappa^2(\tU) \omega_2^2\gamma_\tV \cfrac{\beta_{\max}^k \beta_{\min}}{1 - \beta_{\min}}\|\tZ^\ddagger\|_F^2, & \text{if }\; \beta_\tU \neq \beta_\tV \\ 
        \kappa^2(\tV)\beta_\tV^k \|\tX^\ddagger\|_F^2 + \sigma_{\max}^2(\bcirc(\tV^\dagger))\kappa^2(\tU)k \omega_2^2\gamma_\tV \beta_{\max}^k \|\tZ^\ddagger\|_F^2, & \text{if }\; \beta_\tU = \beta_\tV \end{cases}, \label{eq:factrabgs_inner_error_theorem}
    \end{equation}
    where $\kappa^2(\tV) := \sigma_{\max}^2(\bcirc(\tV^\dagger)) \sigma_{\max}^2(\bcirc(\tV))$.
    \label{thm:factrbags}
\end{theorem}

\section{Proofs of Main Results}\label{subsec:proofs}

\subsection{Lemmas}\label{subsec:usefullemmas}

We present four lemmas that play a crucial role in the proofs of our main theoretical results, Theorems~\ref{thm:trbgs} through~\ref{thm:factrbags}.

The first lemma provides bounds on the Frobenius norm of the products of tensors using singular values, including a special case when a projector is one of the product factors. This lemma is key in deriving the tensors' conditioning constants. 

\begin{lemma}\label{lem:boundsforFrobeniusnorm}
    Let $\mathbf{A}$ and $\mathbf{B}$ be conformable matrices. Then
    \begin{equation}\label{eqn:boundsforFrobnorm}
           \sigma_{\min}^2(\mathbf{A}) \| \mathbf{B} \|_F^2 \leq \| \mathbf{A} \mathbf{B} \|_F^2 \leq \sigma_{\max}^2(\mathbf{A}) \| \mathbf{B} \|_F^2.
    \end{equation}
    If $\tM$ and $\tY$ are tensors and $\tP$ is an orthogonal projector tensor, such that the products $\tM \tY$ and $\tP \tM$ exist, it follows that
    \begin{align}
        \sigma_{\min}^2(\bcirc(\tM)) \|\tY\|_F^2 
        &\le \|\tM \tY\|_F^2 \le \sigma_{\max}^2(\bcirc(\tM)) \|\tY\|_F^2 \label{eqn:frobboundpart1}, \\
        \intertext{and}
        \sigma^2_{\min}(\mathbb{E}_{\mathcal{\tP}}[\mathrm{bcirc}(\tP)] \| \tM \|_F^2 
 &\le \mathbb{E}_{\tP}\left[ \| \tP\tM \|_F^2  \right] \le \sigma^2_{\max}(\mathbb{E}_{\mathcal{\tP}}[\mathrm{bcirc}(\tP)] \| \tM \|_F^2 \label{eqn:frobboundpart2}.
    \end{align}
    
\end{lemma}

\begin{proof}
    We start by writing the Frobenius norm of the product $\mathbf{A} \mathbf{B}$ in terms of a 2-norm:
    \begin{align*}
        \| \mathbf{A}\mathbf{B} \|_F^2
        &= \sum\limits_{i} \| \mathbf{A} \mathbf{B}_{:i} \|_2 \\
        &= \sum\limits_{i} \langle \mathbf{A} \mathbf{B}_{:i},  \mathbf{A} \mathbf{B}_{:i} \rangle \\
        &= \sum\limits_{i} \langle \mathbf{A}^*\mathbf{A} \mathbf{B}_{:i}, \mathbf{B}_{:i} \rangle \\
        &= \sum\limits_{i} \mathbf{B}_{:i}^\top(\mathbf{A}^*\mathbf{A}) \mathbf{B}_{:i}.
        \intertext{The equation in~\eqref{eqn:boundsforFrobnorm} is readily obtained from observing}
        \lambda_{\min}(\mathbf{A}^*\mathbf{A}) \| \mathbf{B}_{:i} \|_2^2  &\leq \mathbf{B}_{:i}^\top(\mathbf{A}^*\mathbf{A}) \mathbf{B}_{:i} \leq \lambda_{\max}(\mathbf{A}^*\mathbf{A})^2 \| \mathbf{B}_{:i} \|_2^2, 
        \intertext{so}
         \sigma^2_{\min}(\mathbf{A}) \| \mathbf{B}_{:i} \|_2^2  &\leq \mathbf{B}_{:i}^\top(\mathbf{A}^*\mathbf{A}) \mathbf{B}_{:i} \leq \sigma^2_{\max}(\mathbf{A}) \| \mathbf{B}_{:i} \|_2^2, 
    \end{align*}
    To obtain~\eqref{eqn:frobboundpart1}, it suffices to note 
    \[
        \| \tM \tY \|_F^2
        = \| \bcirc(\tM) \unfold(\tY) \|_F^2. 
    \]

    \noindent For~\eqref{eqn:frobboundpart2}, we use an argument similar to those above to arrive at 
    \begin{align*}
        \mathbb{E}_{\mathcal{\tP}}\left[ \| \tP\tM \|_F^2  \right] 
        &= \sum\limits_{i}
        \mathbb{E}_{\mathcal{\tP}}\langle \bcirc{(\tP)} \unfold(\tM)_{:i}, \bcirc{(\tP)} \unfold(\tM)_{:i} \rangle \\
        &= \sum\limits_{i}
        \mathbb{E}_{\tP}
        \langle \bcirc{(\tP)} \unfold(\tM)_{:i}, \unfold(\tM)_{:i} \rangle \\
        &= \sum\limits_{i}
        \langle \mathbb{E}_{\tP}[\bcirc{(\tP)}] \unfold(\tM)_{:i}, \unfold(\tM)_{:i} \rangle,
    \end{align*}
    where we have used the fact that $\bcirc{(\tP)}$ is an orthogonal projection. Equation~\eqref{eqn:frobboundpart2} follows from the fact that $\mathbb{E}_{\tP}[\bcirc{(\tP)}]$ is symmetric.
\end{proof}

The following lemma gives a formula for the pseudoinverse of a tensor for a system with a unique solution. This is particularly useful for the TRBGS and FacTRBGS algorithms, which are projection-based block methods.

\begin{lemma}
    Suppose $\tA \tX = \tA \tX^*$ has a unique solution $\tX = \tX^*$, then $\tA^\dagger \tA = \tI$, where the pseudoinverse is $\tA^\dagger = (\tA^\top \tA)^{-1} \tA^\top$, and $\tI$ is the tensor identity.\label{lemma:pseudoinverse}
\end{lemma}

\begin{proof}
    By the assumption and Fact~\ref{fact:equivalent_systems}, $\bcirc(\tA) \mathbf{X} = \tB := \bcirc(\tA)\unfold(\tX^*)$ has a unique solution $\mathbf{X} = \unfold(\tX^*)$.  We then have that $\bcirc(\tA)^\dagger = (\bcirc(\tA)^\top \bcirc(\tA))^{-1} \bcirc(\tA)^\top$.  Thus, we have 
    \begin{align*}
        \bcirc(\tA^\dagger \tA) &= \bcirc(\tA^\dagger)\bcirc(\tA)
        \\&= (\bcirc(\tA)^\top \bcirc(\tA))^{-1} \bcirc(\tA)^\top \bcirc(\tA)
        \\&= \mathbf{I}
        \\&= \bcirc(\tI).
    \end{align*}
    Thus, we have $\tA^{\dagger}\tA = \tI$.
\end{proof}

The third lemma produces a bound on the absolute error norm on the approximate solution of a linear system given that the bound of the norm of the residual error is known. It therefore generalizes the derivation of the absolute error from the residual error for all the cases we discuss. 

\begin{lemma}\label{lem:abserrorfromereserror}
Suppose that the system $\tA\tX = \tY := \tA \tX^\ddagger$ has unique solution $\tX = \tX^\ddagger$, and assume that $$\|\tA \tX^{(k)} - \tA \tX^\ddagger\|_F^2 \le \beta \|\tA \tX^{(0)} - \tA \tX^\ddagger\|_F^2 + \gamma$$ for some values $\beta$ and $\gamma$.
Then $$\|\tX^{(k)} - \tX^*\|_F^2 \le \kappa^2(\tA) \beta \|\tX^{(0)} - \tX^*\|_F^2 + \sigma_{\max}^2(\bcirc(\tA^\dagger)) \gamma$$ where $\kappa^2(\tA) := \sigma_{\max}^2(\bcirc(\tA^\dagger)) \sigma_{\max}^2(\bcirc(\tA)).$\label{lem:residual_to_error}
\end{lemma}

\begin{proof}
    We have 
    \begin{align*}
        \|\tX^{(k)} - \tX^\ddagger\|_F^2 &= \|\tA^\dagger \tA(\tX^{(k)} - \tX^\ddagger)\|_F^2
        \\&\le \sigma_{\max}^2(\bcirc(\tA^\dagger)) \|\tA(\tX^{(k)} - \tX^\ddagger)\|_F^2
        \\&\le \sigma_{\max}^2(\bcirc(\tA^\dagger)) \left[\beta \|\tA(\tX^{(0)} - \tX^\ddagger)\|_F^2 + \gamma\right]
        \\&\le \sigma_{\max}^2(\bcirc(\tA^\dagger))\sigma_{\max}^2(\bcirc(\tA))\beta \|\tX^{(0)} - \tX^\ddagger\|_F^2 + \sigma_{\max}^2(\bcirc(\tA^\dagger)) \gamma
    \end{align*}
    where the first equation follows from Lemma~\ref{lemma:pseudoinverse} and the first and last inequalities follow from Lemma~\ref{lem:boundsforFrobeniusnorm} (\ref{eqn:frobboundpart1}).
\end{proof}

Finally, our last lemma, which comes from~\cite{{castillo2024randomized}}, provides a standard orthogonality result between the minimal residual error tensor and the range space of $\tA$.

\begin{lemma}\label{lem:usual_least_squares_theory}\cite[Lemma 5]{castillo2024randomized}
    Let $\tY_{\mathcal{R}_l(\tA)^\perp}$ be the solution of $\tA^* \tW = \mathbf{0}$ that minimizes $\|\tY - \tW\|_F$ for $\tY \in \mathbb{R}^{m \times l \times p}$.  Define $\tY_{\mathcal{R}_l(\tA)} = \tY - \tY_{\mathcal{R}_l(\tA)^\perp}.$ Let $\tX^\ddagger$ be the least-norm minimizer of $\|\tA \tX - \tY\|_F^2$. Then $$\tY_{\mathcal{R}_l(\tA)} = \tA \tX^\ddagger.$$
\end{lemma}

We now prove our main results, Theorems~\ref{thm:trbgs}--\ref{thm:factrbags}.

\subsection{Proofs of main results for TRBGS and TRBAGS}\label{subsubsec:proofsofTRBGSandTRBAGS}

We prove Theorems~\ref{thm:trbgs} and~\ref{thm:trbags}, which state that Algorithms~\ref{alg:trbgs} and~\ref{alg:trbags} converge at least linearly in expectation to solutions of problem~\eqref{eq:regression}.

\begin{proof}[Proof of Theorem~\ref{thm:trbgs}]
 Assume that in the $k$-th iteration, the column block $\tau_k$ is chosen to perform the TRBGS update. Then, we have
 \begin{align*}
     \tA\tX^{(k)} - \tA\tX^\ddagger  &= \tA\tX^{(k-1)} - \tA\tX^{\ddagger} - \tA_{:\tau_k:} (\tA_{:\tau_k:}^* \tA_{:\tau_k:})^{-1}\tA_{:\tau_k:}^* \left( \tA \tX^{(k-1)} - \tB \right)\\
     &= \tA\tX^{(k-1)} - \tA\tX^{\ddagger} - \tP_{\tA_{:\tau_k:}} \left(  \tA\tX^{(k-1)} - \tA\tX^{\ddagger}\right) + \tP_{\tA_{:\tau_k:}}\tB_{\mathcal{R}_l(\tA)^\perp} \\
     &= (\tI - \tP_{\tA_{:\tau_k:}})(\tA\tX^{(k-1)} - \tA\tX^{\ddagger}),
 \end{align*}
 where the second and third equalities follow from~\cite[Lemma 5]{castillo2024randomized} which shows that $\tB = \tA\tX^\ddagger + \tB_{\mathcal{R}_l(\tA)^\perp}$ where $\tA^* \tB_{\mathcal{R}_l(\tA)^\perp} = \mathbf{0}$.  

Applying  the norm and taking the expectation conditioned on the first $(k-1)$ iterations, we have
\begin{align*} 
    \E^{(k-1)} \| \tA\tX^{(k)} - \tA\tX^{\ddagger}\|_F^2 
    &\leq (1 - \sigma_{\min} (\E[\text{bcirc}(\tP_{\tA_{:\tau_k:}})])) \| \tA\tX^{(k-1)} - \tA\tX^{\ddagger} \|_F^2, \nonumber 
\end{align*}
by \eqref{eqn:frobboundpart1} in Lemma~\ref{lem:boundsforFrobeniusnorm}.  
Then, \eqref{eq:trbgs_residual_theorem} follows by utilizing the tower property of conditional expectation and the fact that the sampling distribution is the same in each iteration, and \eqref{eq:trbgs_error_theorem} follows from Lemma~\ref{lem:residual_to_error}.
 \end{proof}

 \begin{proof}[Proof of Theorem~\ref{thm:trbags}]
First, we note
\begin{align*}
    \tA\left(\tX^k - \tX^\ddagger\right) &= \tA\tX^{(k-1)} - \tA\tX^\ddagger- \omega\tA \tens{E}_{\tau_k} \tens{A}_{:\tau_k:}^{*}\tens{R}^{(k-1)} \\ 
    &= \tA\tX^{(k-1)}- \tA\tX^\ddagger - \omega\tens{A}_{:\tau_k:}\tens{A}_{:\tau_k:}^{*}\left(\tA\tens{X}^{(k-1)} - 
    \tA\tX^\ddagger \right) + \omega \tens{A}_{:\tau_k:}\tens{A}_{:\tau_k:}^{*} \tB_{\mathcal{R}_l(\tA)^\perp}\\
    &= \left(\tI - \omega\tens{A}_{:\tau_k:}\tens{A}_{:\tau_k:}^{*}\right)\left(\tA\tens{X}^{(k-1)} - \tA\tX^\ddagger \right),
    \end{align*}
where the second and third equalities follow from Lemma~\ref{lemma:pseudoinverse} which shows that $\tB = \tA\tX^\ddagger + \tB_{\mathcal{R}_l(\tA)^\perp}$ where $\tA^* \tB_{\mathcal{R}_l(\tA)^\perp} = \mathbf{0}$.

Let $\sigma^2 = \max_{\tau \in T} \sigma_{\max}^2(\bcirc(\tA_{: \tau :})).$ Then,
   \begin{align*}
       \|\tA \tX^{(k)} - \tA\tX^\ddagger\|_F^2 &= \|\left(\tI - \omega \tA_{: \tau_k :} \tA_{: \tau_k :}^*\right) (\tA \tX^{(k-1)} - \tA\tX^\ddagger)\|_F^2
       \\&= \|\tA \tX^{(k-1)} - \tA\tX^\ddagger\|_F^2 - 2  \omega \|\tA_{: \tau_k :}^* (\tA \tX^{(k-1)} - \tA\tX^\ddagger)\|_F^2 + \omega^2 \|\tA_{: \tau_k :} \tA_{: \tau_k :}^* (\tA \tX^{(k-1)} - \tA\tX^\ddagger)\|_F^2\\ 
       &\le \|\tA \tX^{(k-1)} - \tA\tX^\ddagger\|_F^2 - 2  \omega \|\tA_{: \tau_k :}^* (\tA \tX^{(k-1)} - \tA\tX^\ddagger)\|_F^2 \\
       &+ \omega^2 \sigma_{\max}^2(\bcirc(\tA_{: \tau_k :}) \|\tA_{: \tau_k :}^* (\tA \tX^{(k-1)} - \tA\tX^\ddagger)\|_F^2 \\
       &\le \|\tA \tX^{(k-1)} - \tA\tX^\ddagger\|_F^2 - 2  \omega \|\tA_{: \tau_k :}^* (\tA \tX^{(k-1)} - \tA\tX^\ddagger)\|_F^2 + \omega^2 \sigma^2 \|\tA_{: \tau_k :}^* (\tA \tX^{(k-1)} - \tA\tX^\ddagger)\|_F^2\\
       &= \|\tA \tX^{(k-1)} - \tA\tX^\ddagger\|_F^2 - (2\omega - \omega^2 \sigma^2) \|\tA_{: \tau_k :}^* (\tA \tX^{(k-1)} - \tA\tX^\ddagger)\|_F^2, 
   \end{align*}  
   where the first inequality follows from formula~\eqref{eqn:frobboundpart1} in Lemma~\ref{lem:boundsforFrobeniusnorm}.

Taking expectation conditioned on the first $k-1$ iterations yields
\begin{align*}
\E^{(k-1)}\left[\|(\tA \tX^{(k)} -\tA\tX^\ddagger)\|_F^2\right] =& \|\tA \tX^{(k-1)} - \tA\tX^\ddagger\|_F^2 - \left(2  \omega - \omega^2  \sigma^2 \right)\E^{(k-1)}\left[\| \tA_{: \tau_k :}^* (\tA \tX^{(k-1)} - \tA\tX^\ddagger)\|_F^2\right]\\
=& \|\tA \tX^{(k-1)} - \tA\tX^\ddagger\|_F^2 - \frac{2  \omega - \omega^2  \sigma^2}{|T|} \sum_{\tau \in T} \| \tA_{: \tau :}^* (\tA \tX^{(k-1)} - \tA\tX^\ddagger)\|_F^2\\
\le& \|\tA \tX^{(k-1)} - \tA\tX^\ddagger\|_F^2 - (2  \omega - \omega^2  \sigma^2)\frac{c_{\min}(T)}{|T|}  \| \tA^* (\tA \tX^{(k-1)} - \tA\tX^\ddagger)\|_F^2\\
\le& \left(1 - (2  \omega - \omega^2  \sigma^2)\frac{c_{\min}(T)}{|T|}\sigma_{\min}^2(\bcirc(\tA) )\right) \|\tA \tX^{(k-1)} - \tA\tX^\ddagger\|_F^2,
\end{align*}
where the last inequality follows from formula  \eqref{eqn:frobboundpart1} in Lemma~\ref{lem:boundsforFrobeniusnorm}. 
By utilizing the tower property of conditional expectation and the fact that the sampling distribution is the same in each iteration, we get equation \eqref{eq:trbags_residual_theorem}, and \eqref{eq:trbags_error_theorem} follows from Lemma~\ref{lem:residual_to_error}.
\end{proof}

\subsection{Proofs of main results for FacTRBGS and FacTRBAGS}\label{subsubsec:proofsoffacTRBGSandfacTRBAGS}
We now prove Theorems~\ref{thm:factrbgs} and~\ref{thm:factrbags}, which illustrate that Algorithms~\ref{alg:factrbgs} and~\ref{alg:factrbags} converge at least linearly in expectation to solutions of the problem~\eqref{eq:fact_regression}.  These proofs follow a straightforward script. They leverage Theorems~\ref{thm:trbgs} and~\ref{thm:trbags}, respectively, to establish linear convergence for the outer system. For the inner system, in Theorem~\ref{thm:factrbgs} we bound the residual using a Pythagorean identity before recursively deriving convergence constants. In Theorem~\ref{thm:factrbags} we follow a similar reasoning but without relying on a Pythagorean identity.

\begin{proof}[Proof of Theorem~\ref{thm:factrbgs}]
    Note that \eqref{eq:factrbgs_outer_residual_theorem} and \eqref{eq:factrbgs_outer_error_theorem} follow directly from Theorem~\ref{thm:trbgs} and the fact that the $\tZ^{(k)}$ iterates for the outer system are independent of the $\tX^{(k)}$ iterates for the inner system.

    Now, we have
    \begin{align*}
        \|\tV\tX^{(k)} - \tV\tX^\ddagger\|_F^2 &= \|\tV\tX^{(k-1)} - \tV\tX^\ddagger - \tV_{:\nu_k:}(\tV_{:\nu_k:}^*\tV_{:\nu_k:})^{-1}\tV_{:\nu_k:}^*(\tV\tX^{(k-1)} - \tZ^{(k)})\|_F^2
        \\&= \left\|(\tI - \tP_{\tV_{:\nu_k:}})(\tV\tX^{(k-1)} - \tV\tX^\ddagger) + \tP_{\tV_{:\nu:}}\left(\tZ^{(k)} - \tZ^\ddagger_{\mathcal{R}_l(\tV)}\right)\right\|_F^2
        \\&= \|(\tI - \tP_{\tV_{:\nu_k:}})(\tV\tX^{(k-1)} - \tV\tX^\ddagger)\|_F^2 + \left\|\tP_{\tV_{:\nu_k:}}\left(\tZ^{(k)} - \tZ^\ddagger_{\mathcal{R}_l(\tV)}\right)\right\|_F^2
        \\&= \|(\tI - \tP_{\tV_{:\nu_k:}})(\tV\tX^{(k-1)} - \tV\tX^\ddagger)\|_F^2 + \left\|\tP_{\tV_{:\nu_k:}}\left(\tZ^{(k)}_{\mathcal{R}_l(\tV)} - \tZ^\ddagger_{\mathcal{R}_l(\tV)}\right) + \tP_{\tV_{:\nu_k:}}\tZ^{(k)}_{\mathcal{R}_l(\tV)^\perp}\right\|_F^2
        \\&= \|(\tI - \tP_{\tV_{:\nu_k:}})(\tV\tX^{(k-1)} - \tV\tX^\ddagger)\|_F^2 + \left\|\tP_{\tV_{:\nu_k:}}\left(\tZ^{(k)}_{\mathcal{R}_l(\tV)} - \tZ^\ddagger_{\mathcal{R}_l(\tV)}\right)\right\|_F^2
        \\&\le \|(\tI - \tP_{\tV_{:\nu_k:}})(\tV\tX^{(k-1)} - \tV\tX^\ddagger)\|_F^2 + \left\|\tZ^{(k)}_{\mathcal{R}_l(\tV)} - \tZ^\ddagger_{\mathcal{R}_l(\tV)}\right\|_F^2
        \\&\le \|(\tI - \tP_{\tV_{:\nu_k:}})(\tV\tX^{(k-1)} - \tV\tX^\ddagger)\|_F^2 + \left\|\tZ^{(k)} - \tZ^\ddagger\right\|_F^2
    \end{align*}
    where the third equation follows from the definition of an orthogonal projector.  

    Taking expectation conditional on the first $k-1$ iterations and applying \eqref{eqn:frobboundpart2} from Lemma~\ref{lem:boundsforFrobeniusnorm}, we have
    \begin{align*}
        \E^{(k-1)}\| \tV\tX^{(k)} - \tV\tX^\ddagger \|^2      
  &\leq  (1- \sigma_{\min}( {\E_{T_\tV}^{(k-1)}[\text{bcirc}( \mathcal{P}_{\tV_{:\nu_k : }})]) }) \| {\tV\tX^{(k-1)}} - \tV\tX^\ddagger\|_F^2 + \E_{T_\tU}^{(k-1)} \|( {\tZ^{\ddagger}} - {\tZ^{(k)}}) \|_F^2
  \\&= \alpha_{\tV} \| {\tV\tX^{(k-1)}} - \tV\tX^\ddagger\|_F^2 + \E_{T_\tU}^{(k-1)} \|( {\tZ^{\ddagger}} - {\tZ^{(k)}}) \|_F^2.
    \end{align*}

    Recalling again that the iterations updating the $\tZ$ iterates are independent of the $\tX$ iterates, we have 
\[\mathbb{E}^{(k-1)}\| \tV\tX^{(k)} - \tV\tX^\ddagger \|^2 \leq \alpha_\tV\|\tV\tX^{(k-1)} - \tV\tX^\ddagger \|^2 + \kappa^2(\tU)\alpha_\tU {\|\tZ^{(k-1)} - \tZ^\ddagger\|^2}.\]
Let  $\alpha_{\max} := \max\{\alpha_\tU, \alpha_\tV\}$ and $\alpha_{\min} := \min\left\{ \frac{\alpha_{\tU}}{\alpha_{\tV}}, \frac{\alpha_{\tU}}{\alpha_{\tV}}\right\}$. When $\alpha_\tU \neq \alpha_{\tV}$, by recursion we get 
\begin{align*}
    \mathbb{E}\| \tV\tX^{(k)} - \tV\tX^\ddagger \|_F^2 &\le \alpha_\tV^k \|\tV\tX^{(0)} - \tV\tX^\ddagger\|_F^2 + \kappa^2(\tU)\left(\sum_{s = 1}^k \alpha_\tU^s \alpha_\tV^{k-s}\right) \|\tZ^{(0)} - \tZ^\ddagger\|_F^2
    \\&\le \alpha_\tV^k \|\tV\tX^\ddagger\|_F^2 + \kappa^2(\tU)\alpha_{\max}^k \sum_{s=1}^{k} \alpha^s_{\min}\|\tZ^\ddagger\|_F^2\\
    &\le \alpha_\tV^k \|\tV\tX^\ddagger\|_F^2 + \kappa^2(\tU)\frac{\alpha_{\max}^k \alpha_{\min}}{1 - \alpha_{\min}}\|\tZ^\ddagger\|_F^2.\end{align*}
When $\alpha_\tU = \alpha_\tV$, we get  
\begin{align*}
    \mathbb{E}\| \tV\tX^{(k)} - \tV\tX^\ddagger \|_F^2 &\le \alpha_\tV^k \|\tV\tX^{(0)} - \tV\tX^\ddagger\|_F^2 + \kappa^2(\tU)\left(\sum_{s = 1}^k \alpha_\tU^s \alpha_\tV^{k-s}\right) \|\tZ^{(0)} - \tZ^\ddagger\|_F^2
    \\&\le \alpha_\tV^k \|\tV\tX^\ddagger\|_F^2 + \kappa^2(\tU)k \alpha_{\max}^k \|\tZ^\ddagger\|_F^2.
\end{align*}

The final result \eqref{eq:factrbgs_inner_error_theorem} follows from Lemma~\ref{lem:residual_to_error}.
\end{proof}

\begin{proof}[Proof of Theorem~\ref{thm:factrbags}]
    The justification for the results in Theorem (\ref{thm:errorsboundsonfactrbags}) is essentially a combination of the proof of (\ref{thm:errorsboundsonfactrbgs}) and the proof of (\ref{thm:trbags}). 
    
    First, observe that \eqref{eq:factrbags_outer_residual_theorem} and \eqref{eq:factrbags_outer_error_theorem} follow directly from Theorem~\ref{thm:trbags} and the fact that the $\tZ^{(k)}$ iterates for the outer system are independent of the $\tX^{(k)}$ iterates for the inner system.

     Now, we have
    \begin{align*}
        \|\tV\tX^{(k)} - \tV\tX^\ddagger\|_F^2 
        &= \|\tV\tX^{(k-1)} - \tV\tX^\ddagger - \omega_2\tV_{:\nu_k:}\tV_{:\nu_k:}^*(\tV \tX^{(k - 1)} - \tZ^{(k)})\|_F^2 \\
        &= \|\tV\tX^{(k-1)} - \tV\tX^\ddagger - \omega_2\tV_{:\nu_k:}\tV_{:\nu_k:}^*(\tV \tX^{(k - 1)} - \tZ^{(k)})\|_F^2 \\
        &= \|(\tI - \omega_2\tV_{:\nu_k:}\tV_{:\nu_k:}^*)(\tV\tX^{(k-1)} - \tV\tX^\ddagger) + \omega_2\tV_{:\nu_k:}\tV_{:\nu_k:}^*(\tZ^{(k)} - \tV\tX^\ddagger) \|_F^2.
        \intertext{But $\tV\tX^{\dagger} = \tZ^{\dagger}_{\mathcal{R}_l(\tV)}$, so we have}
       \|\tV\tX^{(k)} - \tV\tX^\ddagger\|_F^2 &= \|(\tI - \omega_2\tV_{:\nu_k:}\tV_{:\nu_k:}^*)(\tV\tX^{(k-1)} - \tV\tX^\ddagger) 
        + \omega_2\tV_{:\nu_k:}\tV_{:\nu_k:}^*(\tZ^{(k)} - \tZ^{\dagger}_{\mathcal{R}_l(\tV)})\|_F^2 \\
         &= \|(\tI - \omega_2\tV_{:\nu_k:}\tV_{:\nu_k:}^*)(\tV\tX^{(k-1)} - \tV\tX^\ddagger) 
        + \omega_2\tV_{:\nu_k:}\tV_{:\nu_k:}^*(\tZ^{(k)} - \tZ^{\dagger})\|_F^2
        \end{align*}
        since $\tV^*\tZ^{\dagger}_{\mathcal{R}_l(\tV)^\perp}= 0$ by Lemma (\ref{lemma:pseudoinverse}). Then,
        \begin{align*}
            \|\tV\tX^{(k)} - \tV\tX^\ddagger\|_F^2 
            &\leq 2(\|(\tI - \omega_2\tV_{:\nu_k:}\tV_{:\nu_k:}^*)(\tV\tX^{(k-1)} - \tV\tX^\ddagger) \|_F^2 + 
            \| \omega_2\tV_{:\nu_k:}\tV_{:\nu_k:}^*(\tZ^{(k)} - \tZ^{\dagger})\|_F^2). 
        \end{align*}
        The second term of the previous inequality can be bounded using Lemma (\ref{lem:boundsforFrobeniusnorm}), as
        \begin{align*}
        \| \omega_2\tV_{:\nu_k:}\tV_{:\nu_k:}^*(\tZ^{(k)} - \tZ^{\dagger})\|_F^2
        &\leq \omega_2^2 \sigma^2_{\max}(\mathrm{bcirc}(\tV_{:\nu_k:})) \sigma^2_{\max}(\mathrm{bcirc}(\tV_{:\nu_k:}^*)) \| (\tZ^{(k)} - \tZ^{\dagger})\|_F^2, \\
        &\le \omega_2^2 \frac{\gamma_{\tV}}{2} \| (\tZ^{(k)} - \tZ^{\dagger})\|_F^2, \\
        \intertext{while the first term uses an argument similar to one found in the proof of Theorem~\ref{thm:trbags} and results in}
        \|(\tI - \omega_2\tV_{:\nu_k:}\tV_{:\nu_k:}^*)(\tV\tX^{(k-1)} - \tV\tX^\ddagger) \|_F^2
        &\leq \|\tV \tX^{(k-1)} - \tV\tX^\ddagger\|_F^2 - (2\omega_2 - \omega_2^2 \sigma_\tV^2) \|\tV_{: \nu_k :}^* (\tV \tX^{(k-1)} - \tV\tX^\ddagger)\|_F^2.
    \end{align*}
    Then,
    \begin{align*}
        \E^{(k-1)}\|\tV\tX^{(k)} - \tV\tX^\ddagger\|_F^2 
        &\leq 2\|\tV \tX^{(k-1)} - \tV\tX^\ddagger\|_F^2 - 2(2\omega_2 - \omega_2^2 \sigma_\tV^2) \E^{(k-1)}_{T_\tV}[\|\tV_{: \nu_k :}^* (\tV \tX^{(k-1)} - \tV\tX^\ddagger)\|_F^2] \\
        &\quad\quad\quad\quad\quad\quad\quad\quad\quad + \omega_2^2 \gamma_{\tV} \E^{(k-1)}_{T_\tU}[\| \tZ^{(k)} - \tZ^{\dagger})\|_F^2].
    \end{align*}
    Using
    \[
        \E^{(k-1)}_{T_\tV}[\|\tV_{: \nu_k :}^* (\tV \tX^{(k-1)} - \tV\tX^\ddagger)\|_F^2]
        \geq \frac{c_{\min}(T_\tV)}{|T_\tV|} \sigma^2_{\min}(\mathrm{bcirc}(\tV)) \|\tV \tX^{(k-1)} - \tV\tX^\ddagger)\|_F^2,
    \]
    we get
    \begin{align*}
        \E^{(k-1)}\|\tV\tX^{(k)} - \tV\tX^\ddagger\|_F^2 
        & \leq 2\left(1 - (2\omega_2 - \omega_2^2 \sigma_\tV^2)\frac{c_{\min}(T_\tV)}{|T_\tV|} \sigma^2_{\min}(\mathrm{bcirc}(\tV))\right) \| \tV \tX^{(k-1)} - \tV\tX^\ddagger)\|_F^2 \\
        &\quad\quad\quad\quad\quad\quad\quad\quad\quad+ \omega_2^2 \gamma_{\tV} \E^{(k-1)}_{T_\tU}[\| \tZ^{(k)} - \tZ^{\dagger})\|_F^2] \\
        &\le \beta_\tV \| \tV \tX^{(k-1)} - \tV\tX^\ddagger)\|_F^2
     + \omega_2^2 \gamma_{\tV} \kappa^2(\tU) \beta_\tU \| \tZ^{(k-1)} - \tZ^{\dagger}\|_F^2.
    \end{align*}
    Letting $\beta_{\max}:= \max\{\beta_\tU, \beta_\tV\}$ and $\beta_{\min}:= \min\left\{ \frac{\beta_{\tU}}{\beta_{\tV}}, \frac{\beta_{\tU}}{\beta_{\tV}}\right\}$ and using the final recursive argument in the proof of Theorem (\ref{thm:errorsboundsonfactrbgs}) yields \eqref{eq:factrbags_inner_residual_theorem} and, from there, Lemma (\ref{lem:residual_to_error}) is applied to obtain~\eqref{eq:factrabgs_inner_error_theorem}.
\end{proof}

\section{Numerical Experiments}

In this section, we provide numerical experiments that illustrate our theoretical convergence results for
the methods proposed, TRBGS, TRBAGS, FacTRBGS, and FacTRBAGS. We measure the convergence using the relative error and residual errors, respectively
\[
    \frac{\|\tX^{(t)} - \tX^\ddagger\|_F}{\| \tX^\ddagger \|}, \text{ and } \|\tA \tX^{(t)} - \tA\tX^\ddagger\|_F.
\]
Additionally, we explore the effect of different types of block sampling. Finally, we apply of these methods to video and image deblurring and examine their performance.

\subsection{Synthetic Data}\label{subsec:synthetic_experiments}

For each experiment, we generate a synthetic tensor linear system that is either consistent or inconsistent and we vary the tensors' sizes to produce over-determined and under-determined systems. \\ 

\textbf{Consistent system.} To test the TRBGS algorithm~\ref{alg:trbgs} and TRBAGS algorithm~\ref{alg:trbags}, we construct tensors $\tA \in \mathbb{R}^{m \times n \times p}$ with each entry sampled i.i.d. from the standard normal distribution. We generate a solution to our consistent system $\tX_{\text{gen}} \in \mathbb{R}^{n \times l \times p}$ with each entry sampled i.i.d.\ from the standard normal distribution. The measurement tensor for our consistent system $\tB \in \mathbb{R}^{m \times l \times p}$ is then calculated as the product $\tB= \tA \tX_{\text{gen}}$. Our theoretical convergence result, Theorems~\ref{thm:trbgs} and~\ref{thm:trbags} state that Algorithms~\ref{alg:trbgs} and~\ref{alg:trbags} will compute the least-norm solution $\tX^\ddagger = \mathrm{argmin}_{\tX} \Vert \tA \tX - \tB \Vert$ whether the system is consistent or inconsistent. Our numerical experiments illustrated in Figures~\ref{fig:TRBGS_overdetermined_consistent},~\ref{fig:TRBGS_underdetermined_consistent},~\ref{fig:TRBAGS_overdetermined_consistent},~\ref{fig:TRBAGS_underdetermined_consistent}, corroborate this result. \\

\textbf{Inconsistent system.} For this case, we generate all the tensors except for $\tB$ as done in the consistent system case. For TRBGS and TRBAGS, we construct $\tB = \tA \tX_{\text{gen}} + 10^{-4}\tB^\perp$, where $\tB^\perp$ to satisfies $\langle \tB^{\perp}, \tA \tX_{\text{gen}} \rangle = 0$, by generating $\tilde{\tB} \in \mathbb{R}^{m \times l \times p}$ with each entry sampled i.i.d.\ from the standard normal distribution, $\widetilde{\tX} = \tA \widetilde{\tB}$, and $\tB^\perp = \widetilde{\tB} - \tA \widetilde{\tX}.$ According to our theorems~\ref{thm:trbgs} and~\ref{thm:trbags}, Algorithms~\ref{alg:trbgs} and~\ref{alg:trbags} with find the least square norm solution $\tX^\ddagger$. Figures~\ref{fig:TRBGS_overdetermined_inconsistent} and~\ref{fig:TRBAGS_overdetermined_inconsistent} show that our theorems hold.

To test FacTRBGS and FacTRBAGS, i.e., Algorithms~\ref{alg:factrbgs} and~\ref{alg:factrbags}, we only study the results for inconsistent systems. We construct the synthetic tensor linear systems using the process is similar to the one described above. We first generate tensors $\tU \in \mathbb{R}^{m \times m_1 \times p}$ and $\tV \in \mathbb{R}^{m_1 \times n \times p}$ with each entry sampled i.i.d.\ from the standard normal distribution. Then, we compute a solution to our inconsistent system $\tX_{\text{gen}} \in \mathbb{R}^{n \times l \times p}$ with each entry sampled i.i.d. from the standard normal distribution and the measurement tensor is given by $\tA = \tU \tV \tX_{\text{gen}} + 10^{-4}\tY^\perp$ where we generate $\tA^\perp$ to satisfy $\langle \tA^{\perp}, \tU \tV \tX_{\text{gen}} \rangle = 0$. Theorems~\ref{thm:errorsboundsonfactrbgs} and~\ref{thm:errorsboundsonfactrbags} state that when the outer least-squares problem  $\|\tU \tZ - \tY\|_F^2$ has a unique minimizer $\tZ^\ddagger$ and the inner system $\tV \tX = \tZ^\ddagger$ is consistent, then FacTRBGS and FacTRBAGS will converge to the system of minimal norm, $\tX^\ddagger = \tV^\dagger \tZ^\ddagger$. Our numerical experiments support this result, as seen in Figures~\ref{fig:factorized_Incons_U_over_V_under_A_under} through~\ref{fig:factorized_Incons_U_over_V_over_A_over}.

\subsubsection{TRBGS and TRBAGS Suites}\label{subsubsec:TRBGSandTRBAGSSuites} 

We examine the performance of TRBGS (Algorithm~\ref{alg:trbgs}) and TRBAGS (Algorithm~\ref{alg:trbags}) under three different considerations: (i) whether $\tA$ is over-determined or under-determined, (ii) whether the system is consistent or inconsistent, and (iii) different sampling block sizes. The only case we don't consider is the under-determined-inconsistent case. For the TRBGS algorithm, the convergence speed increases significantly as the block size increases. For TRBAGS, although we also achieve convergence, the block size does not significantly affect the speed of convergence in the long run. Although this outcome was unexpected, we suspect that, in the long run, the averaging process behaves similarly to performing individual tensor Gauss-Seidel steps. (However, in Appendix~\ref{app:Effect of Block Sizes}, we show that for TRBAGS, larger block sizes matter when the number of iterations is small and when the tensors are large.) Note that, as expected, the residual error in the over-determined system cases, for both consistent and inconsistent systems, is a scalar times the absolute error, and the iterations converge to the least-norm solution. However, in the undetermined consistent system case, the relative error does not converge while the residual error converges. 

\begin{figure}
    \includegraphics[width=0.475\textwidth]{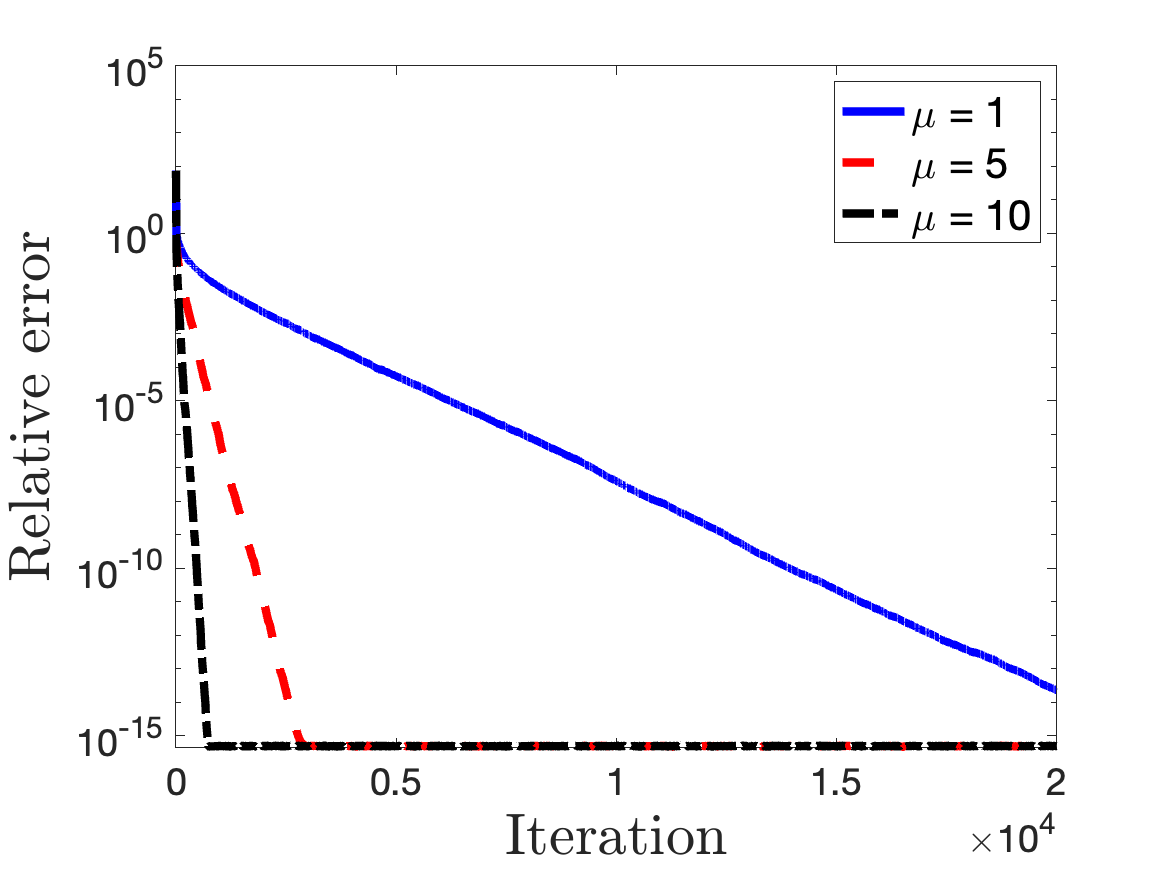}
    \hspace{.25cm}
    \includegraphics[width=0.475\textwidth]
    {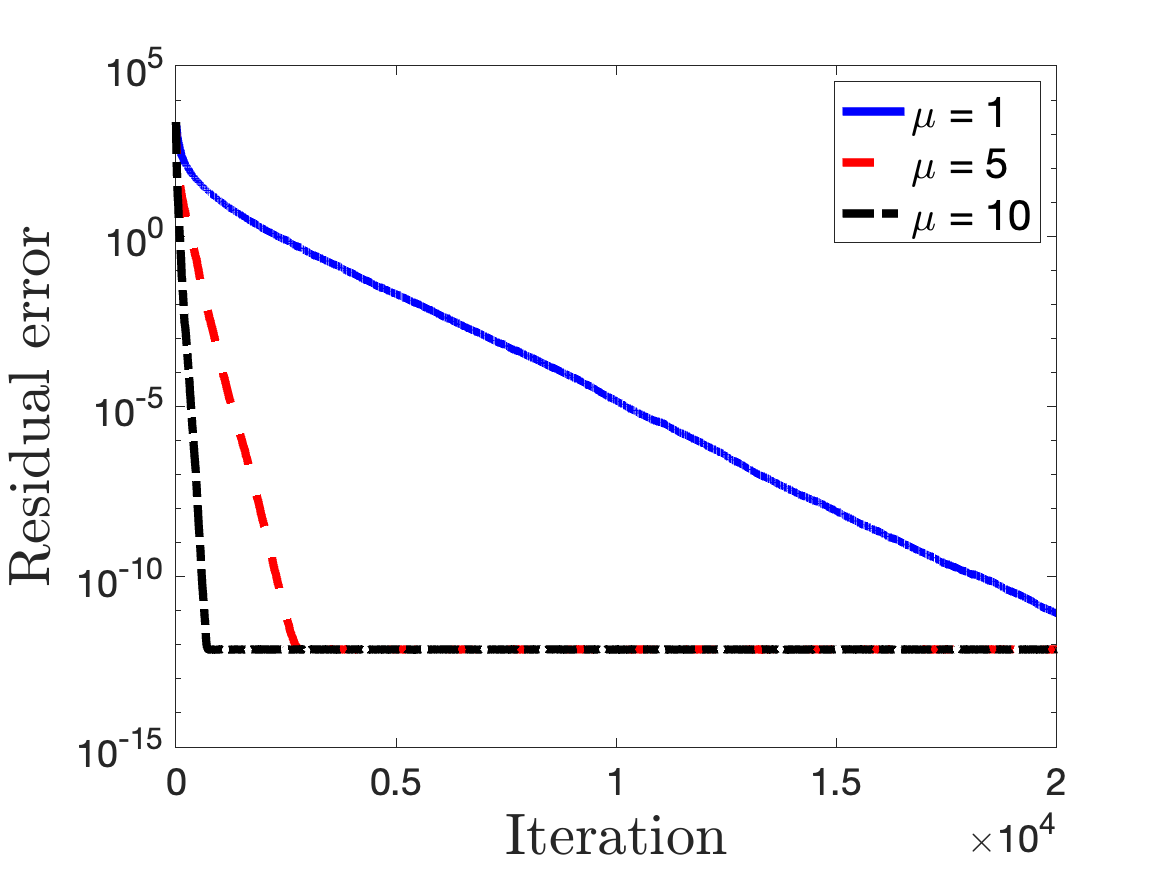}
    \caption{Relative error $\|\tX^{(t)} - \tX^\ddagger\|_F/\|\tX^\ddagger\|_F$ and residual error $\|\tA \tX^{(t)} - \tA\tX^\ddagger\|_F$ versus iteration $t$ of TRBGS on  a consistent linear system when $\tA$ is over-determined. We consider sampling block sizes $|\mu| \in \{1, 5, 10\}$ in each case.}
    \label{fig:TRBGS_overdetermined_consistent}
    \includegraphics[width=0.475\textwidth]{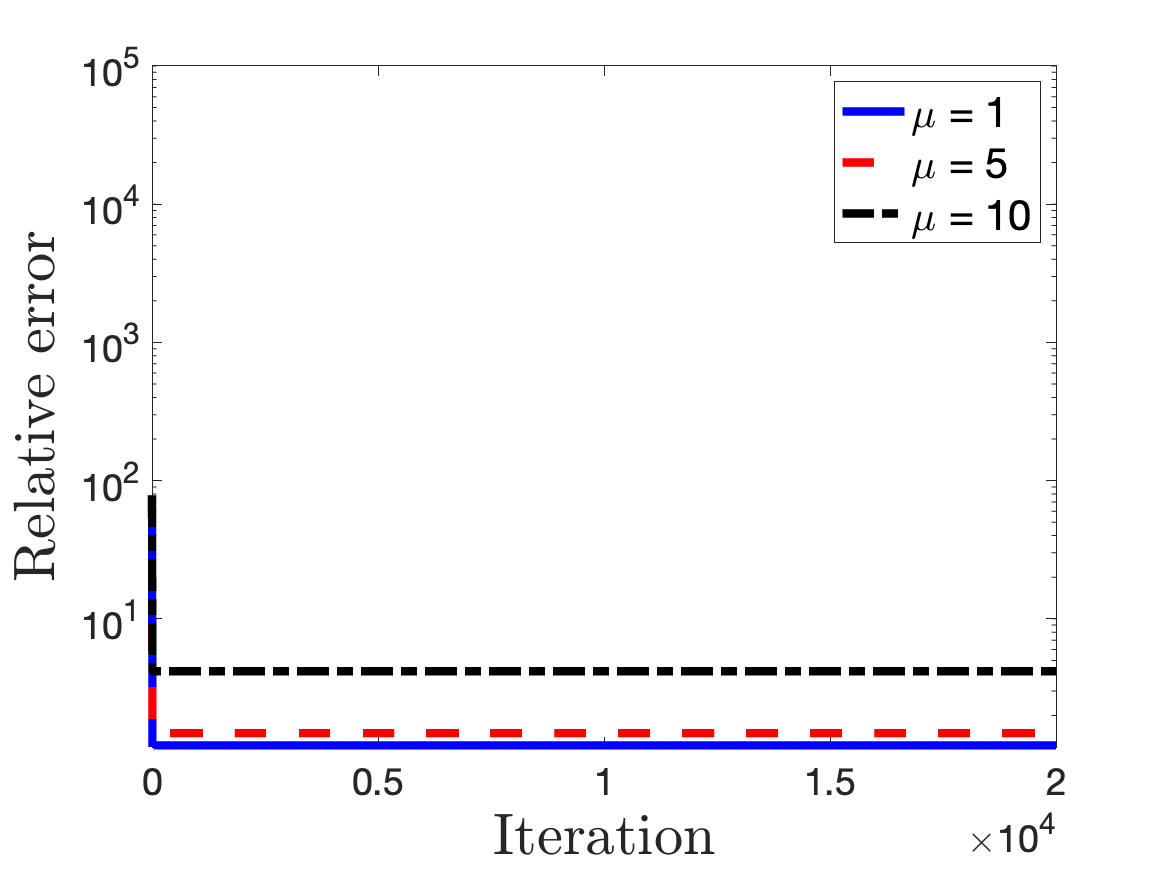}
    \hspace{.25cm}
    \includegraphics[width=0.475\textwidth]{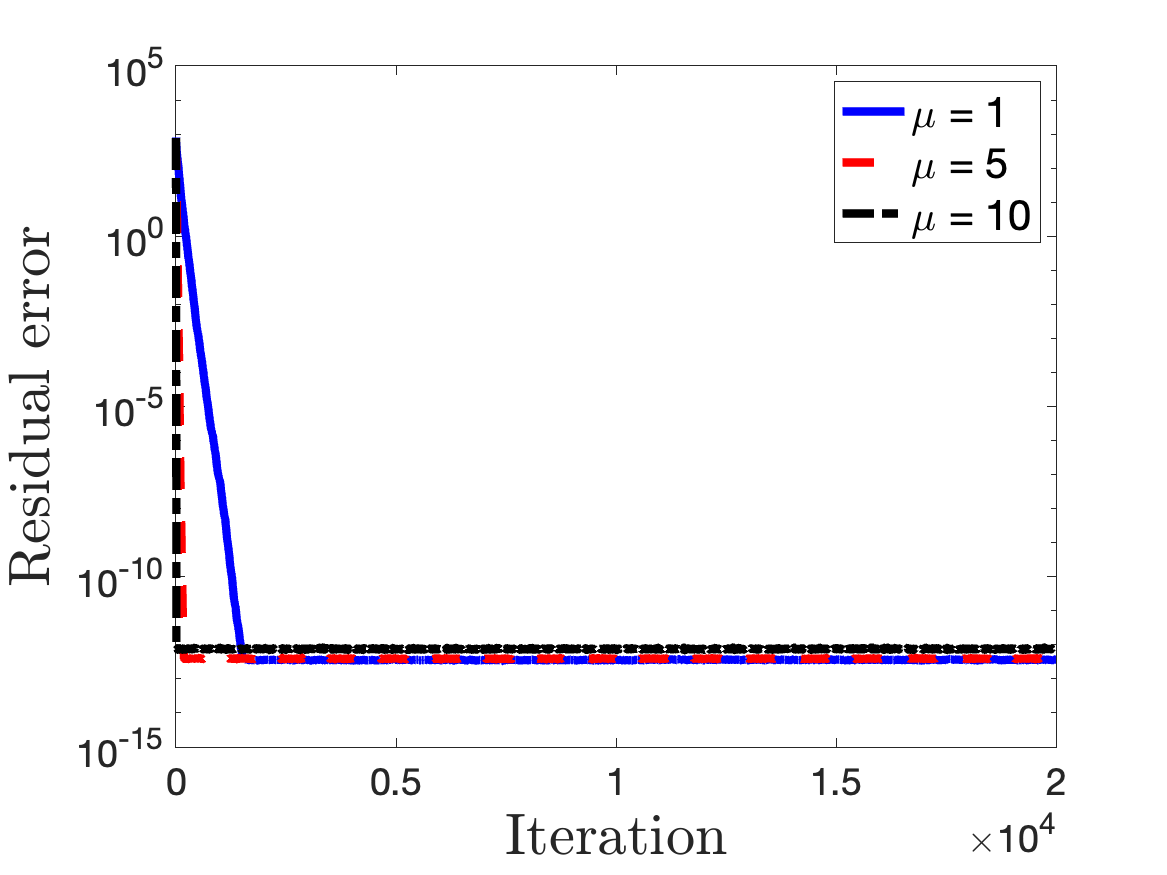}
    \caption{Relative error $\|\tX^{(t)} - \tX^\ddagger\|_F/\|\tX^\ddagger\|_F$ and residual error $\|\tA \tX^{(t)} - \tA\tX^\ddagger\|_F$ versus iteration $t$ of TRBGS on a consistent linear system when $\tA$ is under-determined. We consider sampling block sizes $|\mu| \in \{1, 5, 10\}$ in each case.}
    \label{fig:TRBGS_underdetermined_consistent}
    \includegraphics[width=0.475\textwidth]{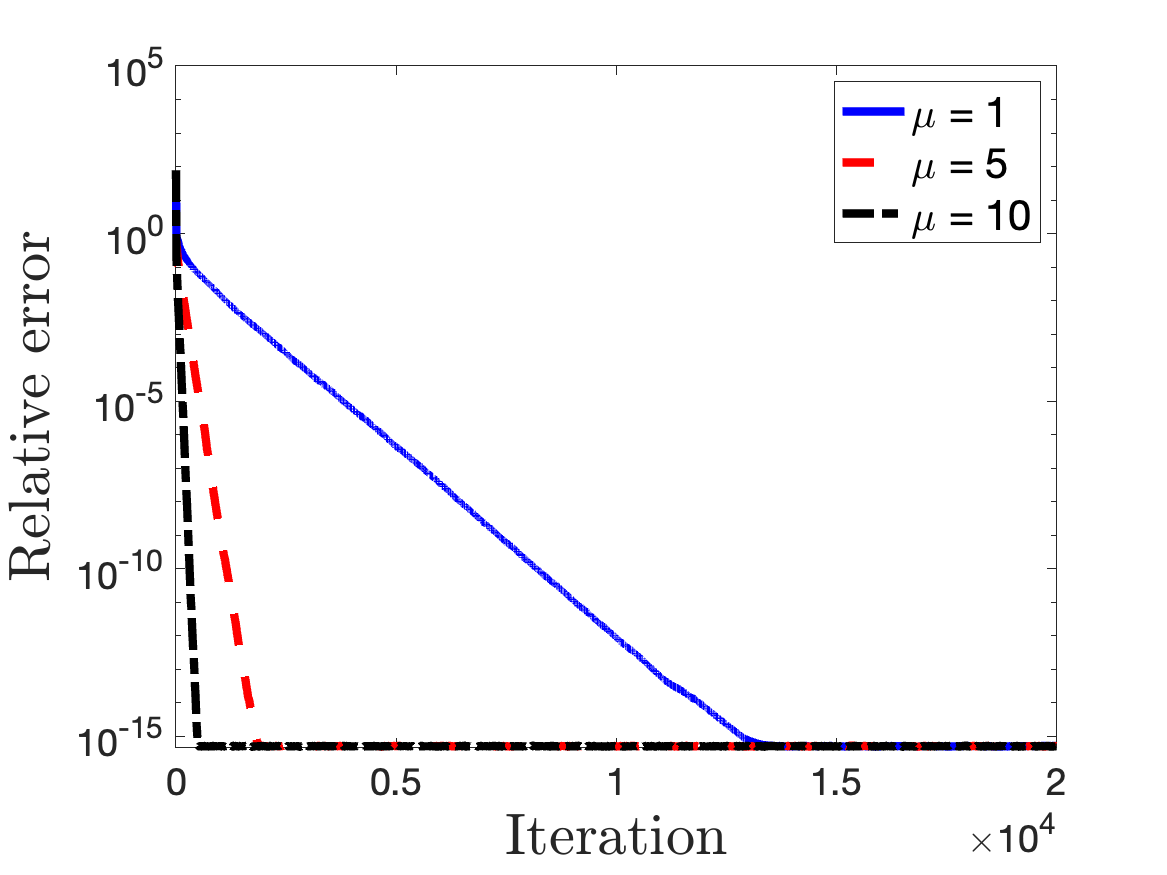}
    \hspace{.25cm}
    \includegraphics[width=0.475\textwidth]
    {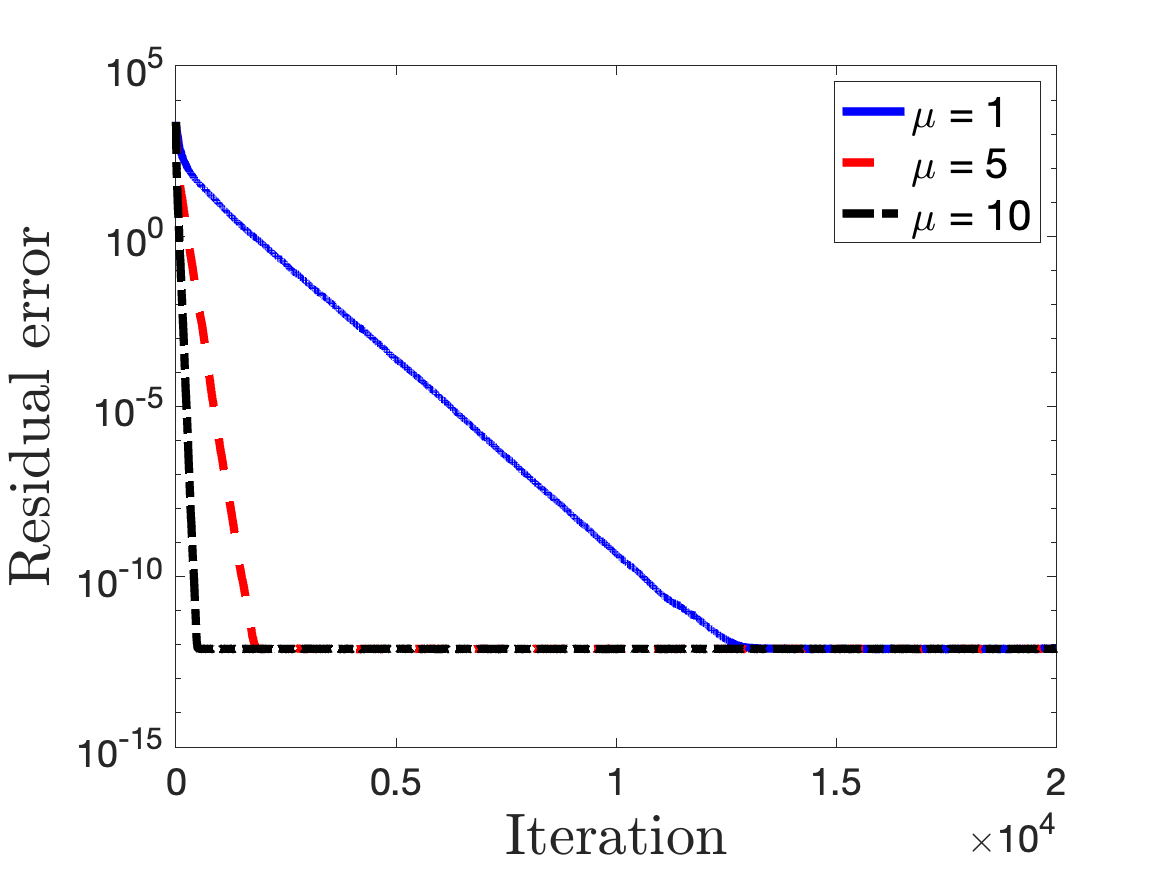}
    
    \caption{Relative error $\|\tX^{(t)} - \tX^\ddagger\|_F/\|\tX^\ddagger\|_F$ and residual error $\|\tA \tX^{(t)} - \tA\tX^\ddagger\|_F$ versus iteration $t$ of TRBGS on an inconsistent linear system when $\tA$ is over-determined. We consider sampling block sizes $|\mu| \in \{1, 5, 10\}$ in each case.}
    \label{fig:TRBGS_overdetermined_inconsistent}
\end{figure}

\begin{figure}
    \includegraphics[width=0.475\textwidth]{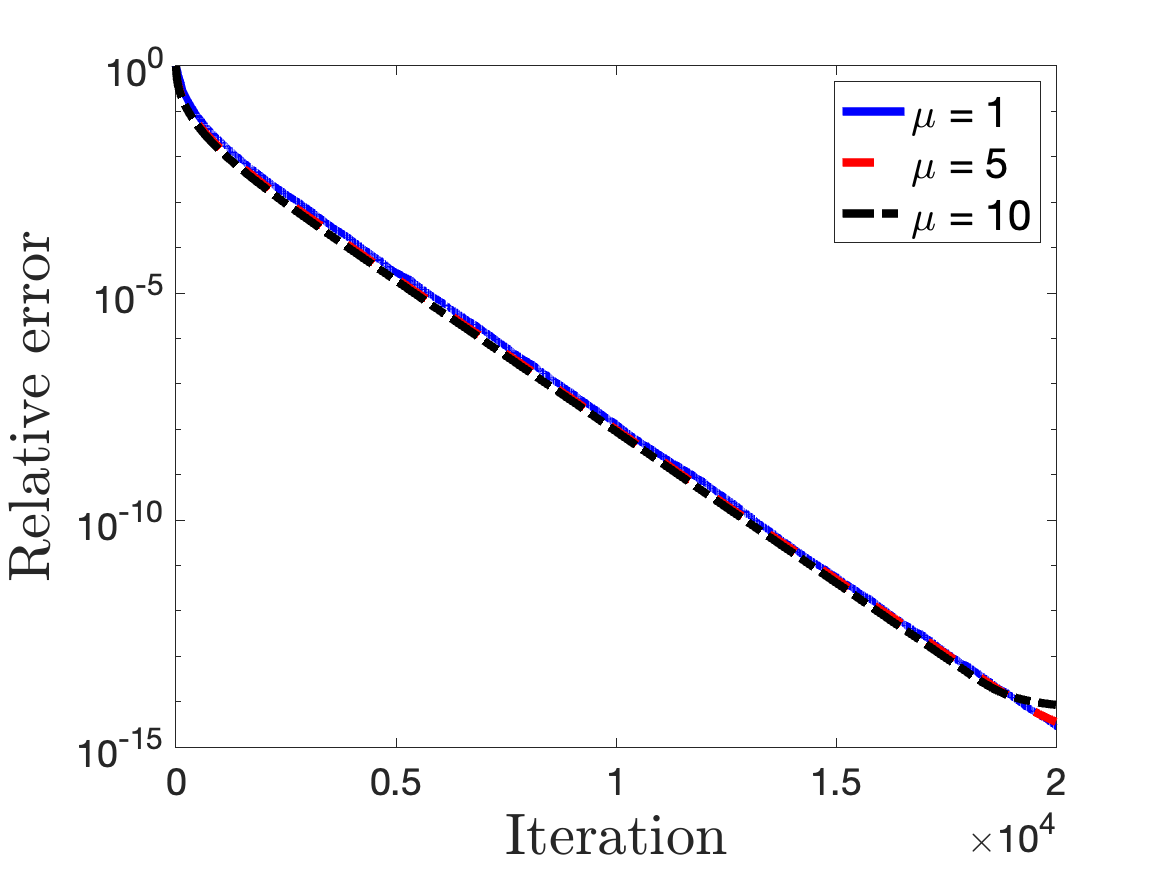}
    \hspace{.25cm}
    \includegraphics[width=0.475\textwidth]
    {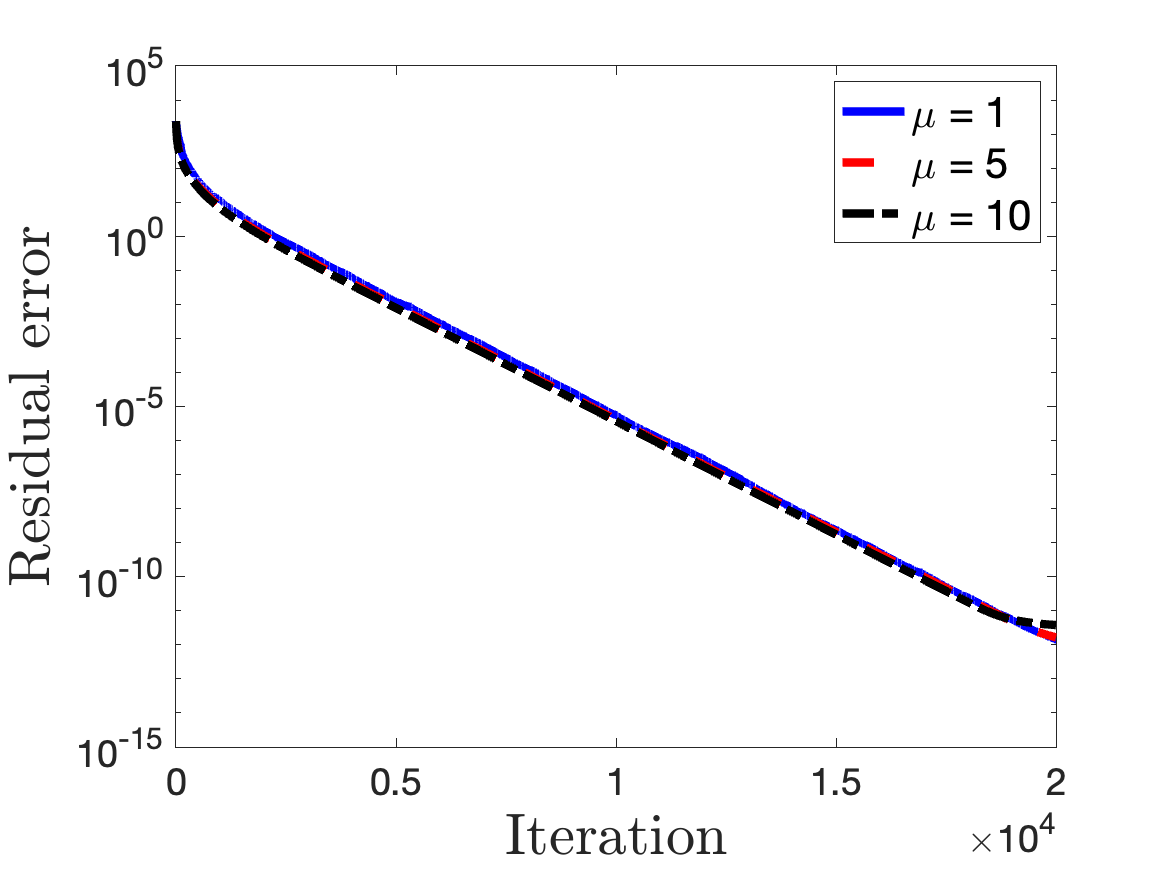}    
    \caption{Relative error $\|\tX^{(t)} - \tX^\ddagger\|_F/\|\tX^\ddagger\|_F$ and residual error $\|\tA \tX^{(t)} - \tA\tX^\ddagger\|_F$ versus iteration $t$ of TRBAGS on  a consistent linear system when $\tA$ is over-determined. We consider sampling block sizes $|\mu| \in \{1, 5, 10\}$ in each case.}
    \label{fig:TRBAGS_overdetermined_consistent}
    \includegraphics[width=0.475\textwidth]{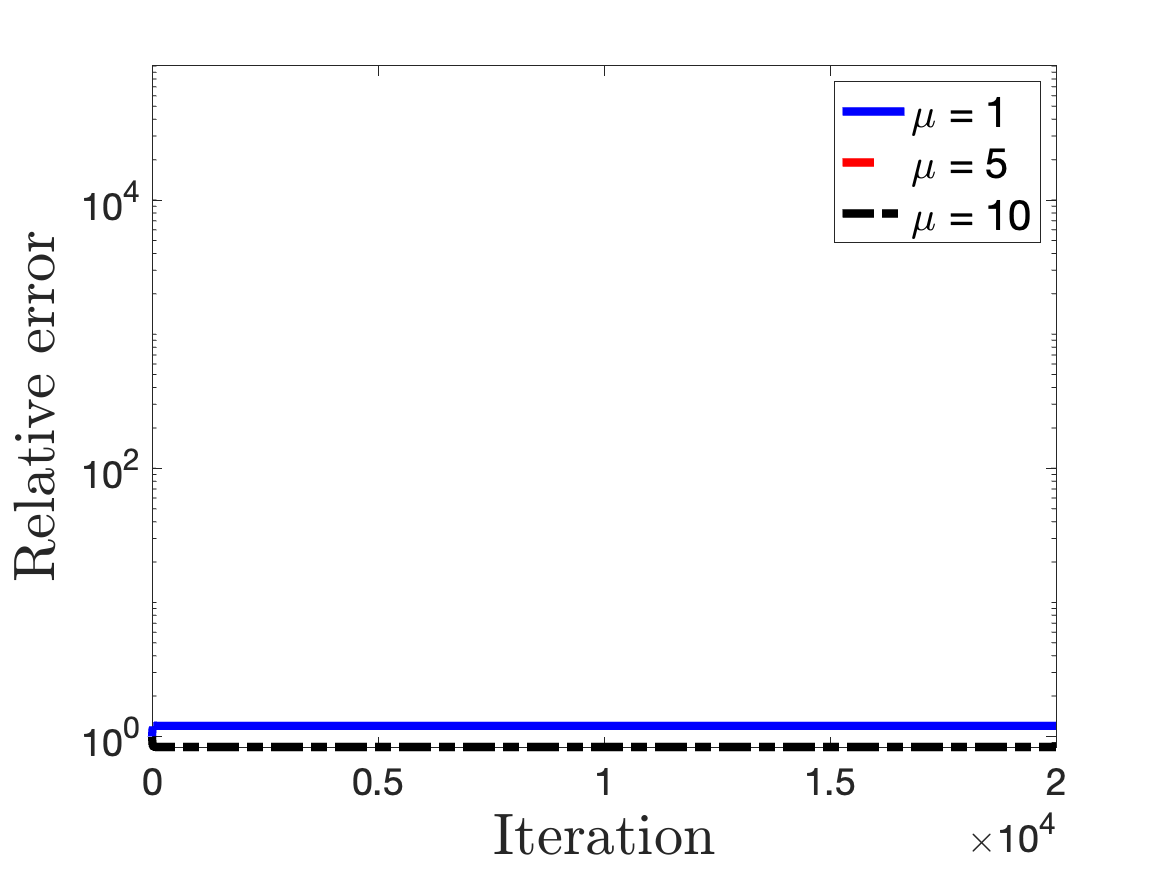}
    \hspace{.25cm}
    \includegraphics[width=0.475\textwidth]{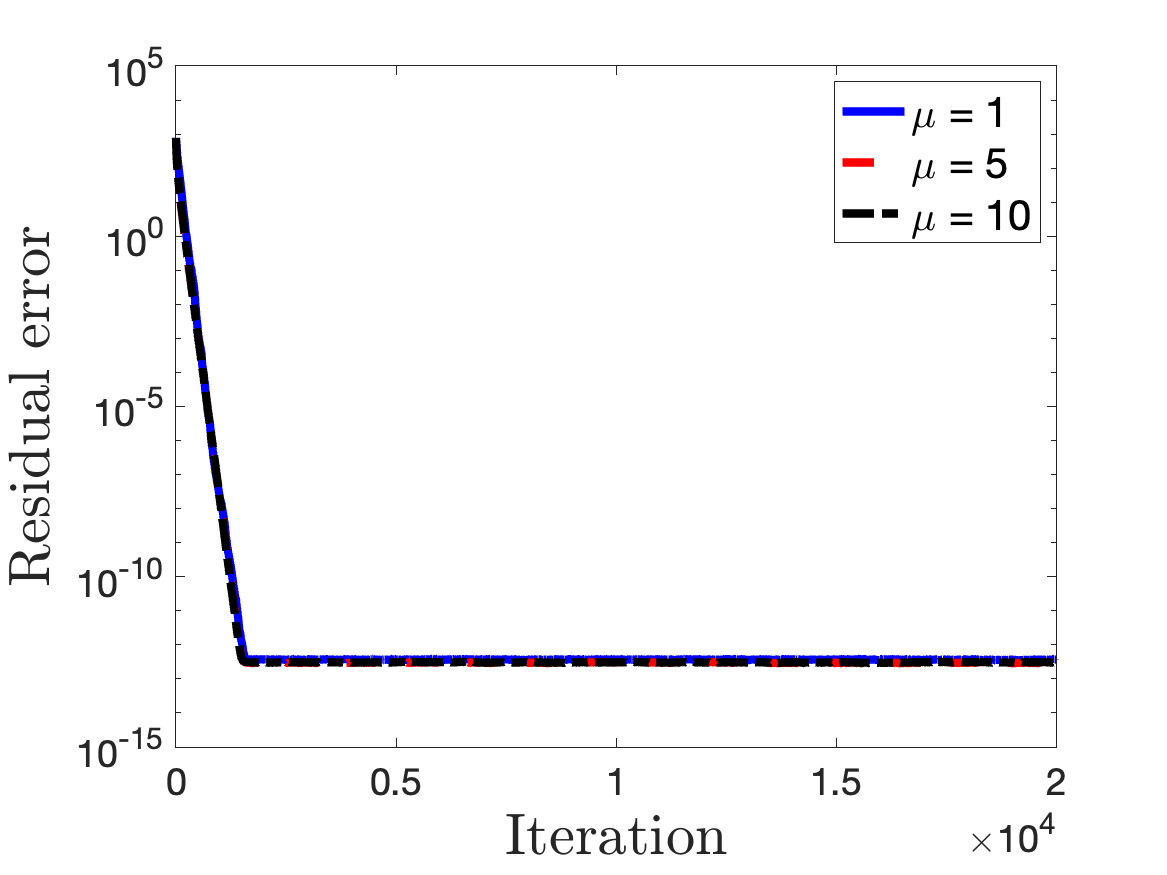}
    \caption{Relative error $\|\tX^{(t)} - \tX^\ddagger\|_F/\|\tX^\ddagger\|_F$ and residual error $\|\tA \tX^{(t)} - \tA\tX^\ddagger\|_F$ versus iteration $t$ of TRBAGS on a consistent linear system when $\tA$ is under-determined. We consider sampling block sizes $|\mu| \in \{1, 5, 10\}$ in each case.}
    \label{fig:TRBAGS_underdetermined_consistent}
    \includegraphics[width=0.475\textwidth]{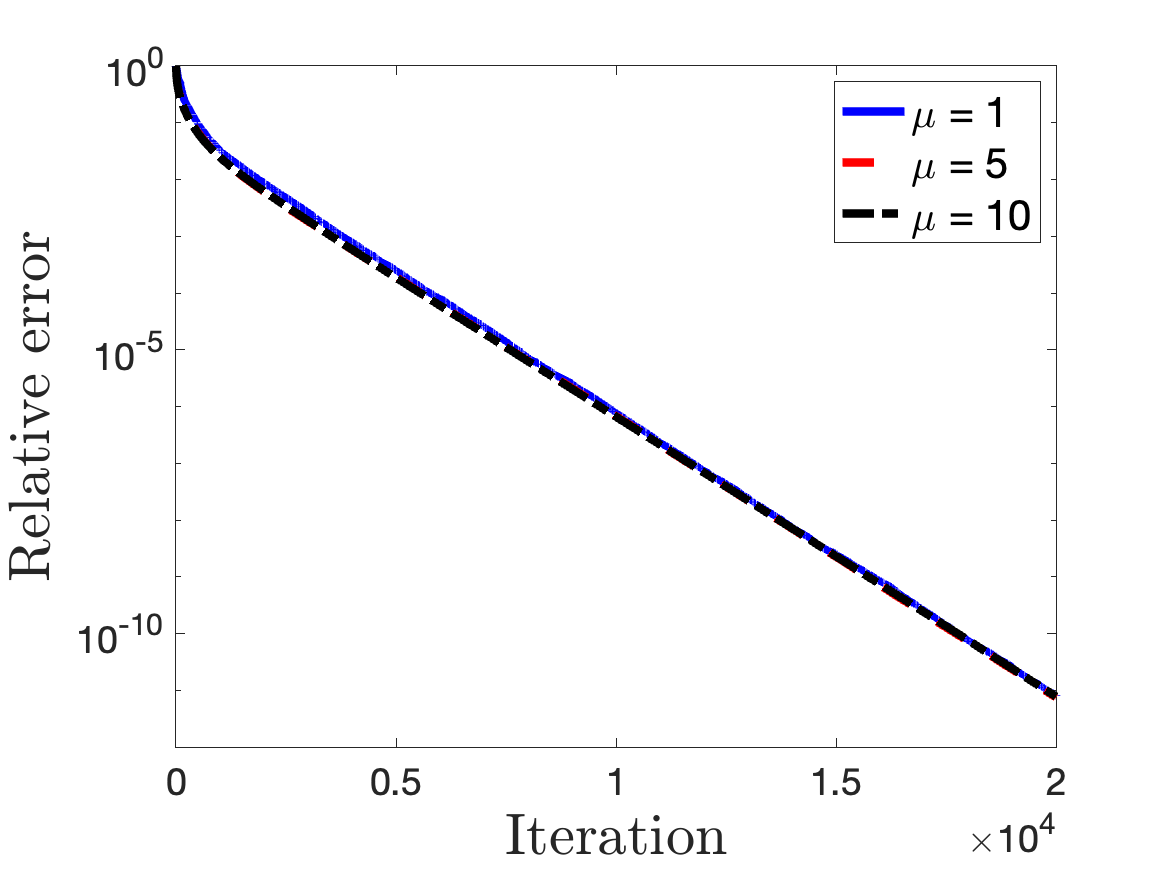}
    \hspace{.25cm}
    \includegraphics[width=0.475\textwidth]
    {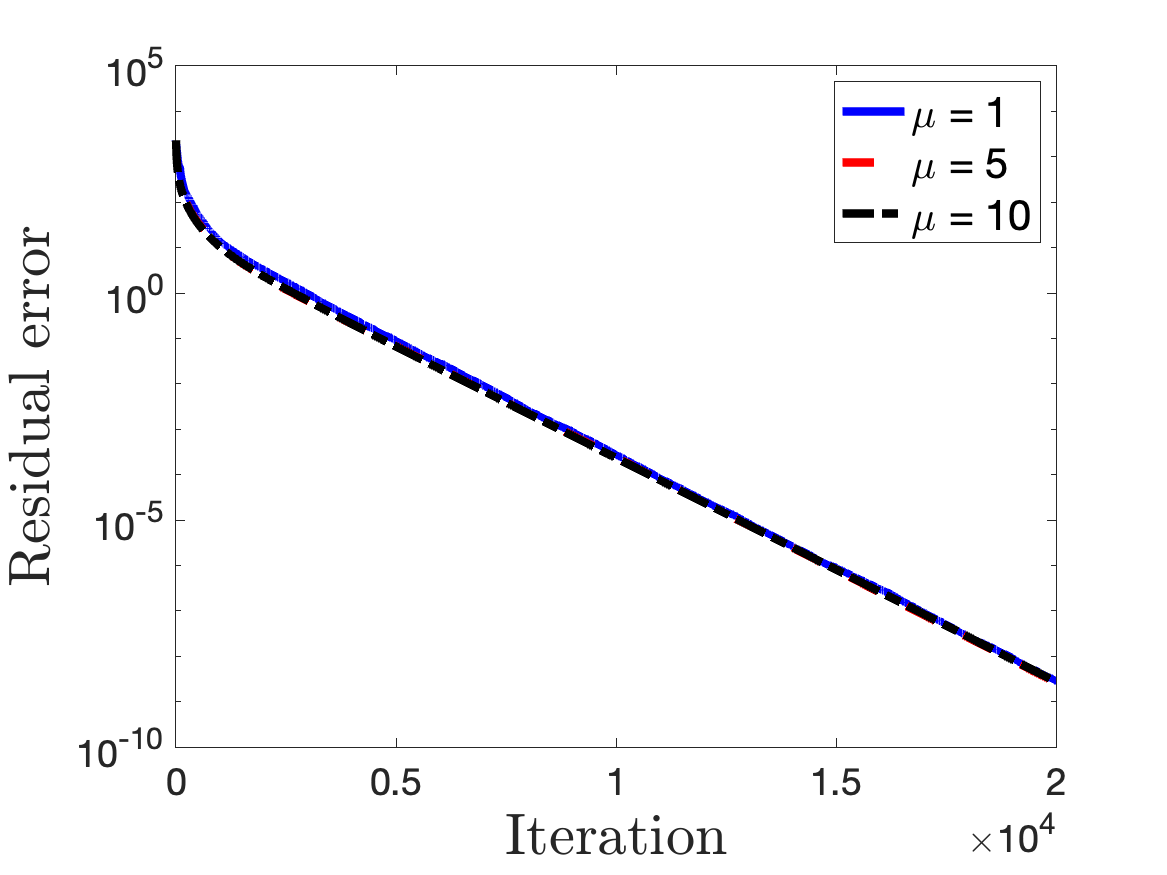}
    \caption{Relative error $\|\tX^{(t)} - \tX^\ddagger\|_F/\|\tX^\ddagger\|_F$ and residual error $\|\tA \tX^{(t)} - \tA\tX^\ddagger\|_F$ versus iteration $t$ of TRBAGS on an inconsistent linear system when $\tA$ is over-determined. We consider sampling block sizes $|\mu| \in \{1, 5, 10\}$ in each case.}
    \label{fig:TRBAGS_overdetermined_inconsistent}
\end{figure}

\subsubsection{Comparison of TRBGS and TRBAGS}

In this section, we compare the performance TRBGS and TRBAGS when the system is over-determined and consistent, over-determined and inconsistent, and under-determined and consistent. Based on the previous section, we use the block size constant at $\mu = 5$ as it yielded visible convergence in a reasonable amount of iterations throughout. Figures~\ref{fig:TRBGS_vs_TRBAGS_overdetermined_consistent},~\ref{fig:TRBGS_vs_TRBAGS_underdetermined_consistent}, and~\ref{fig:TRBGS_vs_TRBAGS_overdetermined_inconsistent} illustrates our results. In most cases when convergence is expected, both methods converge as stated by our theorem and as demonstrated in the previous section. However, for the chosen step-size, TRBGS converges significantly faster than TRBAGS, which we can expect since Algorithm~\ref{alg:trbgs} computes the pseudoinverse. However, if computational costs or memory are an issue, TRBAGS--perhaps with a more precisely defined $\omega$--could be an effective way to obtain our results while avoiding the computation of the pseudoinverse. 

\begin{figure}
    \includegraphics[width=0.475\textwidth]{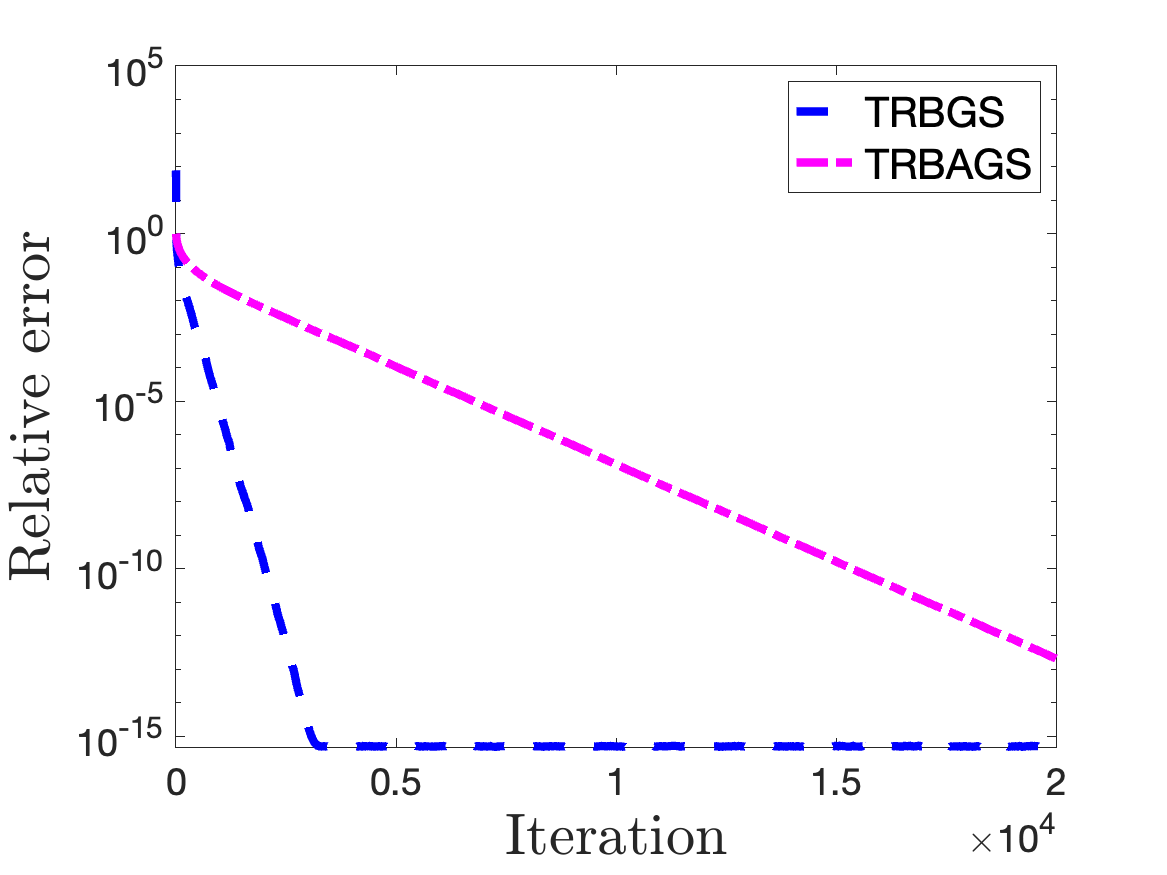}\hspace{.25cm}
    \includegraphics[width=0.475\textwidth]
    {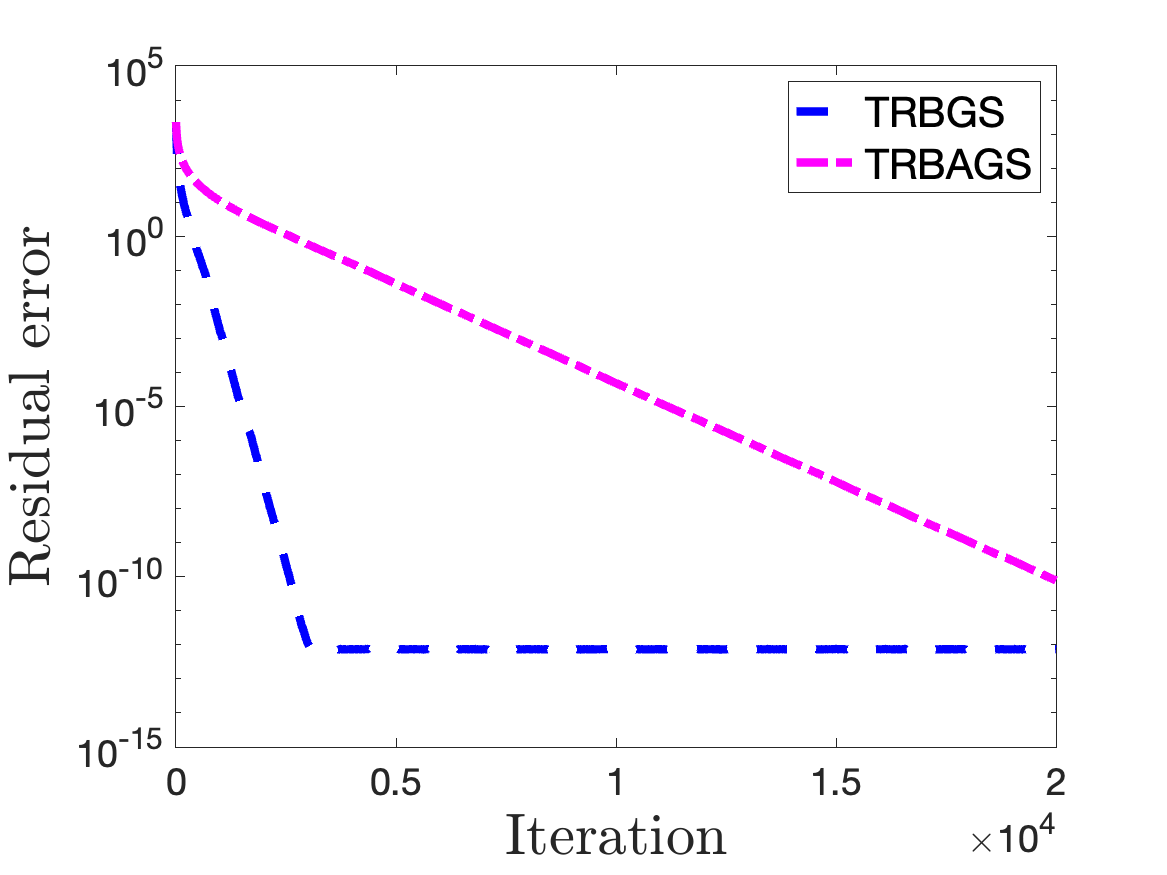}

    \caption{Relative error $\|\tX^{(t)} - \tX^\ddagger\|_F/\|\tX^\ddagger\|_F$ and residual error $\|\tA \tX^{(t)} - \tA\tX^\ddagger\|_F$ versus iteration $t$ for TRBGS and TRBAGS on  a consistent linear system when $\tA$ is over-determined. The block size is $|\mu| = 5$.}
    \label{fig:TRBGS_vs_TRBAGS_overdetermined_consistent}
     \includegraphics[width=0.475\textwidth]{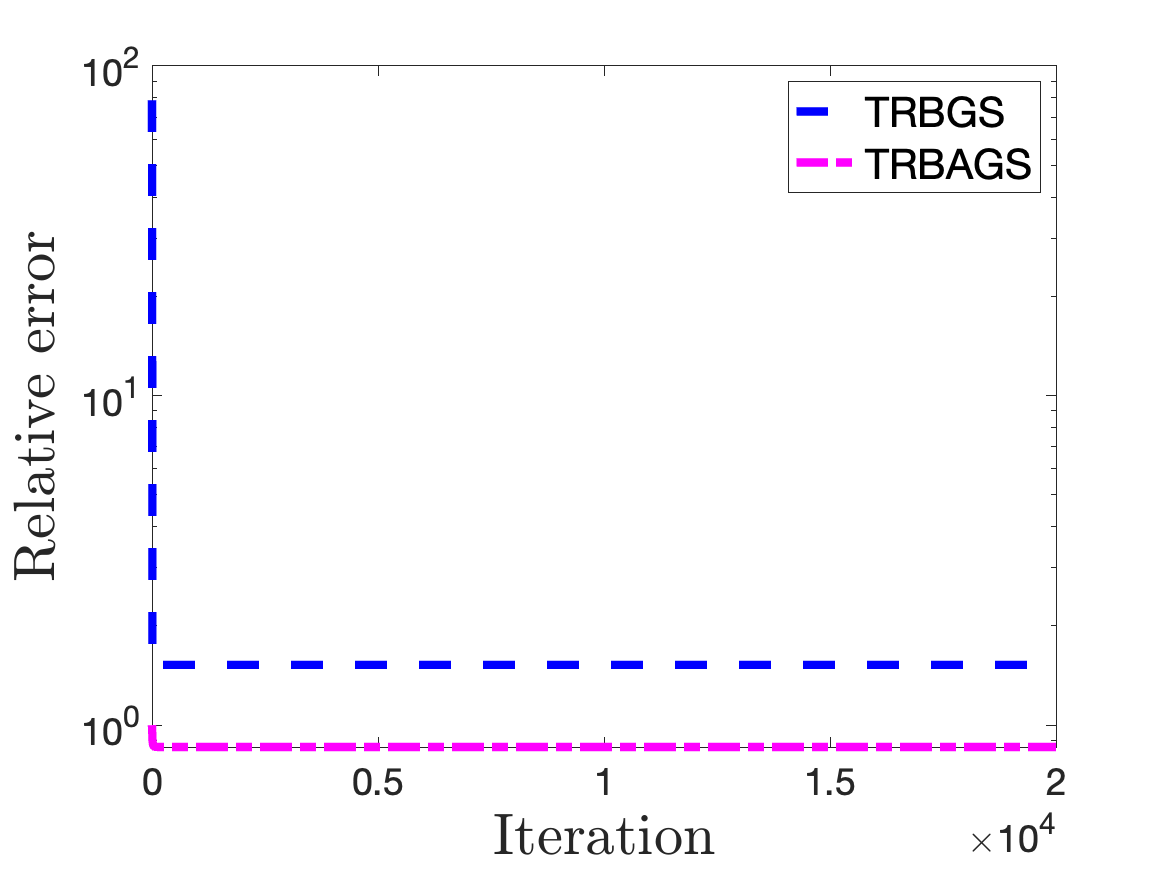}\hspace{.25cm}
    \includegraphics[width=0.475\textwidth]
    {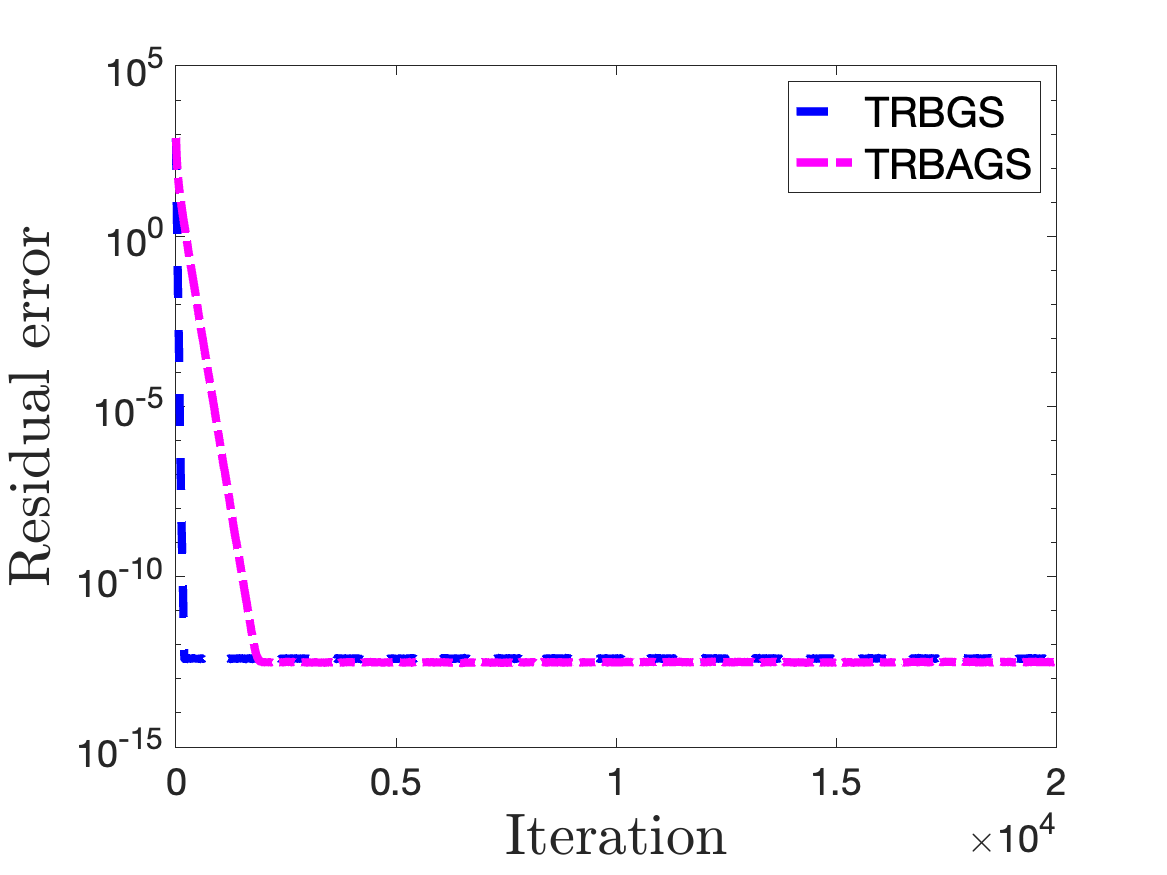}
    
    \caption{Relative error $\|\tX^{(t)} - \tX^\ddagger\|_F/\|\tX^\ddagger\|_F$ and residual error $\|\tA \tX^{(t)} - \tA\tX^\ddagger\|_F$ versus iteration $t$ for TRBGS and TRBAGS on  a consistent linear system when $\tA$ is under-determined. The block size is $|\mu| = 5$.}
    \label{fig:TRBGS_vs_TRBAGS_underdetermined_consistent}
    \includegraphics[width=0.475\textwidth]{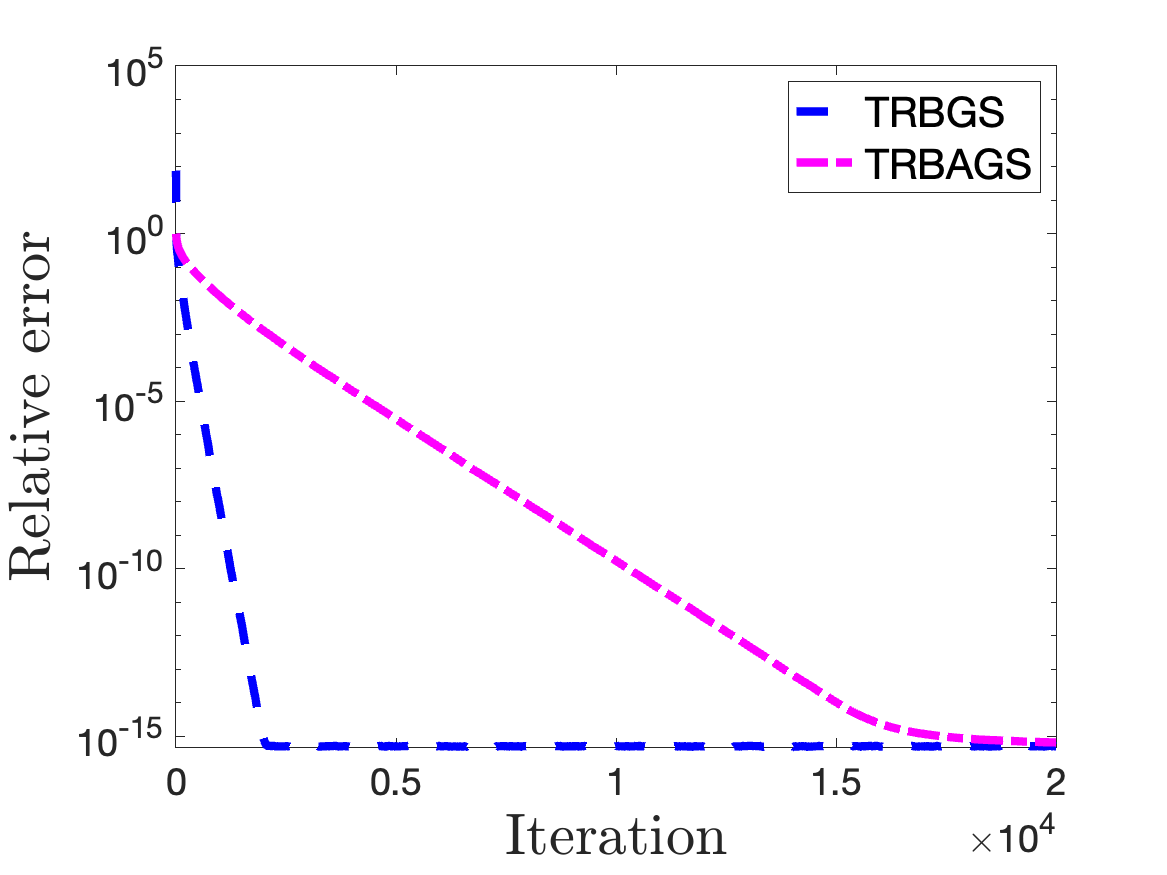}\hspace{.25cm}
    \includegraphics[width=0.475\textwidth]{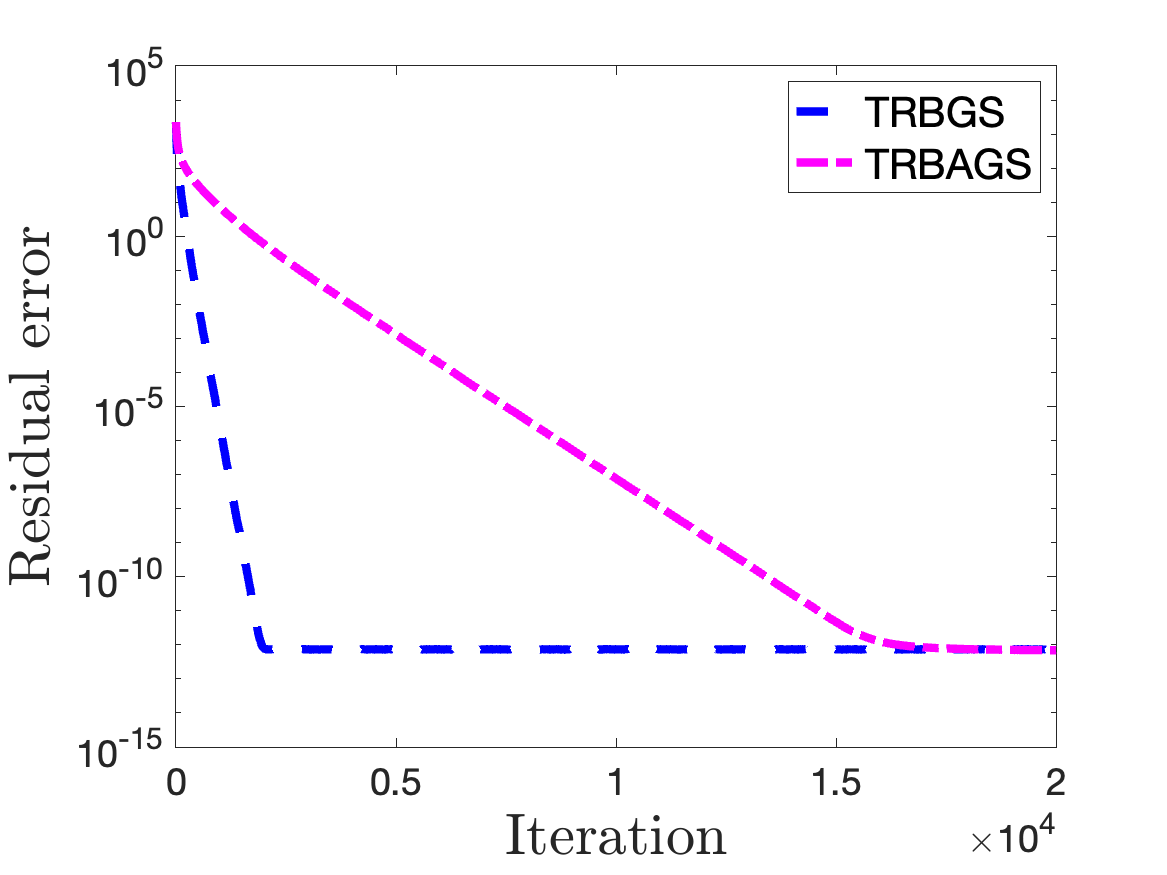}
    
    \caption{Relative error $\|\tX^{(t)} - \tX^\ddagger\|_F/\|\tX^\ddagger\|_F$ and residual error $\|\tA \tX^{(t)} - \tA\tX^\ddagger\|_F$ versus iteration $t$ for TRBGS and TRBAGS on  a consistent linear system when $\tA$ is over-determined. The block size is $|\mu| = 5$.}
    \label{fig:TRBGS_vs_TRBAGS_overdetermined_inconsistent}
\end{figure}

\subsubsection{Comparison FacTRBGS and FacTRBAGS}\label{subsec:FacTRBGSvsFacTRBAGS}

In the following examples, we investigate the performance of FacTRBGS (Algorithm~\ref{alg:factrbgs}) and FacTRBAGS (Algorithm~\ref{alg:factrbags}) for inconsistent systems by analyzing the error convergence for the outer systems $\tU \tZ = \tB$ and inner systems $\tV \tX = \tZ$. That is, we plot all the errors described in Theorem~\ref{thm:errorsboundsonfactrbgs} and~\ref{thm:errorsboundsonfactrbags}. Observe that the error for the inner system is also that of our system of interest $\tA \tX = \tB$. Also, for our algorithms, it does not matter whether the system is inconsistent or consistent. 

We run simulations for $\tA$ under-determined and over-determined by considering various combinations of under-determined and over-determined for $\tU$ and $\tV$. However, we keep both the inner and outer sampling block sizes constant at $\mu = 5$. Note that this number can be changed. We just wanted it fixed in all experiments because the previous experiments with TRBGS and TRBAGS showed that any size--appropriate for the tensor dimensions--should work and we chose $\mu = 5$ because it yielded good results.

\begin{table}[h]
\caption{Cases of the numerical experiments of FacTRBGS (Algorithm~\ref{alg:factrbgs}) and FacTRBAGS  (Algorithm~\ref{alg:factrbags}) algorithms on various sizes of tensors $\tA$, $\tU$ and $\tV$. The dimensions of $\tens{X}$ were $20\times 10\times 30$, the same for all the cases. Cases for which Theorem~\ref{thm:errorsboundsonfactrbgs} and Theorem~\ref{thm:errorsboundsonfactrbags} do not hold are grayed out and these experiments are included in Appendix~\ref{app:Comparison FacTRBGS and FacTRBAGS (Cases in gray cells)}. Cases for which the tensors cannot be formed are blacked out.} \label{tab:factorizedcases}
\small
    \centering
   \begin{tabular}{|c|c|c|c|c|}\hline
   Cases  &\begin{tabular}{@{}c@{}} $\tU$ over-determined \\ $\tV$ under-determined \end{tabular} & \begin{tabular}{@{}c@{}} $\tU$ over-determined \\ $\tV$ over-determined \end{tabular} & \cellcolor{black!25}\begin{tabular}{@{}c@{}} $\tU$ under-determined \\ $\tV$ over-determined \end{tabular} & \cellcolor{black!25}\begin{tabular}{@{}c@{}} $\tU$ under-determined \\ $\tV$ under-determined \end{tabular} \\\hline
 
1. \begin{tabular}{@{}c@{}} $\tA \in \mathbb{R}^{10\times 20 \times 30}$ \\ under-determined \end{tabular}& 
\begin{tabular}{@{}c@{}}$m_1=5$\\ Figure~\ref{fig:factorized_Incons_U_over_V_under_A_under} \end{tabular} & \cellcolor{black}- &
\cellcolor{black!10}\begin{tabular}{@{}c@{}}$m_1=25$\\ Figure~\ref{fig:factorized_Incons_U_under_V_over_A_under} \end{tabular} &
\cellcolor{black!10}\begin{tabular}{@{}c@{}}$m_1=15$\\ Figure~\ref{fig:factorized_Incons_U_under_V_under_A_under}
\end{tabular} 
\\\hline 

2. \begin{tabular}{@{}c@{}} $\tA \in \mathbb{R}^{30\times 20 \times 30}$ \\ over-determined \end{tabular}& 
\begin{tabular}{@{}c@{}}$m_1=15$ \\            Figure~\ref{fig:factorized_Incons_U_over_V_under_A_over} \end{tabular} & 
\begin{tabular}{@{}c@{}}$m_1=25$\\ Figure~\ref{fig:factorized_Incons_U_over_V_over_A_over} \end{tabular} &
\cellcolor{black!10}\begin{tabular}{@{}c@{}}$m_1=35$\\ Figure~\ref{fig:factorized_Incons_U_under_V_over_A_over} \end{tabular}
& \cellcolor{black}-  
\\\hline 

\end{tabular}
\end{table}

Table~\ref{tab:factorizedcases} summarizes the properties of systems considered and indicates the figures corresponding to each case. Note that some cells in the table are left empty and color-coded black because no combination of the tensors $\tU$ and $\tV$ with the indicated property in the first row can produce a tensor $\tA$ with the desired property in the first column. Also, note that the shaded gray cells indicate cases where our theory shows that convergence is not guaranteed.

Figures~\ref{fig:factorized_Incons_U_over_V_under_A_under},~\ref{fig:factorized_Incons_U_over_V_under_A_over}, and~\ref{fig:factorized_Incons_U_over_V_over_A_over} show that our algorithms produce the expected convergence results. We can see that for all the cases covered by our theoretical result, Theorem~\ref{thm:errorsboundsonfactrbgs} and Theorem~\ref{thm:errorsboundsonfactrbags}, the inner systems ($\tV \tX = \tZ$) provide the desired least-norm solution $\tX$. More notable is what happens to the outer system $\tU \tZ = \tB$. When $\tV$ is under-determined, the residual error converges to 0, i.e., the algorithm finds a least-norm solution, however, the relative error does not converge. This implies that approximations of the least-norm (in particular, one that is {\it not} exact) solution to the outer system are sufficient as an input to ensure the convergence of the inner system. 

\begin{figure}[h!]
    \includegraphics[width=0.44\textwidth]{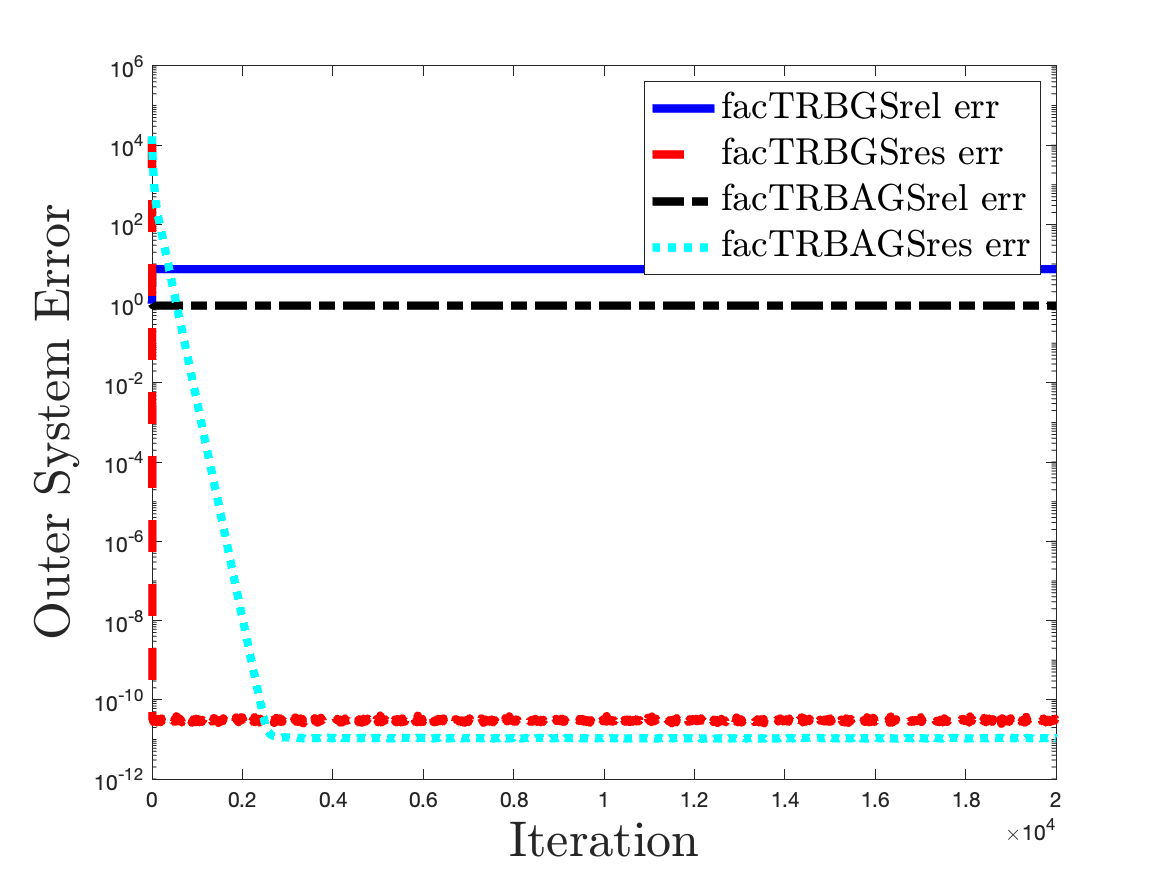}\hspace{.25cm}
    \includegraphics[width=0.44\textwidth]{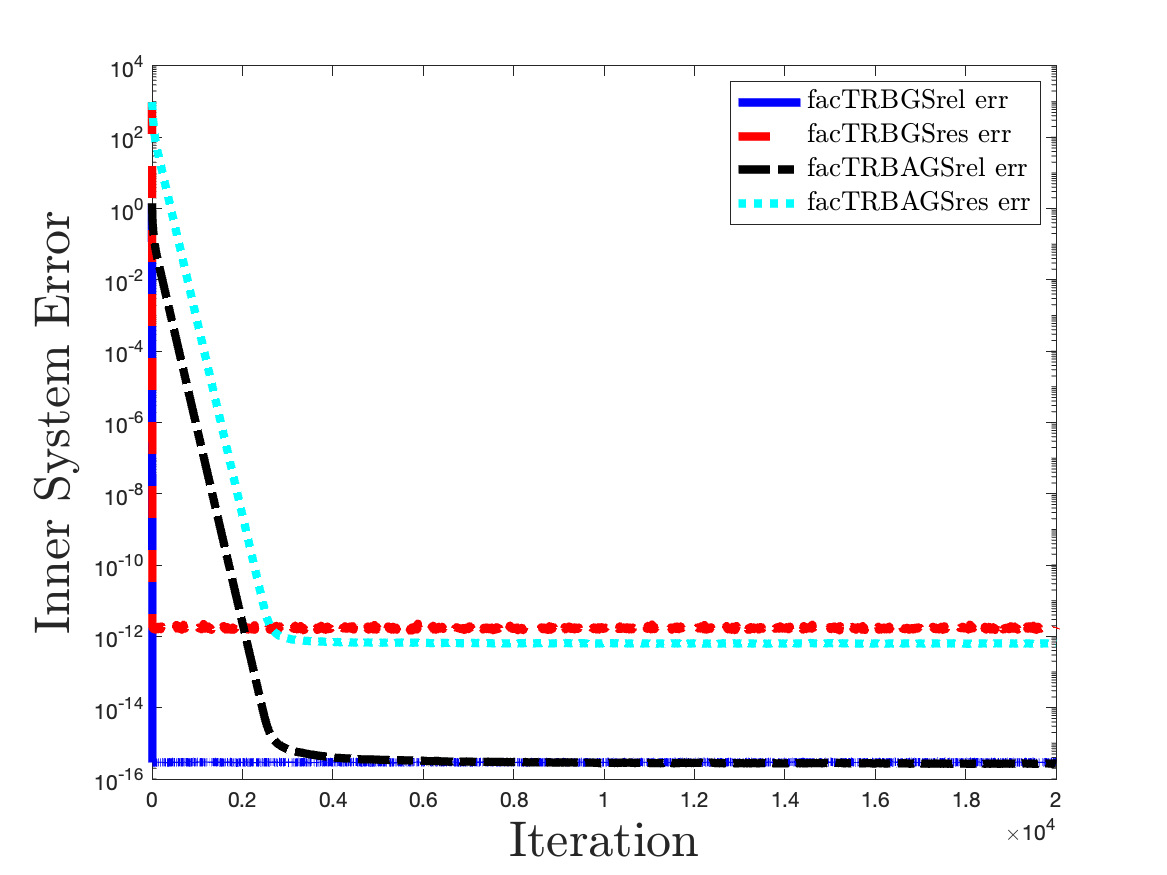}
    
    \caption{Relative error (rel err) and residual error (res err) versus iteration $t$ for the interlaced outer system $\tU \tZ = \tB$ and inner system $\tV \tX = \tZ$. Here, FacTRBGS and FacTRBAGS are applied to an inconsistent tensor system where $\tA$ is under-determined, $\tU$ is over-determined, and $\tV$ is under-determined. The block size is kept constant at $|\mu| = 5$.}
    \label{fig:factorized_Incons_U_over_V_under_A_under}
    \includegraphics[width=0.44\textwidth]{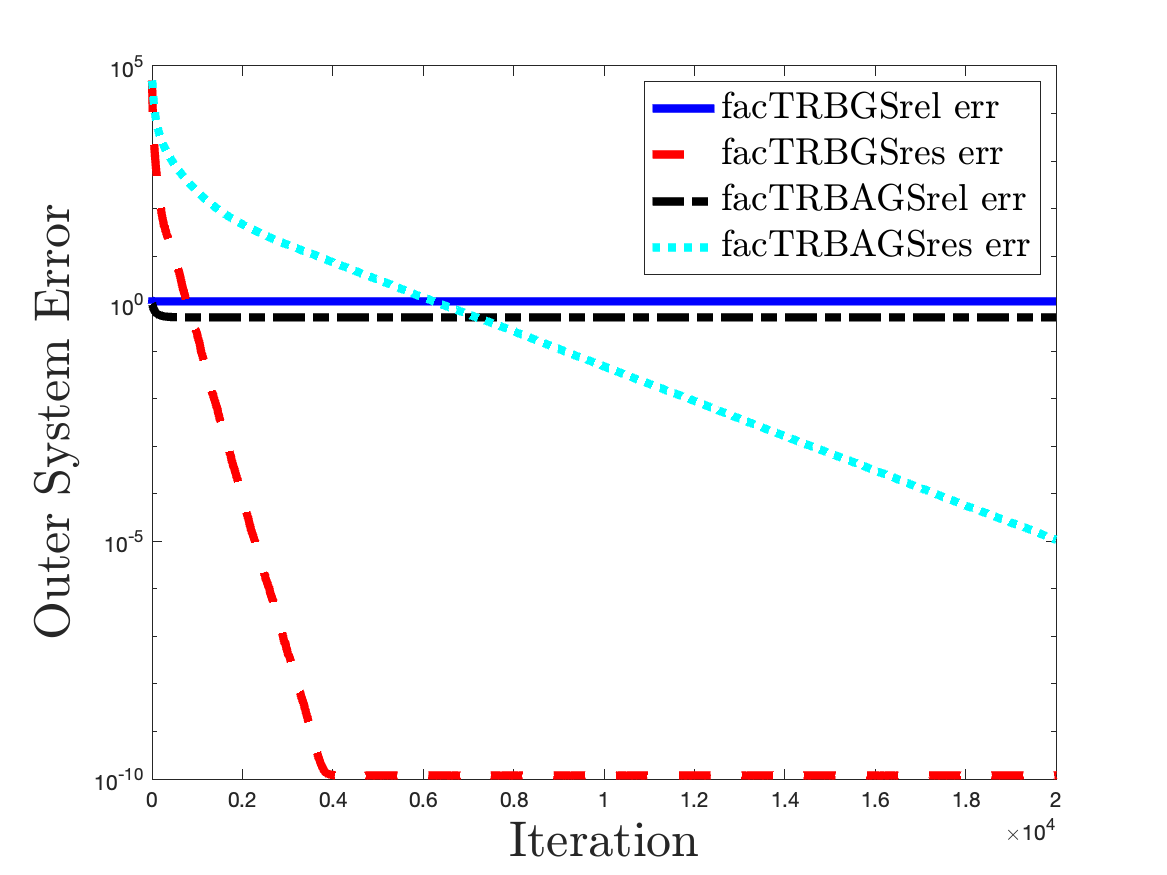}\hspace{.25cm}
    \includegraphics[width=0.44\textwidth]{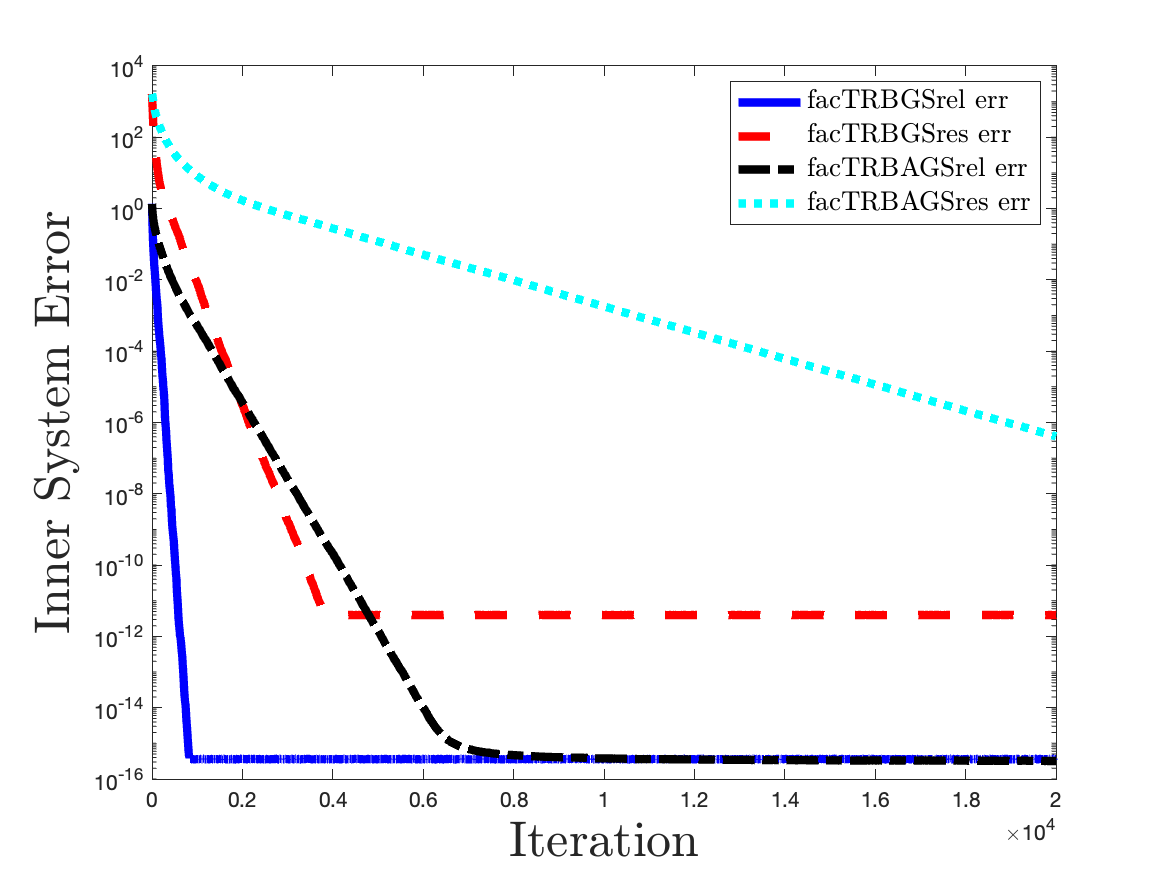}
    
    \caption{Relative error (rel err) and residual error (res err) versus iteration $t$ for the interlaced outer system $\tU \tZ = \tB$ and inner system $\tV \tX = \tZ$. Here, FacTRBGS and FacTRBAGS are applied to an inconsistent tensor system where $\tA$ is over-determined, $\tU$ is over-determined, and $\tV$ is under-determined. The block size is kept constant at $|\mu| = 5$.}
    \label{fig:factorized_Incons_U_over_V_under_A_over}
    \includegraphics[width=0.44\textwidth]{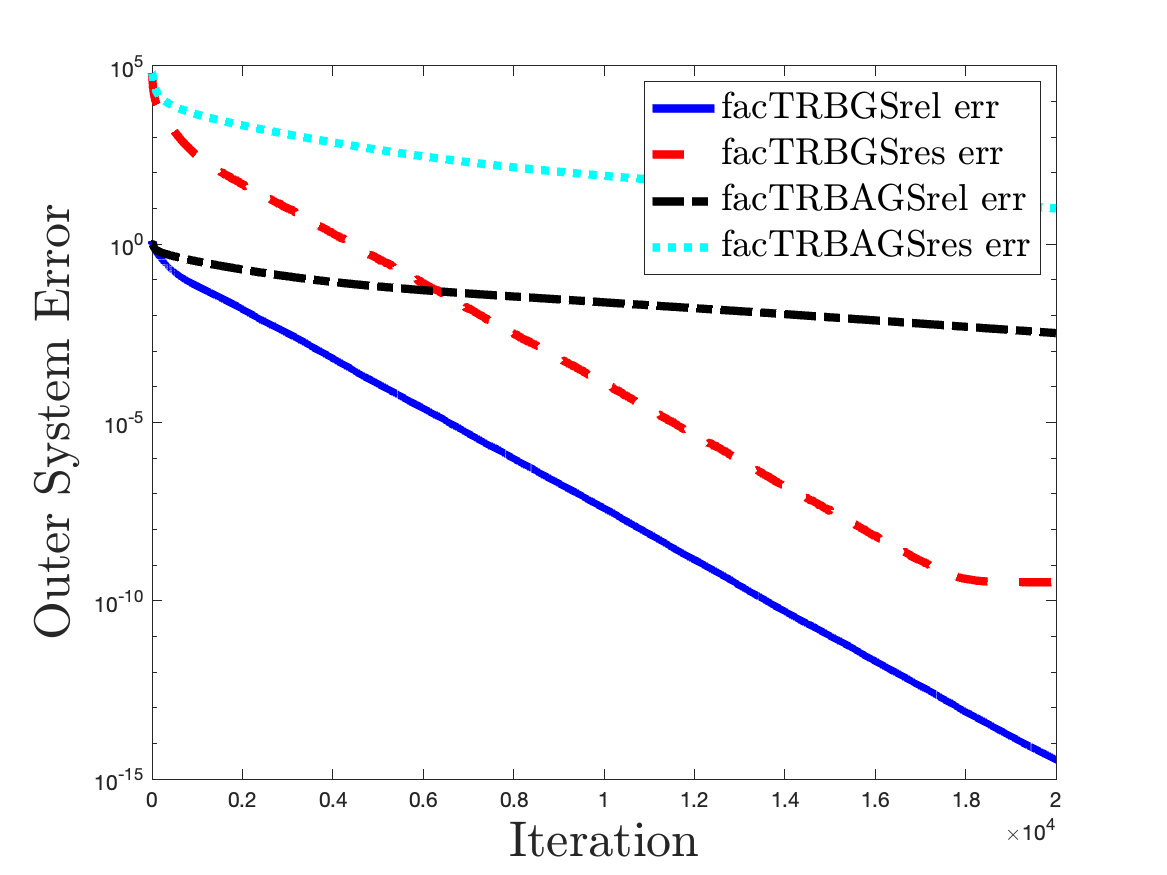}\hspace{.25cm}
    \includegraphics[width=0.44\textwidth]{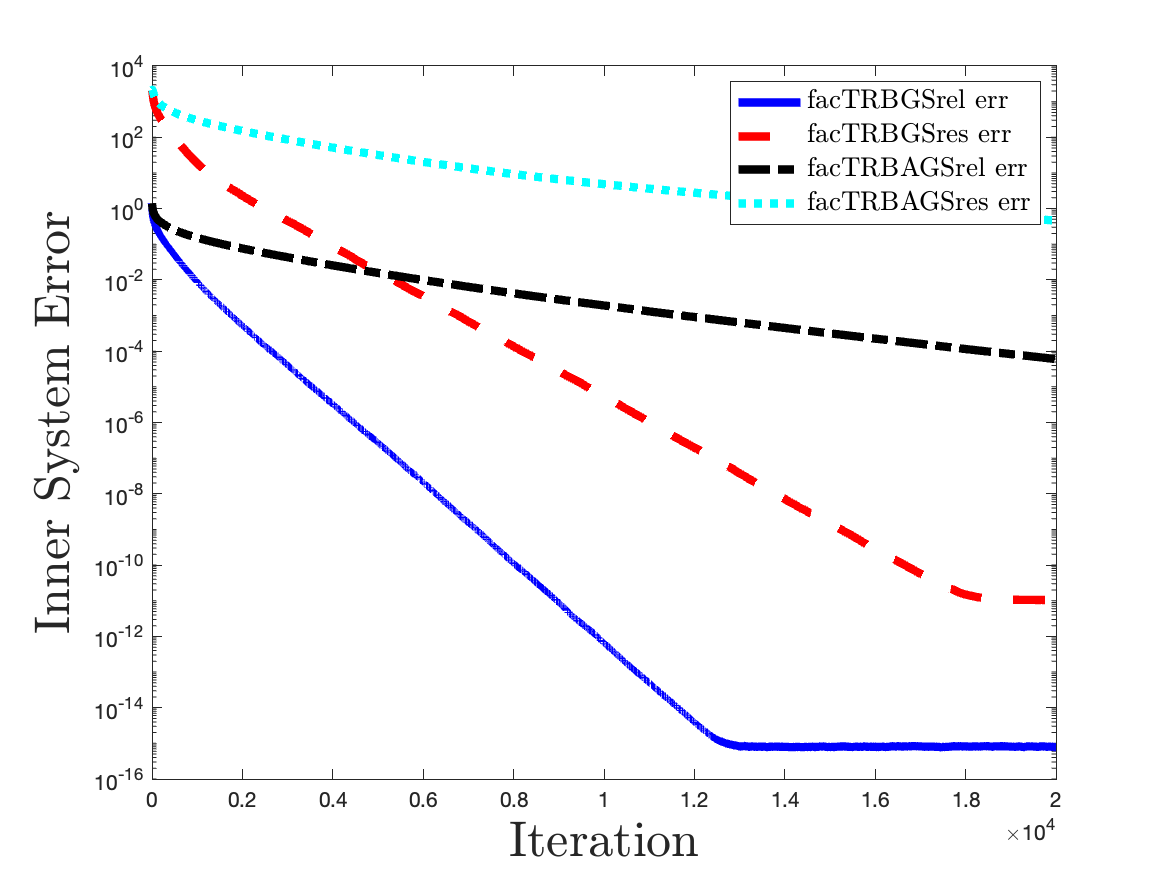}
    
    \caption{Relative error (rel err) and residual error (res err) versus iteration $t$ for the interlaced outer system $\tU \tZ = \tB$ and inner system $\tV \tX = \tZ$. Here, FacTRBGS and FacTRBAGS are applied to an inconsistent tensor system where $\tA$ is over-determined, $\tU$ is over-determined, and $\tV$ is over-determined. The block size is kept constant at $|\mu| = 5$.}
    \label{fig:factorized_Incons_U_over_V_over_A_over}
\end{figure}

Note that, for simplicity, we set the parameters $\omega_1 = \omega_2$ to $1$ throughout our calculations. For this value, FacTRBGS performs better than FacTRBAGS. Although this is not the focus of our work, we show that the convergence can be sped up by using another parameter value. (Figure~\ref{fig:factorized_Incons_U_over_V_over_A_over_step_size_5} in Appendix~\ref{app:Comparison FacTRBGS and FacTRBAGS (Faster FacTRBAGS convergence)} illustrates this faster convergence when we vary $\omega_1$ and $\omega_2$ for the case where $\tA$, $\tU$ and $\tV$ are all over-determined.) In general, one could find an optimal parameter interval using the condition on the parameters $\omega_1$ and $\omega_2$ stated in the theoretical results, but such optimal intervals depend on singular values whose computations can be costly.  

We also carried out experiments for the cases in the gray cells of Table~\ref{tab:factorizedcases} and the results can be found Appendix~\ref{app:Comparison FacTRBGS and FacTRBAGS (Cases in gray cells)}. From the lack of convergence of the algorithms for all inner systems, we can assume that divergence is almost certain for cases in which our Theorem offers no guarantees. However, note that we obtain the least-norm solution for the outer system when $\tA$, $\tU$, and $\tV$ are all under-determined. This comes from the construction of the inconsistent system, which may have made the residual 0 in this case.

\subsection{Video Deblurring}\label{subsec:deblurring_experiments}

We address the problem of recovering a true video tensor 
$\tX \in \R^{m \times n \times p}$ from a blurry video tensor $\tY \in \R^{m \times n \times p}$, given a known blurring operator. Specifically, we assume that each frame of the video undergoes blurring via a circular convolution kernel $\mat{H} \in \R^{m_1 \times n_1}$. Without loss of generality, the kernel can be extended to $\mat{H} \in \R^{m \times n}$ by padding it with zeros. Then, the relationship between the slices of the blurry and true video tensors can be expressed as the matrix system
\begin{equation*}\label{eqn:deblurringeqn}
    \underbrace{\begin{bmatrix} \mat{H}_1 & \mat{H}_m & \cdots &   \mat{H}_2 \\
 \mat{H}_2 & \mat{H}_{1} & \cdots  &\mat{H}_3\\ \mat{H}_3 &\mat{H}_2 & \cdots  &\vdots \\ \vdots & \vdots &\vdots &\vdots \\ \mat{H}_{m-1} & \mat{H}_{m-2}  &\cdots &\mat{H}_m\\
 \mat{H}_m & \mat{H}_{m-1} & \cdots  & \mat{H_1}\end{bmatrix}}_{\mat{H}}
 \underbrace{\begin{bmatrix} \bigg| & \cdots & \bigg| \\ \\ \tt{Vec}\left(\tX_{1::}\right) &\cdots & \tt{Vec}\left(\tX_{m::} \right) \\ \\\bigg| & \cdots & \bigg|\\ \end{bmatrix}}_{\mat{X}}
 = 
 \underbrace{\begin{bmatrix} \bigg| & \cdots & \bigg| \\ \\ \tt{Vec}\left(\tY_{1::}\right) &\cdots & \tt{Vec}\left(\tY_{m::} \right) \\ \\\bigg| & \cdots & \bigg|\\ \end{bmatrix}}_{\mat{Y}},
\end{equation*}
where $\mat{H_i} = \text{circ}(\vh_i) \in \R^{n \times n}$ is a circulant matrix and $\vh_i \in \R^n$ denotes $i^{th}$ row of $\mat{H}$, and the matrices $\mat{X}$ and $\mat{Y}$ are obtained by the illustrated matricization of $\tX$ and $\tY$, i.e., by defining the columns by the vectors obtained by vectorizing frontal slices of $\tX$ and $\tY$. Using the definition of the t-product~\ref{def:t-product}, we can rewrite this system as the t-linear system  
$$ \tH \ast \widetilde{\tX} = \ \widetilde{\tY},$$ 
where the $i$-th frontal face of the blurring tensor $\tH \in \R^{n \times n \times m}$ is given by $\mat{H}_i$ and $\widetilde{\tX}$ and $\widetilde{\tY}$ are obtained by suitably refolding $\mat{X}$ and $\mat{Y}$ respectively to dimensions $ n \times p \times m $. 

\begin{figure}[h!]
    \includegraphics[width=0.475\textwidth]{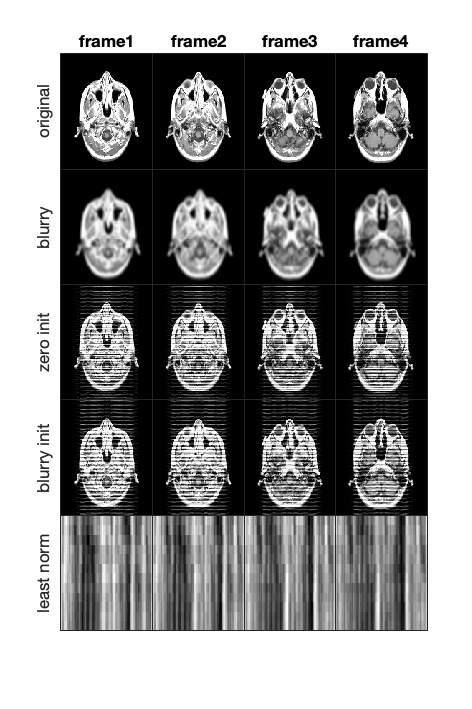}%
    \includegraphics[width=0.475\textwidth]{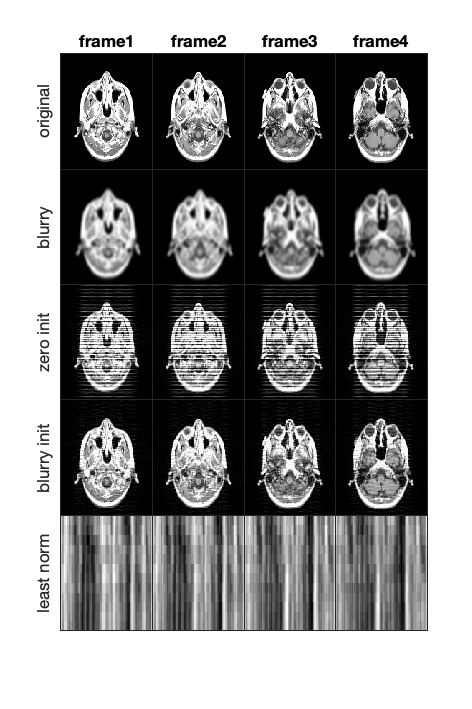}
    
    \vspace{-.5cm}
    (a) Deblurring using TRBGS \hspace{2.5cm} (b) Deblurring using TRBAGS 

    \includegraphics[width=0.65\textwidth]{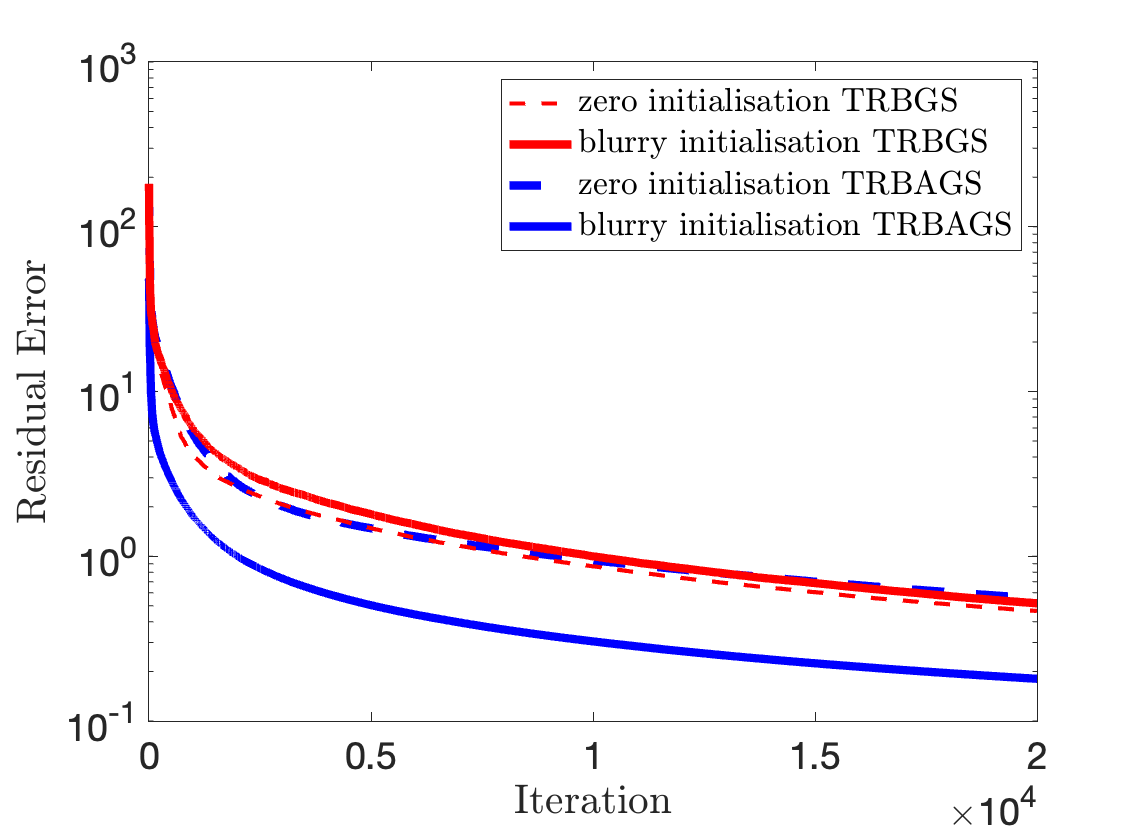}
    
    (c) Errors
    
    \caption{Deblurring of twice sequentially blurred video data using the TRBGS and TRBAGS. In (a) and (b), we have frames from the original video (top row), twice blurred frames (second row), frames recovered using TRBGS and TRBAGS (third and fourth row) respectively. The least norm solution is show on the bottom row. In (c), we  show the residual error of TRBGS and TRBAGS iterates for two different initializations.}\label{fig:TGSdeblurring}
\end{figure}

In our experiment, we consider the recovery of 12 frames of twice blurred MRI images of size $128 \times 128$. First, we circularly blur the images using a $5\times5$ Gaussian filter followed by a $5\times5$ averaging filter. These operations amount to solving a t-product system, $\tU\tV\tX = \tY$, where $\tX\in\R^{128\times 12 \times 128}$ represents the original images, $\tY\in\R^{128\times12\times128}$ the blurred images, and $\tU$ and $\tV \in \R^{128 \times 128 \times 128}$ represents the tensor obtained by the successive Gaussian blurring and average blurring operators.

We can visualize the results of our experiment in Figure~\ref{fig:TGSdeblurring}, which shows that we can recover the frames with a residual error of about $10^{-1}$ for TRBAGS, while the error for TRBGS is larger. Interestingly, TRBAGS performs better than TRBGS, with fewer artifacts in the recovered images. We believe that the pseudoinverse computation may be the source of the larger numerical errors in TRGBS. Although the error is non-negligible, the recovered images qualitatively match the original video data much more closely than the blurred versions, especially for TRBAGS, demonstrating the experiment's effectiveness for practical purposes. We omit factorized versions since the results in Section~\ref{subsec:FacTRBGSvsFacTRBAGS} show that they have performances comparable to TRBGS and TRBAGS, and thus, we would expect the same qualitative results.

\section{Conclusion}

We have developed two pairs of algorithms to solve tensor linear regression problems. Our first two methods, TRBGS and TRBAGS, find the least-norm solution for the problem~\eqref{eq:regression}, with the particularity that the formulation of TRBGS uses a pseudoinverse while TRBAGS is pseudoinverse free. The second pair of methods, FacTRBGS (with pseudoinverse) and FacTRBAGS (pseudoinverse free), addresses the case in which the measurement tensor is given as a t-product of two tensors~\eqref{eq:fact_regression}. We provide theoretical guarantees that all these methods yield the least-norm solution whether the system is consistent or inconsistent. We also show their effectiveness on synthetic and real data through numerical experiments with various considerations such as over-determined/under-determined systems and sampling block sizes. For the real-data case in particular, we apply TRBGS and TRBAGS to an image deblurring problem and recover a qualitatively good match for the original images. We also illustrate that divergence occurs in cases outside the assumptions of our theorem, demonstrating the strength of our theoretical results. 

Our methods add to the broader framework of extending matrix regression algorithms to the tensor setting, and we see future directions for improvements. In our numerical experiments, we set all step sizes to 1 in our inverse-free methods, following standard practice. This choice worked well for our purposes. Although hyperparameter selection can generally be computationally costly, future work could explore the effect of different step-size parameters on convergence by considering multiple linear systems. An example of faster convergence illustrated in the appendix justifies this future direction. Additionally, some randomized iterative matrix algorithms (e.g.,~\cite{du2021randomized}) use weighted averages of block rows/columns rather than a uniform weight of 1, as we have done in our work. Investigating the impact of weight selection on the performance of our algorithms is another important direction for future research.

\section*{Acknowledgments}

The initial research for this effort was conducted at the Research Collaboration Workshop for Women in Data Science and Mathematics (WiSDM), August 2023 held at the Institute for Pure and Applied Mathematics (IPAM). Funding for the workshop was provided by IPAM, AWM and DIMACS (NSF grant CCF1144502).

This material is based on work supported by the National Science Foundation under Grant No. DMS-1928930, while several of the authors were in residence at the Mathematical Sciences Research Institute in Berkeley, California, during the summer of 2024.  

Several of the authors also appreciate the support provided to them at a SQuaRE at the American Institute of Mathematics. The authors thank AIM for providing a supportive and mathematically rich environment.

This material is based upon work supported by the National Science Foundation under Grant No. DMS-1929284 while several of the authors were in residence at the Institute for Computational and Experimental Research in Mathematics in Providence, RI, during the ``Randomized Algorithms for Tensor Problems with Factorized Operations or Data" Collaborate@ICERM.

JH was partially supported by NSF DMS \#2211318 and NSF CAREER \#2440040. DN was partially supported by NSF DMS \#2011140. 

\bibliographystyle{plain}
\bibliography{main2}

\appendix

\section{Effect of Block Sizes (Continued)}\label{app:Effect of Block Sizes}

Figures~\ref{fig:appendTRBAGS_smallitsnum} and~\ref{fig:appendTRBAGS_smallitsnumlarge} are an illustration of the assertion in Section~\ref{subsubsec:TRBGSandTRBAGSSuites} that, in the inverse-free (TRBAGS) case, the block size affects the speed of TRBAGS more significantly over a small number of iterations, and/or when the tensor dimensions are larger.

\begin{figure}[h]
    \includegraphics[width=0.475\textwidth]{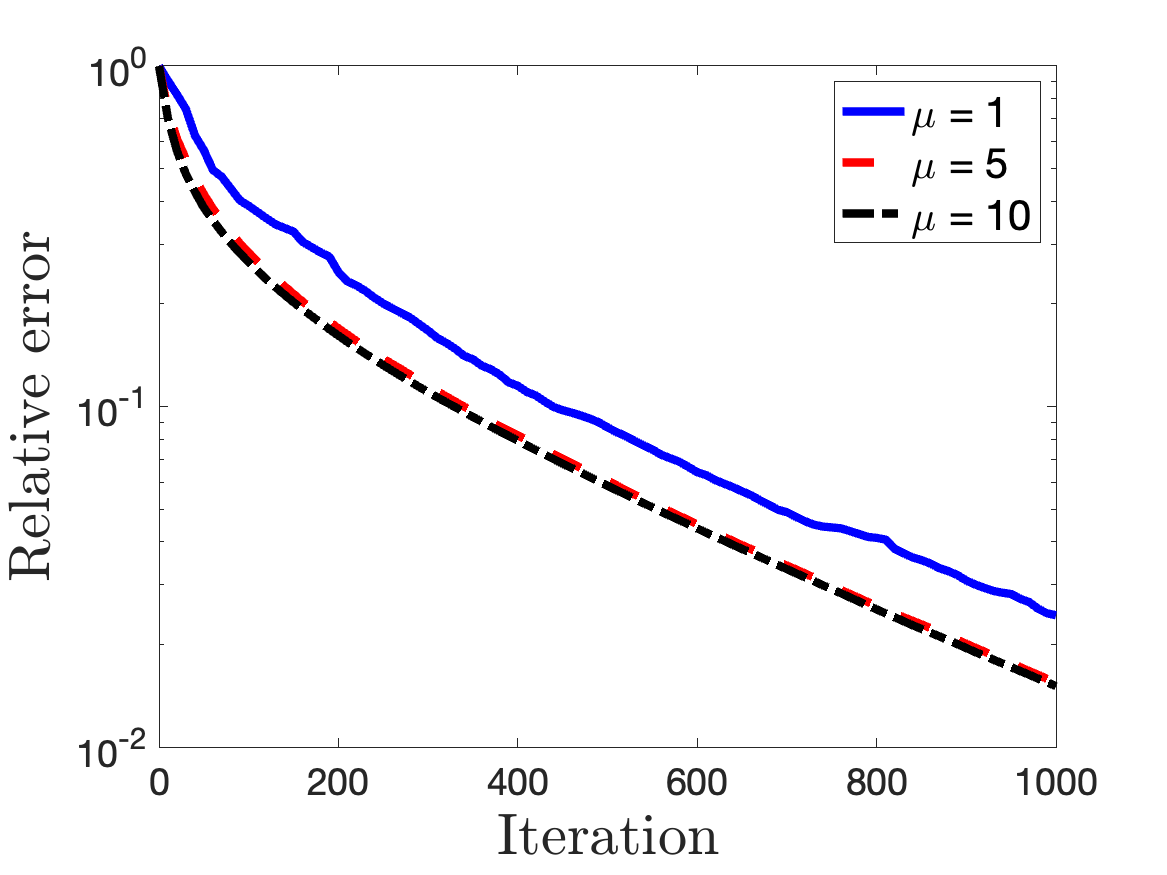}
    \hspace{.25cm}%
    \includegraphics[width=0.475\textwidth]{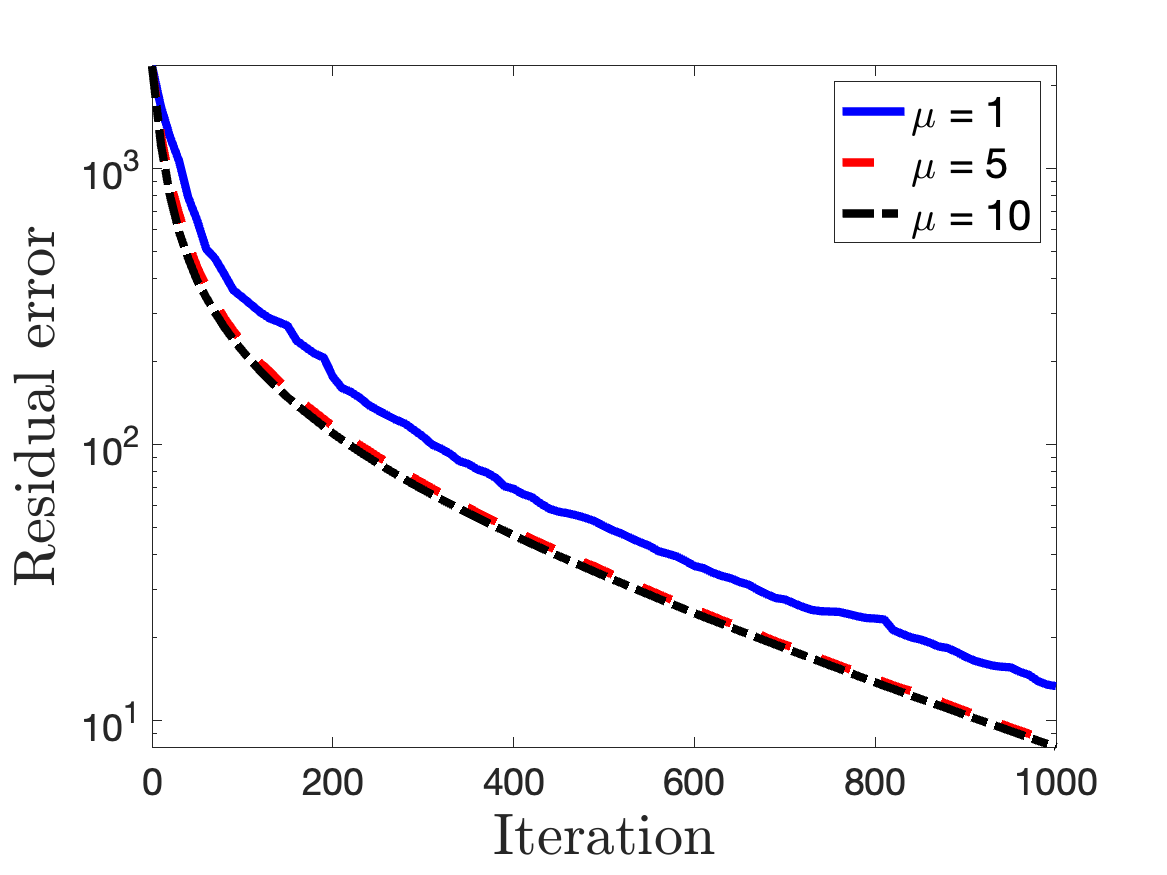}
    
    \caption{Relative error $\|\tX^{(t)} - \tX^\ddagger\|_F/\|\tX^\ddagger\|_F$ and residual error $\|\tA \tX^{(t)} - \tA\tX^\ddagger\|_F$ versus iteration $t$ of TRBAGS on a consistent linear system when $\tA$ is over-determined and of size $\R^{30 \times 20 \times 30}$. We consider sampling block sizes $|\mu| \in \{1, 5, 10\}$ in each case.}
    \label{fig:appendTRBAGS_smallitsnum}
\end{figure}

\begin{figure}[h]
    \includegraphics[width=0.475\textwidth]{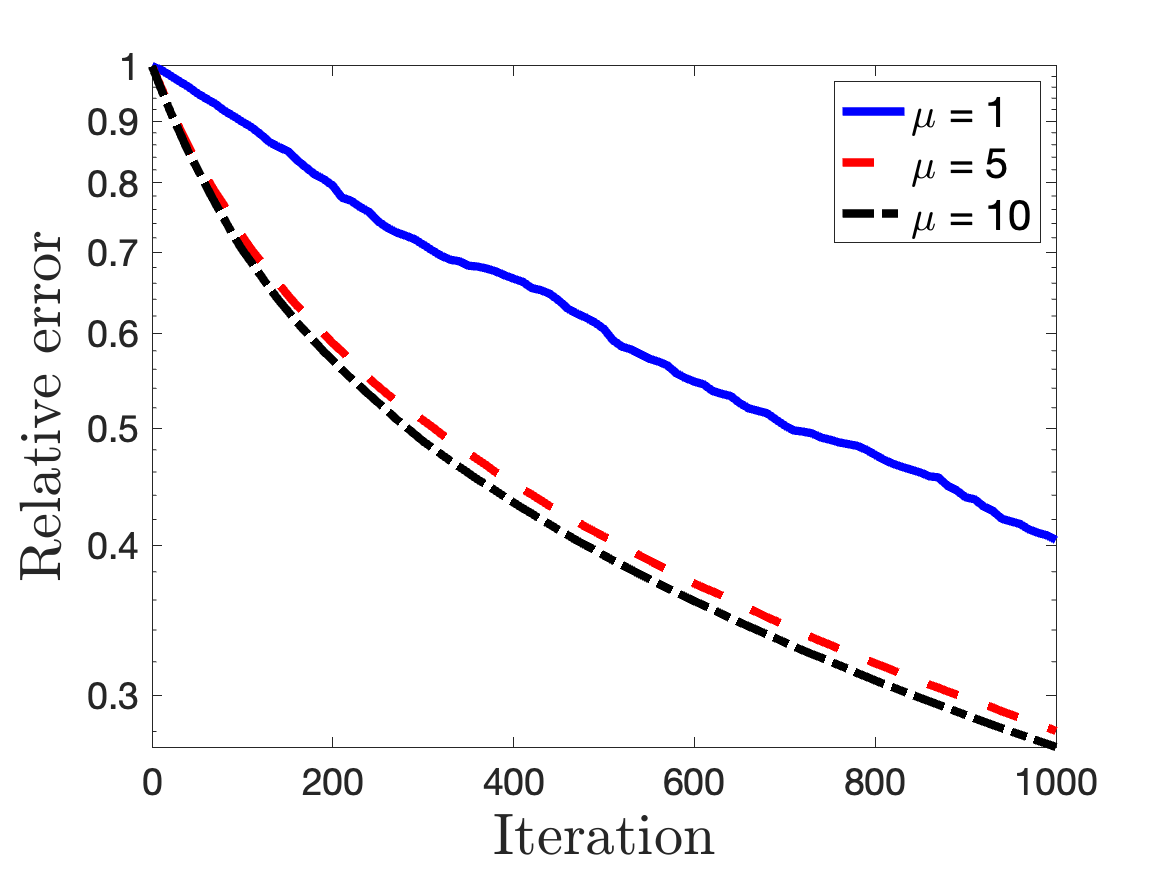}
    \hspace{.25cm} %
    \includegraphics[width=0.475\textwidth]{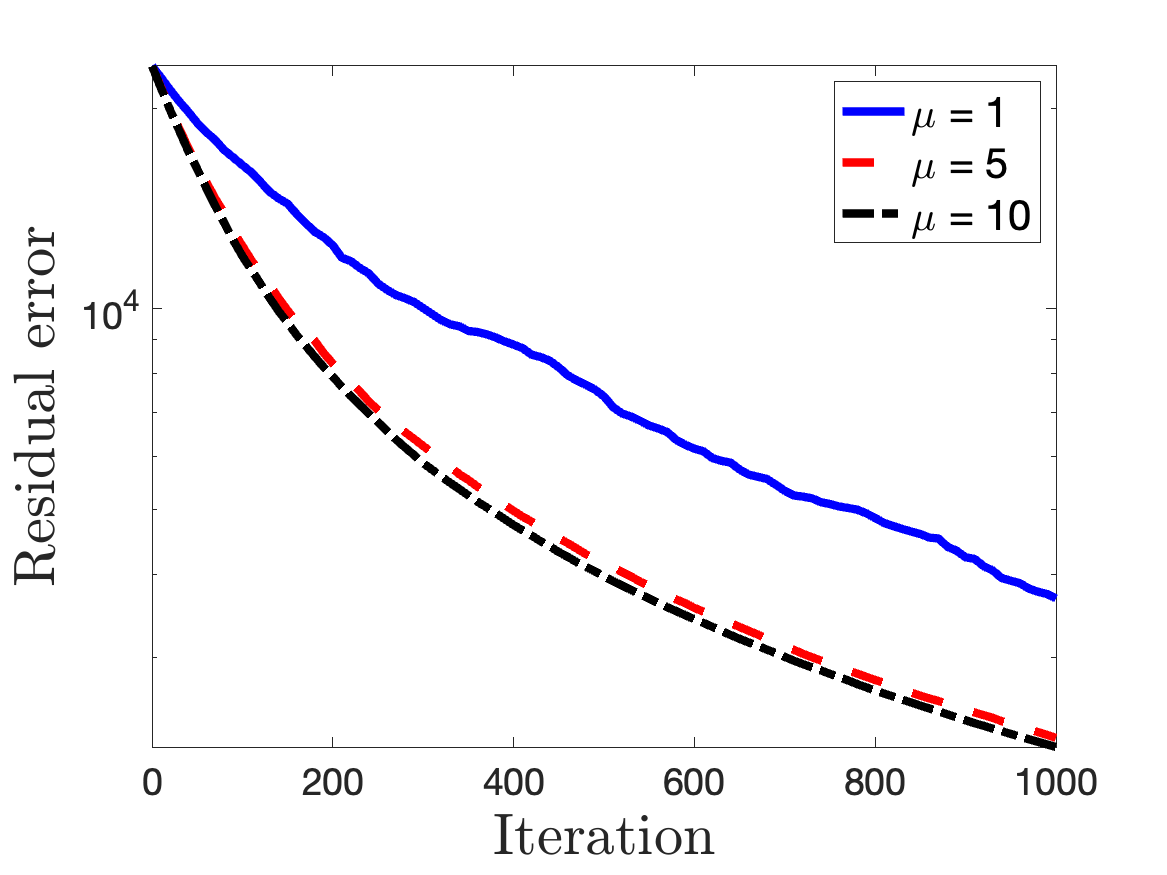}
    
    \caption{Relative error $\|\tX^{(t)} - \tX^\ddagger\|_F/\|\tX^\ddagger\|_F$ and residual error $\|\tA \tX^{(t)} - \tA\tX^\ddagger\|_F$ versus iteration $t$ of TRBAGS on a consistent linear system when $\tA$ is over-determined and of size $\R^{300 \times 200 \times 30}$. We consider sampling block sizes $|\mu| \in \{1, 5, 10\}$ in each case.}
    \label{fig:appendTRBAGS_smallitsnumlarge}
\end{figure}

\newpage
\section{Comparison FacTRBGS and FacTRBAGS (Faster FacTRBAGS convergence)}\label{app:Comparison FacTRBGS and FacTRBAGS (Faster FacTRBAGS convergence)}

Figure~\ref{fig:factorized_Incons_U_over_V_over_A_over_step_size_5} illustrates the fact that we can achieve faster convergence by changing the step-size values.

\begin{figure}[h]
    \includegraphics[width=0.475\textwidth]{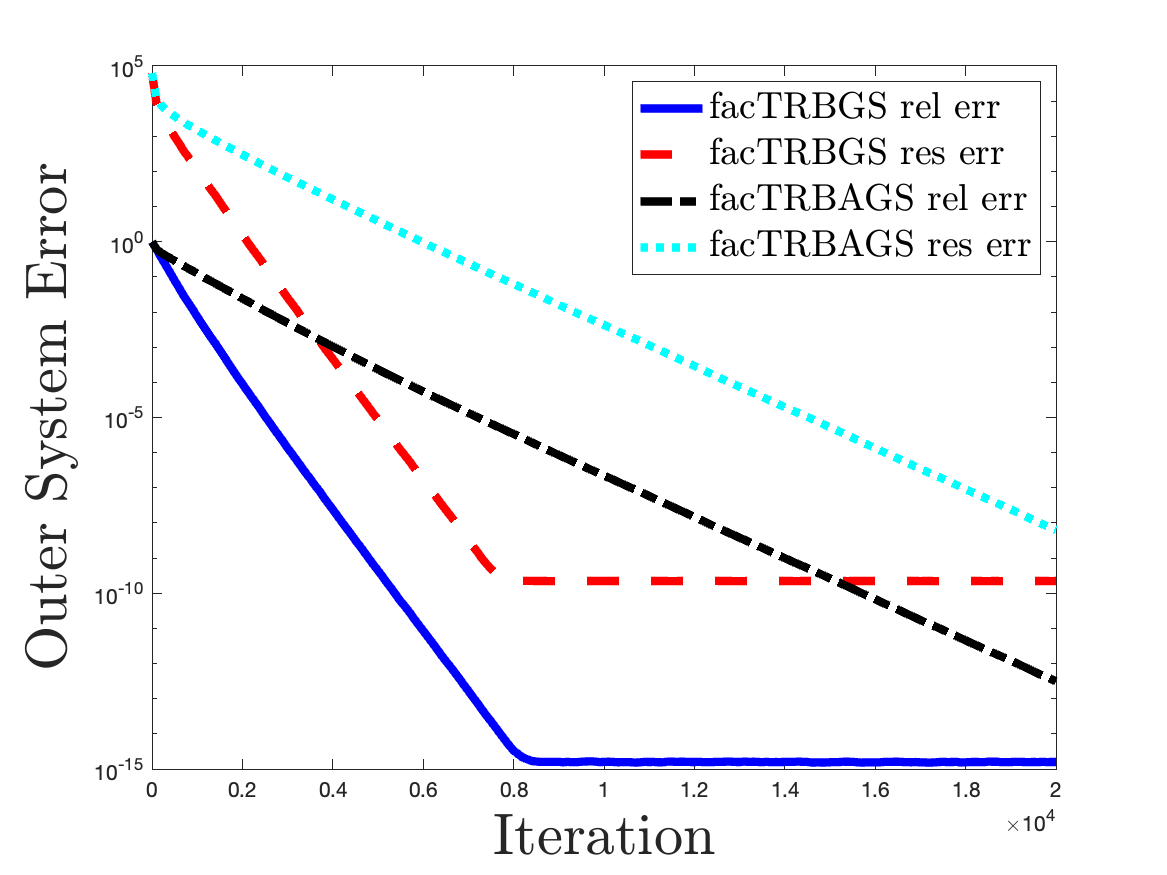}\hspace{.25cm}
    \includegraphics[width=0.475\textwidth]{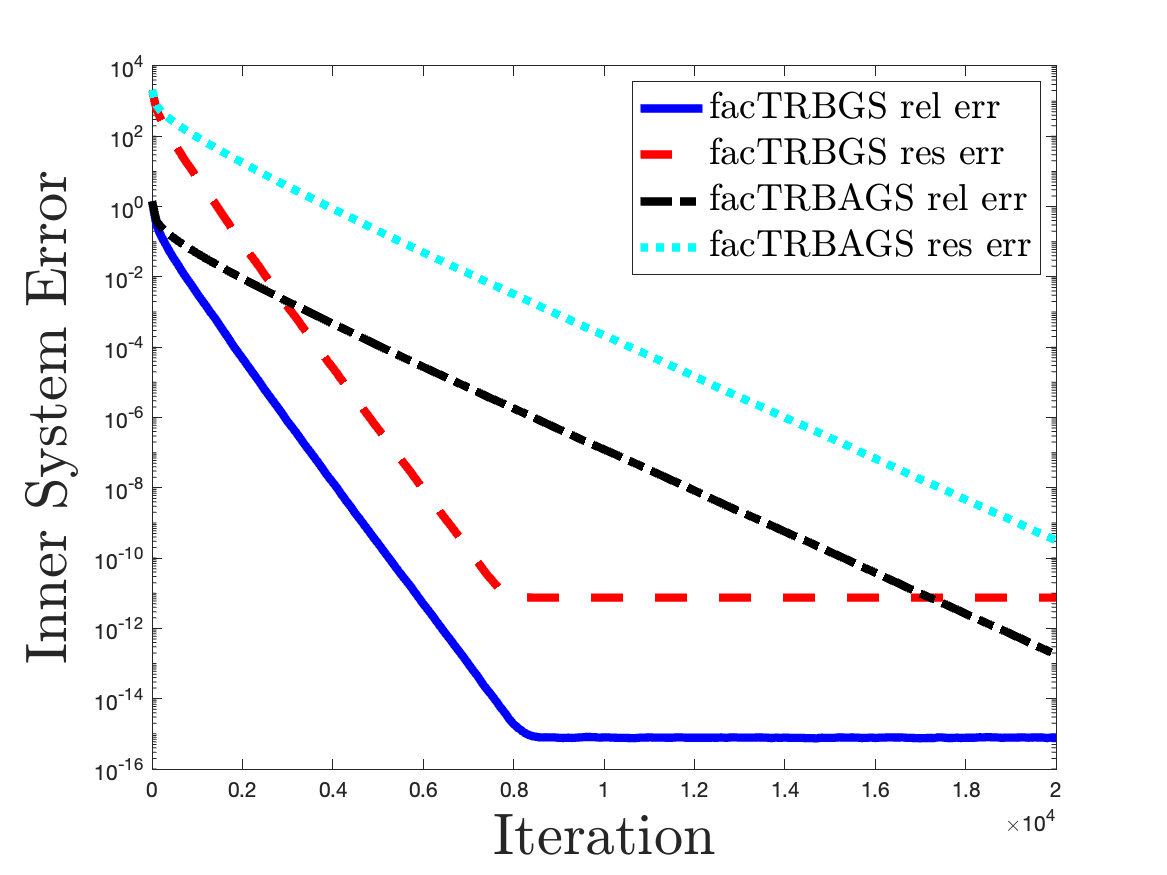}
    
    \caption{Relative error (rel err) and residual error (res err) versus iteration $t$ for the interlaced outer system $\tU \tZ = \tB$ and inner system $\tV \tX = \tZ$. Here, FacTRBGS and FacTRBAGS are applied to an inconsistent tensor system where $\tA$ is over-determined, $\tU$ is over-determined, and $\tV$ is over-determined. The block size is kept constant at $|\mu| = 5$.}
    \label{fig:factorized_Incons_U_over_V_over_A_over_step_size_5}
\end{figure}

\section{Comparison FacTRBGS and FacTRBAGS (Cases in gray cells)}\label{app:Comparison FacTRBGS and FacTRBAGS (Cases in gray cells)}

Figures~\ref{fig:factorized_Incons_U_under_V_over_A_under} through~\ref{fig:factorized_Incons_U_under_V_under_A_under} illustrate the fact that in cases for which theoretical guarantees do not hold, we do not achieve convergence for the inner systems.

\begin{figure}[h]
    \includegraphics[width=0.475\textwidth]{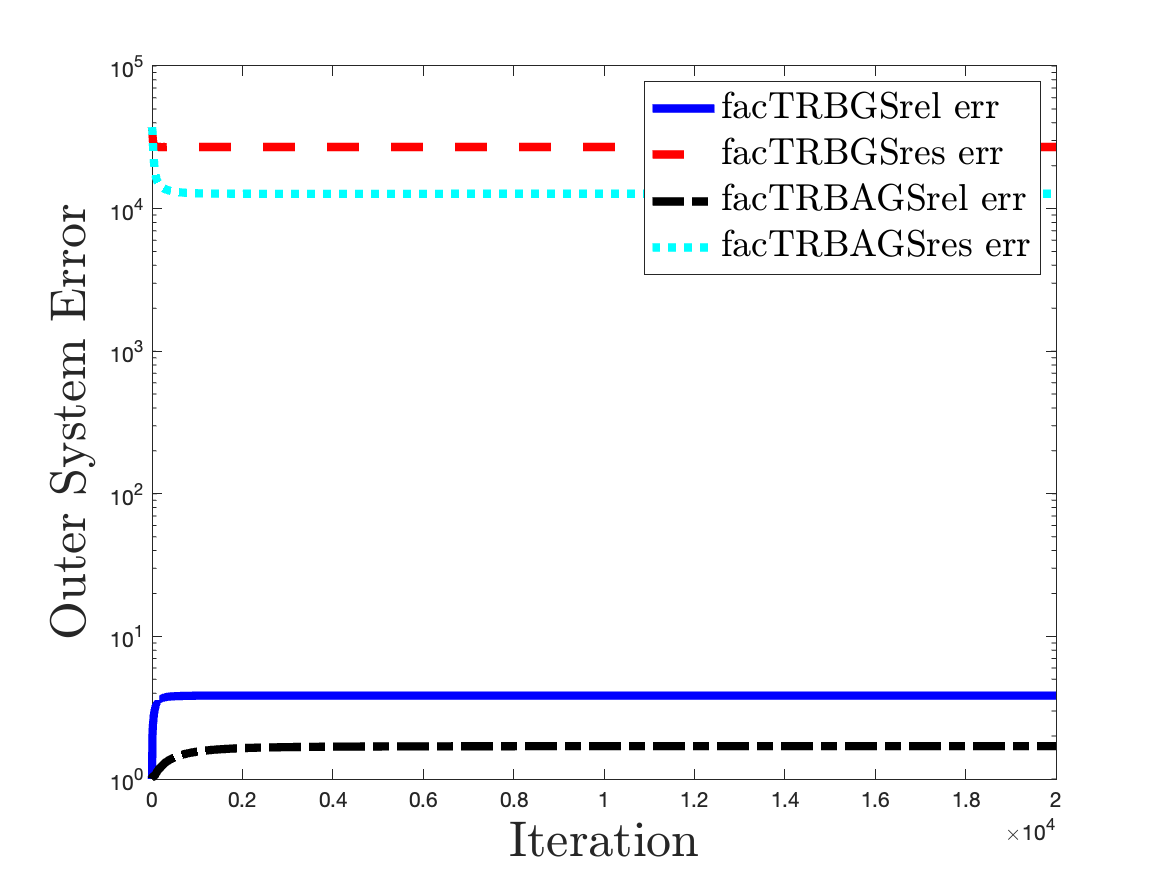}\hspace{.25cm}
    \includegraphics[width=0.475\textwidth]{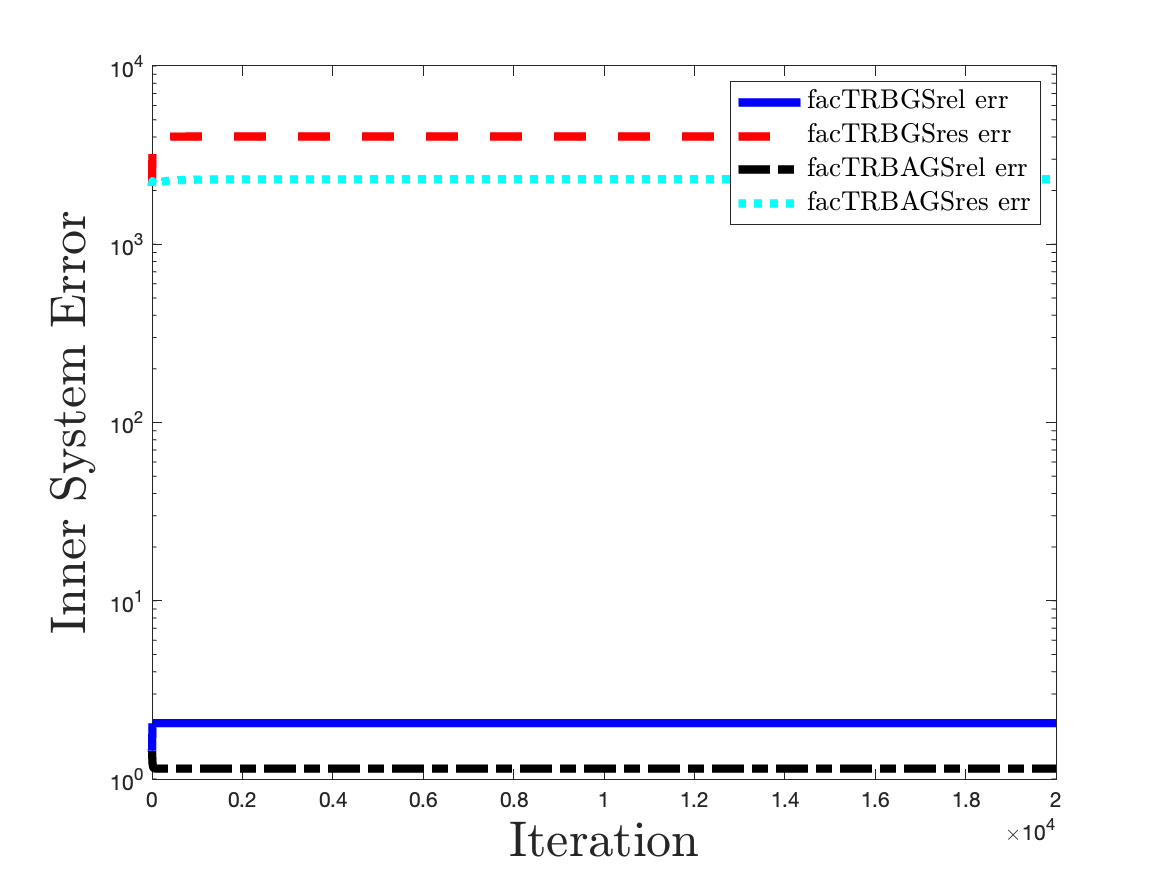}
    
    \caption{Relative error (rel err) and residual error (res err) versus iteration $t$ for the interlaced outer system $\tU \tZ = \tB$ and inner system $\tV \tX = \tZ$. Here, FacTRBGS and FacTRBAGS are applied to an inconsistent tensor system where $\tA$ is under-determined, $\tU$ is under-determined, and $\tV$ is over-determined. The block size is kept constant at $|\mu| = 5$.}
    \label{fig:factorized_Incons_U_under_V_over_A_under}
    \includegraphics[width=0.475\textwidth]{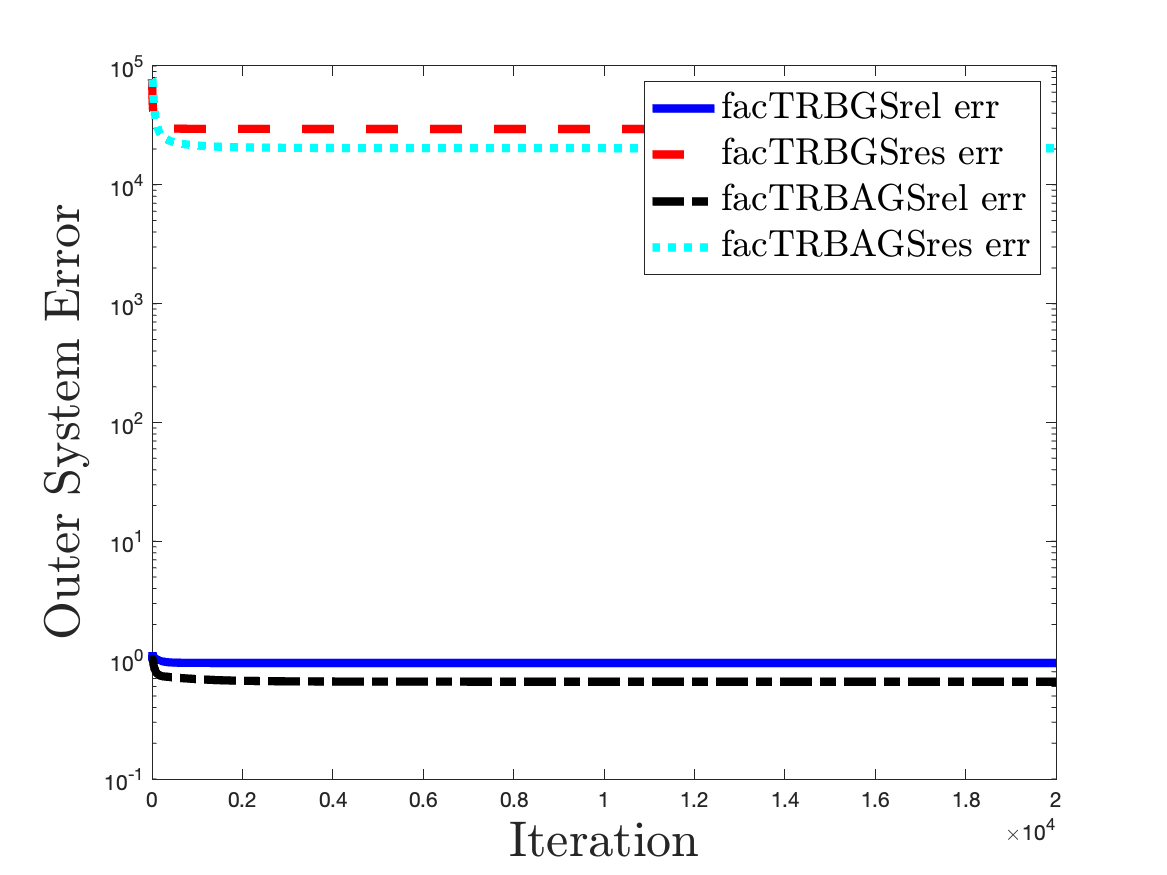}\hspace{.25cm}
    \includegraphics[width=0.475\textwidth]{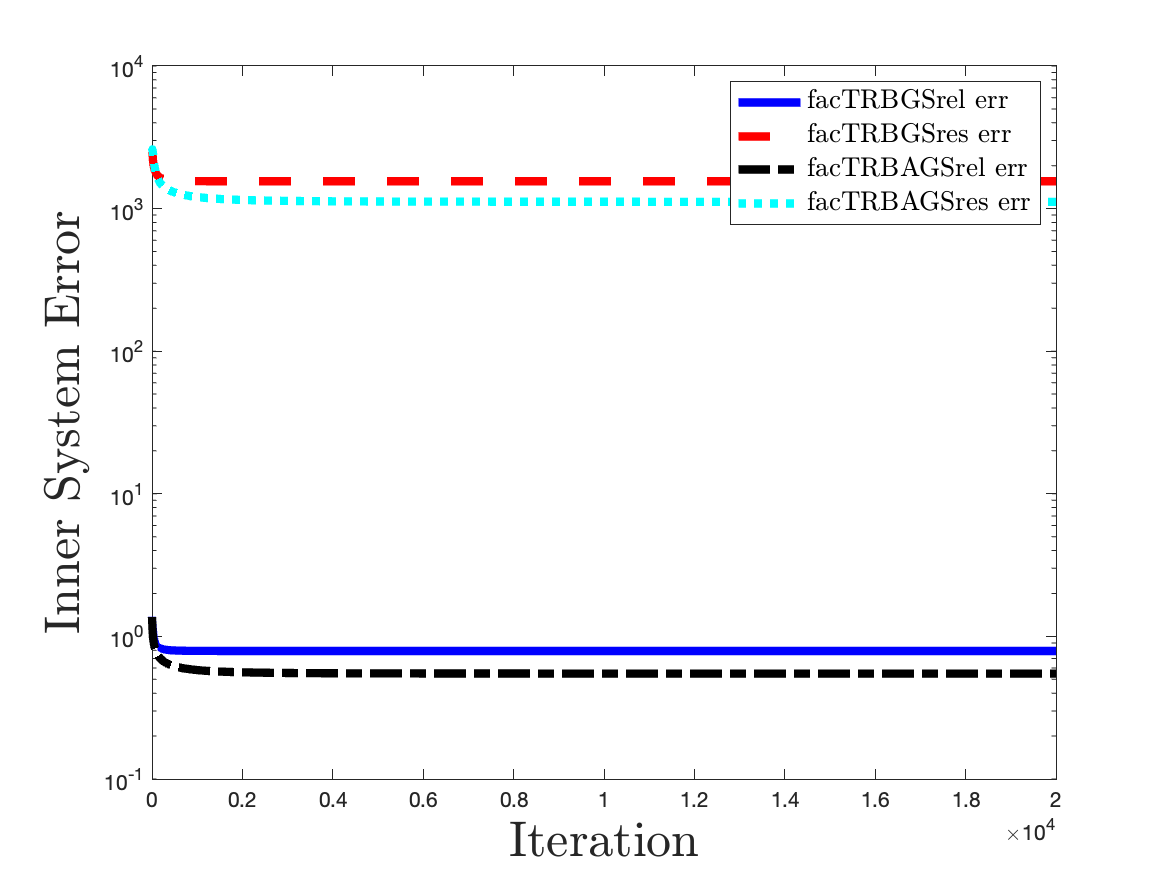}
    
    \caption{Relative error (rel err) and residual error (res err) versus iteration $t$ for the interlaced outer system $\tU \tZ = \tB$ and inner system $\tV \tX = \tZ$. Here, FacTRBGS and FacTRBAGS are applied to an inconsistent tensor system where $\tA$ is over-determined, $\tU$ is under-determined and $\tV$ is over-determined. The block size is kept constant at $|\mu| = 5$.}
    \label{fig:factorized_Incons_U_under_V_over_A_over}
    \includegraphics[width=0.475\textwidth]{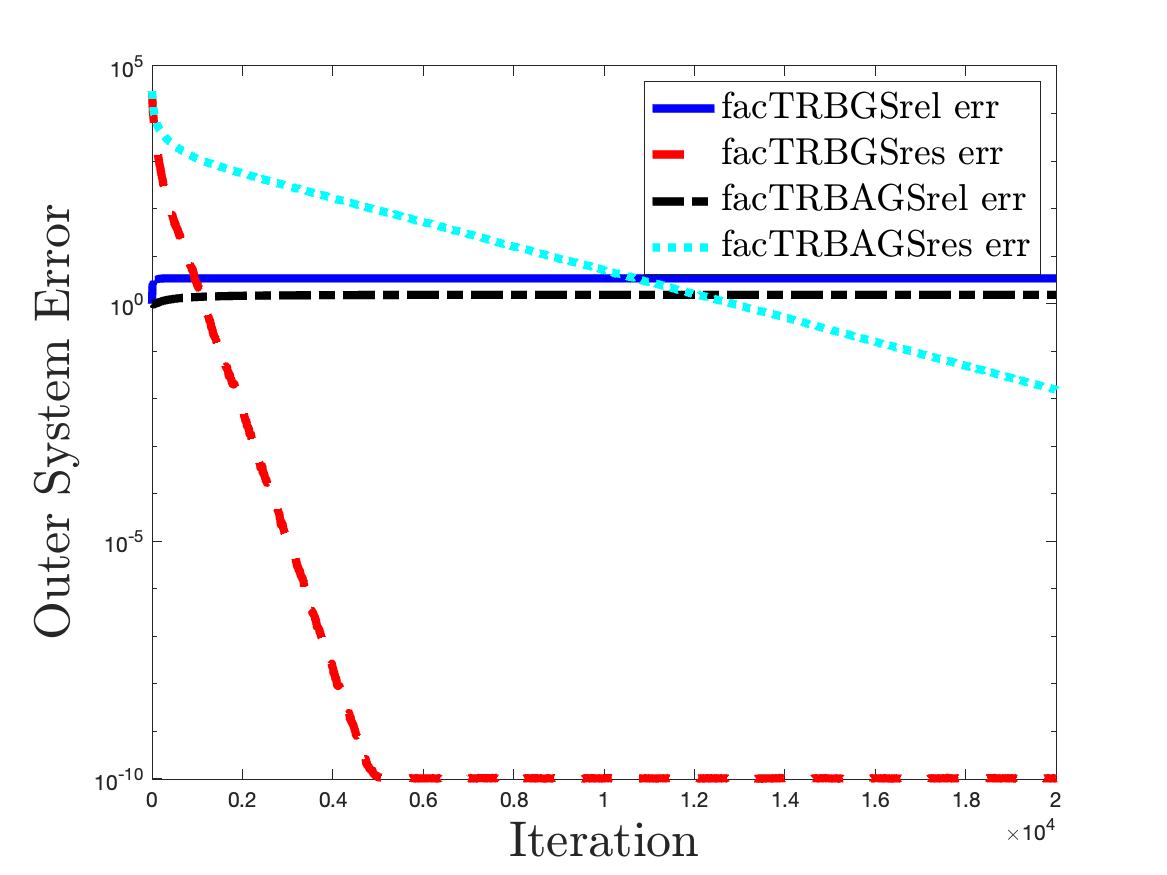}\hspace{.25cm}
    \includegraphics[width=0.475\textwidth]{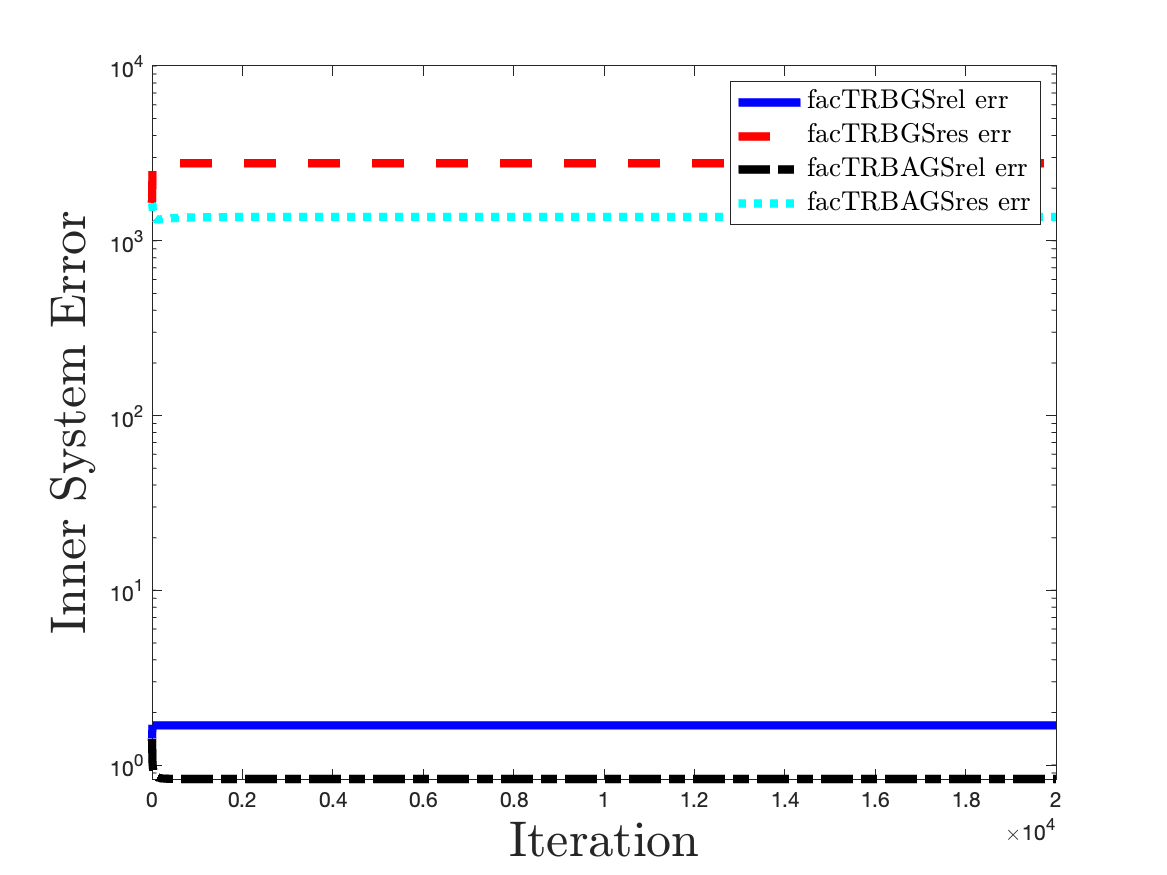}
    
    \caption{Relative error (rel err) and residual error (res err) versus iteration $t$ for the interlaced outer system $\tU \tZ = \tB$ and inner system $\tV \tX = \tZ$. Here, FacTRBGS and FacTRBAGS are applied to an inconsistent tensor system where $\tA$ is under-determined, $\tU$ is under-determined and $\tV$ is under-determined. The block size is kept constant at $|\mu| = 5$.}
    \label{fig:factorized_Incons_U_under_V_under_A_under}
\end{figure}\vspace{1cm}

\end{document}